\documentclass[a4paper,10pt]{amsart}
\usepackage[a4paper,top=2.5cm,bottom=2.5cm,left=2.5cm,right=2.5cm]{geometry}
\usepackage{macros}

\begin{document}
\title[Comparing metric completions with Verdier quotients]{Comparing metric completions of triangulated categories with Verdier quotients}

\author{\textsc{Cyril Matou\v sek}}
\address{Paderborn University,
Institute of Mathematics,
Warburger Strasse 100,
33098 Paderborn,
Germany}
\email{cyril.matousek@uni-paderborn.de}

\date{\today}

\begin{abstract}
Neeman's metric on a triangulated category assigns lengths to morphisms and allows us to construct triangulated metric completions of the category.
We prove that the metric completion of a~triangulated category with respect to a constant metric is always a subcategory of the Verdier quotient by the thick subcategory generated by cones of zero-length morphisms, and we provide a sufficient criterion for equality.
To that end, we define continuous functors between completions which induce triangulated functors between completions.
Moreover, we give two applications. First, we show that the additive equivalence discovered by Cummings and Gratz between the Paquette-Yıldırım completion and a specific metric completion of an Igusa-Todorov discrete cluster category is, in fact, a~triangulated equivalence. Second, we fully classify metric completions of bounded derived categories of coherent sheaves over weighted projective lines of domestic or tubular type.
\end{abstract}

\maketitle

{
\hypersetup{linkcolor=black}
\tableofcontents
}

\section{Introduction}
\label{sec:introduction}

Taking metric completions of triangulated categories is a unique method to construct new triangulated categories. Using the idea behind convergence of sequences of elements of a metric space, one can construct a metric completion of a triangulated category by adding formal colimits of Cauchy sequences of objects and morphisms to the category in such a way that the resulting category is again triangulated.
Despite being discovered relatively recently,
the~method of metric completions has already spawned an array of applications in in homological algebra, representation theory, and algebraic geometry \cite{BiswasChenRahulParkerZheng24,CanonacoHaesemeyerNeemanStellari24,CummingsGratz24,DeynLankRahul25,Neeman21C,Neeman24,Neeman26,Rahul25A,Rahul25B,SunZhang21}.

Metrics and metric completions of triangulated categories have been initially introduced by Neeman in \cite{Neeman18} with the purpose to provide an alternative proof to Rickard's results on derived Morita theory \cite{Rickard89} that for coherent rings \(\perfect(R)\simeq\perfect(S)\) if and only if \(\derived^b(\modf\dashmodule R)\simeq\derived^b(\modf\dashmodule S)\).   

The~same metric-based strategy works even for a larger class of triangulated categories such as noetherian, weakly approximable categories (see \cite{CanonacoNeemanStellari25,Neeman18,Neeman23,Neeman25B}) and thus offers a~generalisation of derived Morita theory to, for example, a~derived category of a separated noetherian scheme.

What makes the technique of metric completions such a groundbreaking result, however, are not necessarily only its applications to derived Morita theory, but also the fact that it allows us to construct a new triangulated category from the data of an old one. 
Besides the standard constructions of a triangulated subcategory and Verdier quotient, there are not that many methods to reproduce a triangulated category. For example, the constructions of a relative stable category \cite{BalmerStevenson21,Beligiannis00} or an orbit category \cite{Keller05}  deliver a~new triangulated category. Some of the rare such methods rely on rather strong additional assumptions such as the existence of a separable ring object and a tensor-triangulated structure in \cite{Balmer11}, or a suitable enhancement.

On the other hand, the condition for a family of subcategories of a triangulated category to be a~good metric (Definition~\ref{D:Metric}) is relatively easy to satisfy. For example, any triangulated subcategory automatically determines the corresponding constant metric, and both an aisle and a coaisle of any $t$-structure provide us with their own metrics (Subsection~\ref{SS:MetricsAndMetricCompletions}).
One is then, in principle, capable of constructing many different completions of a given triangulated category with respect to different choices of a metric.

There are, nevertheless, certain unnaturalities of the completion process. The metric completion \(\mathfrak{S}({\mathcal{S}})\) (see Definition~\ref{D:Completion}) of a triangulated category \(\mathcal{S}\) is formally an intersection
of two subcategories  \(\mathfrak{L}({\mathcal{S}}),\mathfrak{C}({\mathcal{S}})\) of the additive functor category \(\Mod\dashmodule\mathcal{S}\), where \(\mathfrak{L}({\mathcal{S}})\) consists of colimits of so-called \textit{ Cauchy sequences} with respect to the metric, while the functors from \(\mathfrak{C}({\mathcal{S}})\) send morphisms perceived by the~metric as very short to isomorphisms. This two-step process can make certain desirable constructions rather completed, as an~operation behaving nicely with respect to one of the defining categories does not have to preserve the~other. For example, unless a metric is \textit{embeddable} (in the sense of \cite{BiswasChenRahulParkerZheng24}), there is no obvious way to obtain a functor \(\mathcal{S}\rightarrow\mathfrak{S}({\mathcal{S}})\).

Similar obstacle was addressed by Sun and Zhang in \cite{SunZhang21} in their study of functionality of the completion construction. They define a notion of a~\textit{compression} functor (Definition~\ref{D:Compression}) between triangulated categories with good metrics and prove that an adjoint pair \(\mathcal{S}\rightleftarrows\mathcal{T}\) of compressions determines a triangulated functor \(\mathfrak{S}({\mathcal{T}})\rightarrow\mathfrak{S}({\mathcal{S}})\) in one direction. The requirement of the existence of an adjoint pair is there precisely to ensure compatibility with the \(\mathfrak{L}\) and \(\mathfrak{C}\) constructions simultaneously.  

In our first main result, we propose a definition of a \textit{continuous} functor (Definition~\ref{D:Continuous}) between triangulated categories with metrics, which is more restrictive than the one of a compression but no longer requires an adjoint pair to induce functors on completions.
\begin{theorem}[Theorem~\ref{T:CompletionOfContinuousFunctor}]
\label{T:IntroductionContinuous}
Let \(\mathcal{S}\) and \(\mathcal{T}\) be triangulated categories equipped with good metrics and \(\varphi:\mathcal{S}\rightarrow\mathcal{T}\) a continuous functor.
Then there is a natural triangulated functor \(\mathfrak{S}(\varphi):\mathfrak{S}(\mathcal{S})\rightarrow\mathfrak{S}(\mathcal{T})\) compatible with the action of \(\varphi\) on Cauchy sequences.
\end{theorem}

Another potential issue with metric completions it that they are notoriously hard to calculate, making the list of examples in the existing literature rather limited.
With the sole exception of \cite{CummingsGratz24}, every instance of an explicitly calculated metric completion relied on realising the completion as a subcategory an auxiliary, well-understood triangulated category known as \textit{a good extension} (Definition~\ref{D:GoodExtension}).
As a~consequence, most of the contemporary results employ only those metrics on a triangulated category \(\mathcal{S}\) determined by either $t$-structures on \(\mathcal{S}\), or $t$-structures on a good extension \(\mathcal{T}\) of \(\mathcal{S}\).
Only few systematic classifications of metric completions \cite{CummingsGratz24,Matousek26C,Matousek26B} of a specific triangulated category have been done.

In his previous work \cite{Matousek26B}, the author identified that certain completions of \(\derived^b(\modf\dashmodule R)\) for a well-behaved hereditary noetherian ring $R$ were equivalent to \(\derived^b(\modf\dashmodule Q)\) for $Q$ a universal localisation of~$R$. This result relied heavily on the theory of the aforementioned good extensions. Since \(\derived^b(\modf\dashmodule Q)\) can be also obtained from \(\derived^b(\modf\dashmodule R)\) as a Verdier quotient by some thick subcategory, it is natural to ask whether similar results are possible for other triangulated categories. Using the theory of continuous functors (Theorem~\ref{T:IntroductionContinuous}), we arrived at an affirmative answer (Theorem~\ref{T:IntroductionQuotient}) even without resorting to good extensions.

Technically, a good metric on a triangulated category \(\mathcal{S}\) is a collection \(\mathcal{M}=\{B_n\}_{n\in\N}\) of subcategories of \(\mathcal{S}\) (Definition~\ref{D:Metric}). We can, in principle, make all the \(B_n\)'s equal to a fixed thick subcategory \(\mathcal{C}\subseteq\mathcal{S}\). Such a constant metric postulates that a length of a morphism $f$ is zero if and only if \(\cone f\in\mathcal{C}\).

\begin{theorem}[Theorem~\ref{T:EmbeddingIntoQuotient}, Corollary~\ref{C:CountablyGeneratedCompletionVerdierQuotient}]
\label{T:IntroductionQuotient}
Let \(\mathcal{C}\) be a thick subcategory of a triangulated category \(\mathcal{S}\). Then there is a fully faithful triangulated functor \(\mathfrak{S}_{\mathcal{C}}(\mathcal{S})\xhookrightarrow{}\mathcal{S}/\mathcal{C}\) where the metric completion \(\mathfrak{S}_{\mathcal{C}}(\mathcal{S})\) is taken with respect to the constant metric \(\{B_n\}_{n\in\N}=\{\mathcal{C}\}_{n\in\N}\).

Moreover, if \(\mathcal{S}\) is Hom-finite and \(\kk\)-linear over a commutative ring \(\kk\) and \(\mathcal{C}\) is countably generated as a triangulated subcategory, then \(\mathfrak{S}_{\mathcal{C}}(\mathcal{S})\simeq\mathcal{S}/\mathcal{C}\).
\end{theorem}

In Section~\ref{sec:DCC} and Section~\ref{sec:CompletingWPL}, we present two sample applications of the fact that we may now recover Verdier quotients as specific instances of metric completions.  

The first application revolves around
discrete cluster categories \(\mathcal{C}_p\), introduced by Igusa and Todorov \cite{IgusaTodorov15} as a generalisation of cluster categories of infinite Dynkin type \cite{HolmJoergensen12}. Those are triangulated categories possessing a combinatorial model using arcs of an infinity-gon (Subsection~\ref{SS:DiscreteClusterCategories}). They have led to multiple applications in infinity cluster combinatorics \cite{AugustCheungFaberGratzSchroll23,Fisher17,GratzHolmJoergensen19,GratzZvonareva23,Murphy25A}. In \cite{PaquetteYildirim21}, Paquette and Yıldırım constructed a combinatorial completion \(\overline{\mathcal{C}_p}\) (without using any metric techniques) for the discrete cluster category by taking a Verdier quotient of a larger cluster category. On the level of the combinatorial model, this completion procedure adjoins additional arcs to the model, allowing for further combinatorial arguments \cite{Franchini25,Murphy25B,CanakciKalckPressland25}.

Cummings and Gratz conducted an exhaustive study of metric completions on discrete cluster categories via $t$-structure approximations in \cite{CummingsGratz24}. They provide an explicit combinatorial description of all possible metric completions of the category taken with respect to good additive metrics determined by aisles and coaisles of $t$-structures.
On top of that, for the specific choice of a coaisle metric of a non-degenerate, bounded above $t$-structure, the coaisle metric completion turns out to be additively equivalent to the Paquette-Yıldırım combinatorial completion (see \cite[Corollary~6.24, Remark~6.25]{CummingsGratz24}). The authors of \cite{CummingsGratz24} ask whether this additive equivalence could be triangulated.

Given that \(\overline{\mathcal{C}_p}\) is formed as a certain Verdier quotient, we can now view it as metric completions as well via Theorem~\ref{T:IntroductionQuotient}. The question now boils down to comparing two potentially different metric completions. Using Theorem~\ref{T:IntroductionContinuous}, we can confirm the conjecture.

\begin{theorem}[Theorem~\ref{T:TheEquivalenceIsAnEquivalence}]
Let \(\mathfrak{S}(\mathcal{C}_p)\) be a metric completion of a discrete cluster category \(\mathcal{C}_p\) with respect to a coaisle metric given by a non-degenerate, bounded above $t$-structure. Then there exists a~triangulated equivalence \(\mathfrak{S}(\mathcal{C}_p)\xrightarrow{\simeq}\overline{\mathcal{C}_p}\), which coincides as an additive functor with the~equivalence from \cite[Corollary~6.24]{CummingsGratz24}.
\end{theorem}

As the final main result, we are able to bring our newfound understanding of Verdier quotients as completions to fruition by generalising the classification result \cite[Theorem~1.1]{Matousek26B}. The original theorem fully classifies additive metric completions of \(\derived^b(\modf\dashmodule R)\) certain hereditary rings $R$, including finite dimensional algebras of tame representation type. Here, we describe the situation in a more intricate situation of the bounded derived category \(\dbcohX\) of coherent sheaves on a weighted projective line \(\X\) of domestic or tubular type.
Note that indecomposable objects of \(\cohX\) can be assigned a numerical invariant called \textit{a slope} attaining values in \(\Q\cup\{\infty\}\).

\begin{theorem}[{Theorem~\ref{T:WeightedCompletions}}]
Let \(\X\) be a weighted projective line of a domestic or tubular type over an algebraically closed field.
Equip \(\dbcohX\) with an additive good metric \(\mathcal{M}=\{B_n\}_{n\in\N}\) and denote the thick subcategory \(\mathcal{B}:=\bigcap_{n\in\N} B_n\).
Then:
\begin{enumerate}[label=(\roman*)]
\item If \(\mathcal{M}\sim\mathcal{M}_{\infty}\vee\mathcal{B}\) (see Subsection~\ref{SS:Compressions}, Lemma~\ref{L:DecompostionMetricsCompletion}) and \(\mathcal{B}\) is countably generated, then
\[
\mathfrak{S}_{\mathcal{M}}\left(\dbcohX\right)\simeq\mathfrak{S}_{\mathcal{B}}\left(\dbcohX\right)\simeq\dbcohX/\mathcal{B}
.\]
 In this case, it holds that \(\mathfrak{S}_{\mathcal{M}}\big(\dbcohX\big)\simeq\mathcal{B}^{\perp}\cap\dbcohX\) if and only if \(\mathcal{B}\) is generated by an~exceptional sequence.
\item Otherwise, there exists \(q\in\Q\cup\{\infty\}\) depending on \(\mathcal{M}\) such that
\[
\mathfrak{S}_{\mathcal{M}}\left(\dbcohX\right)\simeq\mathcal{B}^{\perp}\cap\mathscr{H}^{(q)}
\]
where \(\mathscr{H}^{(q)}\subseteq\dbcohX\) is a thick subcategory generated by all indecomposable objects from \(\cohX\) of slope $q$.
\end{enumerate}
\end{theorem}

The abelian category \(\cohX\) has a~tilting object, so \(\dbcohX\) is triangle equivalent to a~bounded derived category of a finite dimensional algebra, while not necessarily a hereditary one.
In fact, for \(\X\) domestic \(\dbcohX\simeq\derived^b(\modf\dashmodule A)\) where $A$ is a tame hereditary finite dimensional algebra.
The key new aspect of the result is the case of \(\X\) tubular, where no previous classification of completions was available. 

\subsubsection*{Acknowledgments}
I would like to express my gratitude to Sira Gratz for her supervision during my Ph.D.\ studies.
I wish to thank Gustavo Jasso for introducing me to weighed projective lines.
I would also like to thank Isaac Bird, David Nkansah, Greg Stevenson, and Alexandra Zvonareva for useful discussions. 

At Aarhus University, the author was supported by a research grant (VIL42076) from VILLUM FONDEN.
At Paderborn University, the author was partly supported by the Deutsche Forschungsgemeinschaft (DFG, German Research Foundation): Project-ID 491392403 (TRR 358).

\subsubsection*{Declaration}
There is a substantial text overlap with two previously unpublished chapters \cite[Introduction, Project~C]{Matousek26D} of the author's Ph.D.~thesis, which served as a draft version of Section~\ref{sec:introduction}, Section~\ref{sec:preliminaries}, and Section~\ref{sec:CompletingWPL} of this manuscript. Corollary~\ref{C:EkvilaleceEkvivalencakEkvivalencata} appeared as \cite[{Lemma~9.2}]{Matousek26D}.
Otherwise, Section~\ref{sec:continuous}, Section~\ref{sec:quotients}, and Section~\ref{sec:DCC} are completely new.

\section{Preliminaries}
\label{sec:preliminaries}

\subsection{Notation and conventions}
By \textit{a category}, we mean a category \(\mathcal{A}\) which is locally small; that is \(\Hom_{\mathcal{A}}(X,Y)\) forms a set. We shall write \(A\in\mathcal{A}\) to express that $A$ is an~object of \(\mathcal{A}\), and \(\mathcal{B\subseteq\mathcal{A}}\) expresses that \(\mathcal{B}\) is a subcategory of \(\mathcal{A}\). All subcategories are considered to be full and replete.

An adjoint pair of functors is depicted as \(F:\mathcal{A}\rightleftarrows\mathcal{B}:G\), meaning that \(F\) is a~left adjoint to $G$.
Any functor \(F:\mathcal{A}\rightarrow\mathcal{B}\) between additive categories is automatically assumed to be additive. 

Let \(\mathcal{A}\) be an additive category and \(A\in\mathcal{A}\).
We denote the two orthogonal categories to $A$ as
\[
A^{\perp}:=\big\{B\in\mathcal{A}:\Hom_{\mathcal{A}}(A,B)\simeq0\big\}
\text{ and }
{}^{\perp}A:=\big\{B\in\mathcal{A}:\Hom_{\mathcal{A}}(B,A)\simeq0\big\}
.\]
The subcategory \(\addf(A)\subseteq\mathcal{A}\) consists, by definition, of direct summands of finite coproducts of \(A\in\mathcal{A}\).
Analogous notations \(\mathcal{B}^{\perp}\), \({}^{\perp}\mathcal{B}\), and \(\addf(\mathcal{B})\) are employed when \(\mathcal{B}\subseteq\mathcal{A}\).

For an essentially small additive category, the category \(\Mod\dashmodule\mathcal{A}\) of modules over \(\mathcal{A}\) consists of contravariant functors \(\mathcal{A}^{\opposite}\rightarrow\Ab\) from \(\mathcal{A}\) to \(\Ab\).
The category \(\mathcal{A}\) embeds into \(\Mod\dashmodule\mathcal{A}\) via the Yoneda embedding \(\Yoneda:\mathcal{A}\rightarrow\Mod\dashmodule\mathcal{A}\), \(A\mapsto\Hom_{\mathcal{A}}(\blank,A)\).

\subsection{Triangulated categories}
The shift on all triangulated categories is denoted as \(\Sigma\).
In this subsection, we fix a triangulated category \(\mathcal{T}\). 
For \(\mathcal{X},\mathcal{Z}\subseteq\mathcal{T}\) we denote
\[
\mathcal{X}*\mathcal{Z}:=\{Y\in\mathcal{T}:\exists X\in\mathcal{X},\exists Z\in\mathcal{Z},\exists \text{ a triangle }X\rightarrow Y\rightarrow Z\rightarrow \Sigma X\}
.\]
The thick subcategory generated by \(\mathcal{C}\subseteq\mathcal{T}\) is denoted as \(\langle\mathcal{C}\rangle\).
If a triangulated subcategory of \(\mathcal{T}\) is generated by finitely many or countably many objects, then it is called \textit{countably generated}.

\begin{definition}[{\cite[Définition~1.3.1]{BeilinsonBernsteinDeligne82},\cite{BeilinsonBernsteinDeligneGabber18}}]
We say that a pair \(\left(\mathcal{T}^{\leq0},\mathcal{T}^{>0}\right)\) of subcategories of \(\mathcal{T}\) is \textit{a~t\=/structure} if the following conditions hold:
\begin{itemize}
    \item \(\mathcal{X}=\addf(\mathcal{X})\) and \(\mathcal{Y}=\addf(\mathcal{Y})\),
    \item \(\mathcal{Y}\subseteq\mathcal{X}^{\perp}\),
    \item \(\mathcal{T}=\mathcal{X}*\mathcal{Y}\).
    \item \(\Sigma\mathcal{T}^{\leq0}\subseteq\mathcal{T}^{\leq0}\).
\end{itemize}

In that case, we call \(\mathcal{T}^{\leq0}\) \textit{an aisle}, \(\mathcal{T}^{>0}\) \textit{a coaisle}, and for each \(n\in\Z\) we denote \(\mathcal{T}^{\leq n}:=\Sigma^{-n}\mathcal{T}^{\leq0}\) and \(\mathcal{T}^{>n}:=\Sigma^{-n}\mathcal{T}^{>0}\). 
Moreover, the $t$-structure \(\left(\mathcal{T}^{\leq0},\mathcal{T}^{>0}\right)\) is said to be \textit{bounded above} if \(\mathcal{T}=\bigcup_{n\in\Z}\mathcal{T}^{\leq n}\),
\end{definition}

We recall that a subcategory \(\mathcal{B}\) of an additive category \(\mathcal{A}\) is called \textit{reflective} if the~inclusion functor \(\mathcal{B}\xhookrightarrow{}\mathcal{A}\) has a left adjoint. Dually, the subcategory \(\mathcal{B}\) is \textit{coreflective} if the inclusion \(\mathcal{B}\xhookrightarrow{}\mathcal{A}\) has a right adjoint. If \(\mathcal{B}\) is both reflective and coreflective, then it is \textit{bireflective}. (Some authors call bireflective subcategories \textit{admissible}.)

The aisle of a $t$-structure \(\left(\mathcal{T}^{\leq0},\mathcal{T}^{>0}\right)\) 
is coreflective within \(\mathcal{T}\). The right adjoint \(\tr^{\leq0}:\mathcal{T}\rightarrow\mathcal{T}^{\leq0}\)  to the inclusion is called \textit{the truncation}. Dually, the coaisle \(\mathcal{T}^{>0}\) is reflective with the truncation functor \(\tr^{>0}:\mathcal{T}\rightarrow\mathcal{T}^{>0}\) as a left adjoint to the respective inclusion.
Any distinguished triangle
\[
X\rightarrow Y\rightarrow Z\rightarrow\Sigma X
\]
in \(\mathcal{T}\) with \(X\in\mathcal{T}^{\leq0}\) and \(Z\in\mathcal{T}^{>0}\) is uniquely isomorphic a distinguished triangle
\[
\tr^{\leq0}X\rightarrow X\rightarrow\tr^{>0}X\rightarrow\Sigma\tr^{\leq0}X
\]
called \textit{the approximation triangle} (see \cite[Proposition~1.3.3]{BeilinsonBernsteinDeligne82}).

We can even find out whether thick subcategory sits in a $t$-structure pair solely based on the existence of adjoints to the inclusion functor.

\begin{proposition}
\label{P:CharacterisationReflectivity}
Let \(\mathcal{S}\) be a thick subcategory of \(\mathcal{T}\). Then:
\begin{itemize}
    \item \(\mathcal{S}\) is reflective \(\Leftrightarrow\) \(\mathcal{S}\) is preenveloping \(\Leftrightarrow\) \(\mathcal{S}\) is enveloping \(\Leftrightarrow\) \(\left({}^{\perp}\mathcal{S},\mathcal{S}\right)\) is a $t$-structure. 
    \item \(\mathcal{S}\) is coreflective \(\Leftrightarrow\) \(\mathcal{S}\) is precovering \(\Leftrightarrow\) \(\mathcal{S}\) is covering \(\Leftrightarrow\) \(\left(\mathcal{S},\mathcal{S}^{\perp}\right)\) is a $t$-structure. 
    \item \(\mathcal{S}\) is bireflective \(\Leftrightarrow\) \(\left({}^{\perp}\mathcal{S},\mathcal{S}\right)\) and \(\left(\mathcal{S},\mathcal{S}^{\perp}\right)\) are both simultaneously $t$-structures.
\end{itemize}
\end{proposition}

\begin{proof}
By \cite[Proposition~1.4]{Neeman10} and \cite[Theorem~2.9]{LakingVitoria20}, the thick subcategory \(\mathcal{S}\) is coreflective if and only if it is (pre)covering. Then \cite[Proposition~2.3]{IyamaYoshino08} implies that \(\big(\mathcal{S},\mathcal{S}^{\perp}\big)\) is a~$t$-structure. Conversely, if \(\big(\mathcal{S},\mathcal{S}^{\perp}\big)\) is a~$t$-structure, then the right adjoint to the inclusion \(\mathcal{S}\xhookrightarrow{}\mathcal{T}\) is given by the truncation.

The proof of the remaining statements is analogous.
\end{proof}

\begin{remark}
A comment on terminology above: a morphism \(f:B\rightarrow A\) in an additive category \(\mathcal{A}\) with \(B\in\mathcal{B}\subseteq\mathcal{A}\) is \textit{a \(\mathcal{B}\)-precover} (also known as \textit{a right \(\mathcal{B}\)-approximation}) if any morphism \(f':B'\rightarrow A\) with \(B'\in\mathcal{B}\) factors through \(f\). \textit{A \(\mathcal{B}\)-cover} (also \textit{a minimal right \(\mathcal{B}\)-approximation}) is a \(\mathcal{B}\)-precover \(f:B\rightarrow A\) such that the equality \(fg=f\) for any \(g\in\End_{\mathcal{A}}(B)\) forces $g$ to be an isomorphism.
\textit{Preenvelopes} and \textit{envelopes} are defined dually and also go under the alias of \textit{(minimal) left approximations}.
\end{remark}

Triples \((\mathcal{U},\mathcal{V},\mathcal{W})\) of thick subcategories of \(\mathcal{T}\) such that \((\mathcal{U},\mathcal{V})\) and \((\mathcal{V},\mathcal{W})\) are $t$-structures are known as \textit{TTF\=/triples} (TTF\(=\) \textit{torsion torsion-free}) and correspond to so-called re\-col\-le\-ments of triangulated categories.
By \cite[{Sous-section~1.4.4}]{BeilinsonBernsteinDeligne82} (see also \cite[Subsection~2.1]{NicolasSaorin09}), in this case, the coaisle truncation \(\mathcal{T}\rightarrow\mathcal{W}\) restricted to \(\mathcal{U}\) is an equivalence \(\mathcal{U}\simeq\mathcal{W}\).

\begin{definition}[{\cite{BondalKapranov90}, \cite[Section~6.3]{Krause21}}]
Let $K$ be a field and \(\mathcal{A}\) a Hom-finite, $K$-linear category.
A~$K$\=/linear autoequivalence \(\Serre:\mathcal{A}\rightarrow\mathcal{A}\) is called \textit{a Serre functor} if
there exists an isomorphism
\[
\Hom_{\mathcal{A}}(X,Y)\simeq \Hom_K\Big(\Hom_{\mathcal{A}}\big(Y,\Serre (X)\big),K\Big)
\]
for all \(X,Y\in\mathcal{A}\) natural in both variables.
\end{definition}

If \(\Sigma^n\) is a Serre functor on a $K$-linear triangulated category \(\mathcal{S}\) for some \(n\in\N\), then we say \(\mathcal{S}\) is \textit{\(n\)-Calabi-Yau}. This property gives an insight on orthogonality within \(\mathcal{S}\), since it then holds that \(\Hom_{\mathcal{S}}(X,Y)\simeq0\) if and only if \(\Hom_{\mathcal{S}}(Y,\Sigma^{n}X)\).

\begin{definition}[{\cite[Section~1.8]{ChenKrause10}}]
Let \(\mathcal{A}\) be a Hom-finite, $K$-linear abelian category. A~$K$\=/linear autoequivalence \(\tau:\mathcal{A}\rightarrow\mathcal{A}\) is called \textit{an Auslander-Reiten translation} if
there exists an isomorphism 
\[
 \Hom_K\Big(\Ext^1_{\mathcal{A}}(X,Y),K\Big)\simeq\Hom_{\mathcal{A}}(Y,\tau X)
\]
for all \(X,Y\in\mathcal{A}\) natural in both variables.  
\end{definition}

If \(\mathcal{A}\) has an Auslander-Reiten translation, it is said to \textit{satisfy Serre duality}. Any such category \(\mathcal{A}\) is necessarily Ext-finite and hereditary because \(\Ext^1_{\mathcal{A}}(X,\blank)\simeq\Hom_{K}(\blank,K)\circ\Hom_{\mathcal{A}}(\blank,\tau X)\) is right exact. It follows that \(\Sigma\tau\) is a Serre functor on \(\derived^b(\mathcal{A})\).

\subsection{Metrics and metric completions}
\label{SS:MetricsAndMetricCompletions}

Fix \(\mathcal{S}\) an essentially small triangulated category.
In this subsection, we go through the~definition of a metric on \(\mathcal{S}\), the metric completion of \(\mathcal{S}\), and their most important properties. The concepts have been developed by Neeman in \cite{Neeman25A}, while we mostly follow the~notation from \cite{CummingsGratz24} and \cite{Matousek26B}.

\begin{definition}
\label{D:Metric}
Consider a family \(\mathcal{M}=\{B_n\}_{n\in\N}\) of subcategories of \(\mathcal{S}\).
We say that \(\mathcal{M}\) is \textit{a good metric} on \(\mathcal{S}\) if for each \(n\in\N\) it holds that:
\begin{enumerate}[label=(\roman*)]
\item \(0\in B_n\),
\item \(B_n*B_n=B_n\),
\item \(\Sigma^{-1}B_{n+1}\cup B_{n+1}\cup\Sigma B_{n+1}\subseteq B_n\).
\end{enumerate}
Moreover, a metric \(\mathcal{M}\) is considered to be 
\textit{additive} if
\begin{enumerate}[label=(\roman*)]
\item[(iv)] \(B_n=\addf(B_n)\) for all \(n\in\N\).
\end{enumerate}
\end{definition}

Metrics on triangulated categories have been originally defined by Neeman in \cite[Definition~1.2]{Neeman18}. In practice, one works most often with good metrics introduced later in \cite[Definition~10]{Neeman20}. The additive metrics appear in \cite{CummingsGratz24} without being named explicitly; the terminology of additive metrics comes from \cite[Definition~2.8]{Matousek26B}.

A metric \(\{B_n\}_{n\in\N}\) on \(\mathcal{S}\) imposes on every morphism \(f:X\rightarrow Y\) in \(\mathcal{S}\) a length in the~spirit of Lawvere's norms on categories \cite{Lawvere73}; the morphism \(f\) has length \(\leq\frac{1}{n}\) if and only if \(\cone f\in B_n\) for some \(n\in\N\).

\begin{example}
For any triangulated subcategory \(\mathcal{C}\subseteq\mathcal{S}\), setting \(B_n:=\mathcal{C}\) for every \(n\in\N\) defines a~good metric \(\{B_n\}_{n\in\N}=\{\mathcal{C}\}_{n\in\N}\) on \(\mathcal{S}\). We call such a metric \textit{a constant metric} (\cite[Subsection~2.2.2]{Matousek26B}) and we shall not be making a formal distinction between a~constant good metric and the corresponding triangulated subcategory.
A length of a morphism with respect to the constant metric \(\mathcal{C}\) is zero if and only if its cone lies in \(\mathcal{C}\), otherwise the length can be interpreted as infinite.

The subcategory \(\mathcal{C}\) is thick if and only if the constant metric \(\mathcal{C}\) is additive.
\end{example}

A Cauchy sequence is a diagram of morphisms which are eventually becoming shorter. The definition below combines \cite[Definition~1.6]{Neeman25A}, \cite[Definition~3.2, Definition~3.11]{CummingsGratz24}, and \cite[Definition~2.2]{Matousek26B}.

\begin{definition}
Let \(\mathcal{M}=\{B_n\}_{n\in\N}\) be a good metric on \(\mathcal{S}\).
By \textit{a sequence} in \(\mathcal{S}\) we mean an \(\N\)\=/directed diagram
\(
\mathbf{E}=(E_n,e_n)_{n\in\N}=E_1\xrightarrow{e_1} E_2\xrightarrow{e_2}E_3\rightarrow\cdots
\)
in \(\mathcal{S}\).

We denote the~composition \(e_{m,n}:=e_n\circ e_{n-1}\cdots\circ e_{m+1}\circ e_m\) for any \(m\leq n\).
We say that \(\mathbf{E}\) \textit{stabilises at} \(\mathcal{C}\subseteq\mathcal{S}\) \textit{from} \(N\in\N\) \textit{onwards} if for all \(n\geq m\geq N\) we have \(\cone e_{m,n}\in\mathcal{C}\). 

We call \(\mathbf{E}\) \textit{a~Cauchy sequence} with respect to \(\mathcal{M}\) if for each \(M\in\N\) the sequence \(\mathbf{E}\) stabilises at \(B_M\) from some \(N\in\N\) number onwards; to simplify the notation, in this case, we say that \(\mathbf{E}\) \textit{stabilises at} \(M\) \textit{from} $N$ \textit{onwards}.
\end{definition}

\begin{definition}[{\cite[Definition~3.16]{CummingsGratz24}}]
Let \(\mathcal{C}\subseteq\mathcal{S}\).
A sequence \(\mathbf{E}=(E_n,e_n)_{n\in\N}\) in \(\mathcal{S}\) is said to be \(\mathcal{C}\)\textit{-trivial} if there is no \(n\in\N\) such that \(E_n\) has a direct summand from \(\mathcal{C}\).
\end{definition}

\begin{definition}[{\cite[Definition~3.4]{CummingsGratz24}}]
Let \(\mathbf{E}=(E_n,e_n)\) be a sequence in \(\mathcal{S}\). We define \textit{the module colimit} of \(\mathbf{E}\) as the functor \(\moco \mathbf{E}:=\varinjlim \Yoneda(\mathbf{E})\in\Mod\dashmodule\mathcal{S}\).  
\end{definition}

\begin{definition}[{\cite[Definition~1.11]{Neeman25A}}]
\label{D:Completion}
Let \(\mathcal{M}=\{B_n\}_{n\in\N}\) be a good metric on \(\mathcal{S}\), and let \(s\in\N\).
We say that \(F\in\Mod\dashmodule\mathcal{S}\) is \textit{compactly supported at} \(B_s\) (or \textit{compactly supported at} \(s\) for short) if \(F(B_s)=0\).
We also define the following subcategories of \(\Mod\dashmodule\mathcal{S}\):
\begin{center}
\begin{itemize}
    \item
    \textit{the pre-com\-ple\-tion} of \(\mathcal{S}\) with respect to \(\mathcal{M}\)
    \[
    \mathfrak{L}_{\mathcal{M}}(\mathcal{S}):=\left\{F\in\Mod\dashmodule\mathcal{S}:F\simeq\moco\mathbf{E} \text{\ \ for some Cauchy sequence } \mathbf{E} \text{ in } \mathcal{S}\right\}
    ,\]
    \item \(\mathfrak{C}_{\mathcal{M}}(\mathcal{S}):=\big\{F\in\Mod\dashmodule\mathcal{S}:\exists n\in\N,F(B_n)\simeq0\big\}\) the category of \textit{compactly supported} elements of \(\Mod\dashmodule\mathcal{S}\) with respect to \(\mathcal{M}\),
    \item
    \(\mathfrak{S}_{\mathcal{M}}(\mathcal{S}):=\mathfrak{L}_{\mathcal{M}}(\mathcal{S})\cap\mathfrak{C}_{\mathcal{M}}(\mathcal{S})\)
    called \textit{the completion} of \(\mathcal{S}\) with respect to \(\mathcal{M}\).
\end{itemize}
\end{center}
\end{definition}

We will omit the index if the good metric is implicit from the context. A sequence \(\mathbf{E}\) in \(\mathcal{S}\) is said to be \textit{compactly supported} if \(\moco\mathbf{E}\) is compactly supported. 
Likewise, an object \(X\in\mathcal{S}\) is \textit{compactly supported} if \(\Yoneda(X)\) is compactly supported. And as long as \(X\in\mathcal{S}\) is compactly supported, it holds that \(\Yoneda(X)\in\mathfrak{S}(\mathcal{S})\) because \(\Yoneda(X)\simeq \moco \mathbf{X}\) for the constant Cauchy sequence \(\mathbf{X}=(X,\id_X)_{n\in\N}\).

\begin{theorem}[{\cite[Theorem~2.14]{Neeman25A}}]
For any good metric \(\mathcal{M}\) on \(\mathcal{S}\), the completion \(\mathfrak{S}_{\mathcal{M}}(\mathcal{S})\) is a triangulated category.
\end{theorem}

The triangulated structure on \(\mathfrak{S}_{\mathcal{M}}(\mathcal{S})\) is induced by the original triangulated structure on \(\mathcal{S}\). The shift of a functor \(F\in\mathfrak{S}_{\mathcal{M}}(\mathcal{S})\) is given by \(\Sigma F:=F\circ\Sigma^{-1}\).
And a candidate triangle in \(\mathfrak{S}_{\mathcal{M}}(\mathcal{S})\) is proclaimed to be distinguished if it is a module colimit of a Cauchy sequence of distinguished triangles from \(\mathcal{S}\).

\begin{example}
\label{E:ZeroMetric}
If we interpret the zero subcategory of \(\mathcal{S}\) as a constant metric on \(\mathcal{S}\), we get an equivalence \(\mathfrak{S}_{0}(\mathcal{S})\simeq\mathcal{S}\). This is because every Cauchy sequence \(\mathbf{E}=(E_n,e_n)_{n\in\N}\) with respect to~$0$ satisfies that \(e_n\) is an isomorphism for all but finitely many \(n\in\N\). Thus \(\mathfrak{L}_0(\mathcal{S})\simeq\Yoneda(\mathcal{S})\).
Every functor from \(\Mod\dashmodule\mathcal{S}\) is compactly supported at $0$ by the definition of an additive functor. 

We conclude \(\mathfrak{S}_{0}(\mathcal{S})=\mathfrak{L}_{0}(\mathcal{S})\cap\mathfrak{C}_0(\mathcal{S})\simeq\Yoneda(\mathcal{S})\), so \(\mathfrak{S}_{0}(\mathcal{S})\simeq\mathcal{S}\) follows by the Yoneda lemma.
\end{example}

\begin{example}[{\cite[Example~1.13]{Neeman25A}}]
On the other hand, if we take \(\mathcal{S}\) as a constant metric on itself, we get \(\mathfrak{S}_{\mathcal{S}}(\mathcal{S})\simeq0\). The only functor \(F\in\Mod\dashmodule\mathcal{S}\) vanishing on the whole of \(\mathcal{S}\) is the zero functor \(F\simeq0\). Hence, there are no non-zero compactly supported elements. 
\end{example}

Every $t$-structure \(\big(\mathcal{S}^{\leq0},\mathcal{S}^{>0}\big)\) on \(\mathcal{S}\) provides two (generally) non-constant good metric, \textit{the aisle metric} \(\big\{\mathcal{S}^{\leq -n}\big\}_{n\in\N}\) and  \textit{the coaisle metric} \(\big\{\mathcal{S}^{> n}\big\}_{n\in\N}\).
For aisle metrics the subsequent completions are fully understood.
\begin{theorem}[{\cite[Theorem~1.1]{CummingsGratz24}}]
\label{T:CummingsGratzAisleTheorem}
Let \(\big(\mathcal{S}^{\leq0},\mathcal{S}^{>0}\big)\) be a $t$-structure on \(\mathcal{S}\). Equip \(\mathcal{S}\) with the aisle metric. Then \(\mathfrak{S}(\mathcal{S})\simeq\bigcup_{n\in\N}\mathcal{S}^{>-n}\).
\end{theorem}

\begin{corollary}
\label{C:CoreflectiveGivesPerp}
Let \(\mathcal{C}\) be a coreflective thick subcategory of \(\mathcal{S}\).
Then it holds that \(\mathfrak{S}_{\mathcal{C}}(\mathcal{S})\simeq\mathcal{C}^{\perp}\subseteq\mathcal{S}\).
\end{corollary}

\begin{proof}
As \(\mathcal{C}\) is coreflective, it is the aisle of the $t$-structure \(\left(\mathcal{C},\mathcal{C}^{\perp}\right)\) on \(\mathcal{S}\) by Proposition~\ref{P:CharacterisationReflectivity}.
But then
\mbox{\(\mathfrak{S}_{\mathcal{C}}(\mathcal{T})\simeq\mathcal{C}^{\perp}\)}
by Theorem~\ref{T:CummingsGratzAisleTheorem} because, in this case, the coaisle \(\mathcal{C}^{\perp}\) is stable under~\(\Sigma\).
\end{proof}

Natural transformations from a module colimit of a Cauchy sequence to a compactly supported functor can be understood in terms of representable functors corresponding to the entries of the sequence. 

\begin{lemma}[{\cite[Lemma 2.10]{Neeman25A}}]
\label{L:BasicFactorisationProperty}
Equip $\mathcal{S}$ with a~good metric.
Let \(E=\moco \mathbf{E}\in\mathfrak{L}(\mathcal{S})\) for~some Cauchy sequence \(\mathbf{E}=(E_n,e_n)_{n\in\N}\) in \(\mathcal{S}\).
Let us have a cohomological functor \(F\in\Mod\dashmodule\mathcal{S}\).

If \(F\) is compactly supported at \(s\in\N\) and \(\mathbf{E}\) stabilises at $s+1$ from \(N\in\N\) onwards,
then
\[
\Hom_{\Mod\dashmodule\mathcal{S}}(\varphi_n,F):\Hom_{\Mod\dashmodule\mathcal{S}}(E,F)\rightarrow\Hom_{\Mod\dashmodule\mathcal{S}}\left(\Yoneda(E_n),F\right)
\]
is an isomorphism for all \(n\geq N\) where \(\varphi_n:\Yoneda(E_n)\rightarrow E\) is the~colimit injection. 
\end{lemma}

It has been proved in \cite[Lemma~3.17]{CummingsGratz24}, that in Krull-Schmidt categories working with additive metrics allows us to remove redundant direct summands from Cauchy sequences. We state a reformulation from \cite[Corollary~4.3]{Matousek26B}.

\begin{lemma}
\label{L:CummingsGratzTrivialisation}
Assume \(\mathcal{S}\) is Krull-Schmidt and equipped with an~additive good metric \(\mathcal{M}=\{B_n\}_{n\in\N}\).
Let \(s\in\N\).
If \(\mathbf{E}\) is a Cauchy sequence in \(\mathcal{S}\) compactly supported at \(s\), then there exists a~\(B_s\)\=/trivial Cauchy sequence \(\mathbf{F}\) in \(\mathcal{S}\) such that \(\mathbf{F}\) is compactly supported at~\(s\) and \(\moco \mathbf{E}\simeq\moco \mathbf{F}\).
\end{lemma}

\subsubsection{Good extensions}

Assume that there exists a fully faithful triangulated functor \(\nu:\mathcal{S}\xhookrightarrow{}\mathcal{T}\). Then the~functors from \(\Mod\dashmodule\mathcal{T}\) act on \(\mathcal{S}\) by restriction, giving us a functor \(\Mod\dashmodule\mathcal{T}\rightarrow\Mod\dashmodule\mathcal{S}\), \(F\mapsto F\circ\nu\).
If we limit our choice of functors from \(\Mod\dashmodule\mathcal{T}\) to representable functors, we obtain \textit{the~restricted Yoneda functor}
\(\YonedaKatakana:\mathcal{T}\rightarrow\Mod\dashmodule\mathcal{S}\), \(T\mapsto\Hom_{\mathcal{T}}\big(\nu(\blank),T\big)\) (see \cite[Notation~3.1]{Neeman25A}).

The notation of the restricted Yoneda functor will be used in the following definition. Also recall from \cite[Definition~1.6.4]{Neeman01} that \textit{the homotopy colimit} \(\hoco\mathbf{E}\) of a sequence \(\mathbf{E}=(E_n,e_n)\) in \(\mathcal{T}\)
is determined by the distinguished triangle
\[
\coprod_{n\in\N}E_n\xrightarrow{(\id_{E_n}-e_n)_{n\in\N}}\coprod_{n\in\N}E_n\rightarrow\hoco\mathbf{E}\rightarrow\coprod_{n\in\N}\Sigma E_n
\]
provided that countable coproducts in \(\mathcal{T}\) exist.

\begin{definition}[{\cite[Definition~3.5]{Neeman18}}]
\label{D:GoodExtension}
Let \(\mathcal{T}\) be a triangulated category with countable coproducts. Equip \(\mathcal{S}\) with a good metric \(\mathcal{M}\).
A fully faithful triangulated functor \(\nu:\mathcal{S}\rightarrow\mathcal{T}\) is called \textit{a good extension} with respect to \(\mathcal{M}\) if \(\moco \mathbf{E}\simeq\YonedaKatakana\big(\hoco\nu(\mathbf{E})\big)\) holds for every Cauchy sequence \(\mathbf{E}=(E_n,e_n)\) in~\(\mathcal{S}\).
\end{definition}

\begin{definition}[{\cite[Definition~3.11]{Neeman25A}}]
\label{D:IntroductionDefinitionGoodExtensionCompletion}
Equip \(\mathcal{S}\) with a good metric \(\mathcal{M}\). Let \(\mathcal{S}\xhookrightarrow{}\mathcal{T}\) be a good extension with respect to \(\mathcal{M}\). We define the following subcategories of \(\mathcal{T}\):    
\begin{itemize}
    \item
    \(
    \mathfrak{L}_{\mathcal{M}}'(\mathcal{S}):=\left\{T\in\mathcal{T}:T\simeq\hoco\nu(\mathbf{E}) \text{\ \ for some Cauchy sequence } \mathbf{E} \text{ in } \mathcal{S}\right\}
    ,\)
    \item \(\mathfrak{C}_{\mathcal{M}}'(\mathcal{S}):=\YonedaKatakana^{-1}\big(\mathfrak{C}_{\mathcal{M}}(\mathcal{S})\big)\),
    \item\(\mathfrak{S}_{\mathcal{M}}'(\mathcal{S}):=\mathfrak{L}_{\mathcal{M}}'(\mathcal{S})\cap\mathfrak{C}_{\mathcal{M}}'(\mathcal{S})\).
\end{itemize}
\end{definition}

The point of the whole construction of \(\mathfrak{S}'_{\mathcal{M}}(\mathcal{S})\) is that it allows us to exhibit the metric completion \(\mathfrak{S}_{\mathcal{M}}(\mathcal{S})\) as a triangulated subcategory of the good extension. Again, we will omit the lower index indicating the metric if it is not necessary.

\begin{theorem}[{\cite[Theorem~3.23]{Neeman25A}}]
Equip \(\mathcal{S}\) with a good metric. Let \(\mathcal{S}\xhookrightarrow{}\mathcal{T}\) be a good extension.
Then \(\mathfrak{S}'(\mathcal{S})\) is a triangulated subcategory of \(\mathcal{T}\) and \(\YonedaKatakana\restriction_{\mathfrak{S}'(\mathcal{S})}:\mathfrak{S}'(\mathcal{S})\xrightarrow{\simeq}\mathfrak{S}(\mathcal{S})\) is a~well-defined triangulated equivalence.
\end{theorem}

The following fact is a direct consequence of \cite[Lemma~2.8]{Neeman96}. As the compact elements \(\mathcal{T}^c\subseteq\mathcal{T}\) respect homotopy colimits, they possess an inherent good extension working for any choice of a metric.  

\begin{theorem}[{\cite[Example~3.9]{Neeman25A}}]
\label{T:PrototypicalGoodExtension}
Let \(\mathcal{T}\) be a triangulated category with countable coproducts.
Then \(\mathcal{T}^c\xrightarrow{}\mathcal{T}\) is a good extension with respect to all good metrics on \(\mathcal{T}^c\).
\end{theorem}

\section{Continuous functors}
\label{sec:continuous}

Throughout this section we fix essentially small triangulated categories \(\mathcal{S}\) and \(\mathcal{T}\).
In Subsection~\ref{SS:Compressions}, we recall the lattice theory of metrics and the Sun-Zhang compression functors.
We introduce continuous functors (Definition~\ref{D:Continuous}) as specific instances of compressions in Subsection~\ref{SS:ContinuousInduced} and show that they induce functors on the level of completions (Theorem~\ref{T:CompletionOfContinuousFunctor}).
Subsection~\ref{SS:ContinuousProperties} is dedicated to proving basic properties of the~induced functors.

\subsection{Lattice of metrics and compressions}
Sun and Zhang introduced compression functors between triangulated categories with good metrics in~\cite{SunZhang21}. They prove that an adjoint pair \(\psi:\mathcal{S}\rightleftarrows\mathcal{T}:\varphi\) of compression functors induces a triangulated functor \(\mathfrak{S}(\mathcal{T})\rightarrow\mathfrak{S}(\mathcal{S})\) (see~\cite[Theorem~1.1]{SunZhang21}). Roughly speaking, the compression property of \(\varphi\) is needed for the induced functor to preserve module colimits of Cauchy sequences, and \(\psi\) being a compression forces the induced functor to preserve compactly supported elements. This is why we need a whole adjoint pair to induce a single functor in one direction only.    
 
Our strategy to remove this obstacle is to define continuous functors more restrictively, so that one such a functor is enough to induce a functor between completions. 

\label{SS:Compressions}
\begin{definition}[{\cite[Definition~4.1]{SunZhang21}}]
\label{D:Compression}
By \textit{a compression}, we mean a triple \((\varphi,\mathcal{M},\mathcal{N})\) where \(\varphi:\mathcal{S}\rightarrow\mathcal{T}\)  is a triangulated functor, \(\mathcal{M}=\{B_n\}_{n\in\N}\) a good metric on \(\mathcal{S}\), and \(\mathcal{N}=\{C_n\}_{n\in\N}\) a good metric on \(\mathcal{T}\) such that for all \(n\in\N\) there exists \(m\in\N\) satisfying \(\varphi(B_m)\subseteq C_n\).    
\end{definition}

By an abuse of notation, we will often say that \(\varphi:\mathcal{S}\rightarrow\mathcal{T}\) is a compression if there is no possibility of confusion about the metrics involved.  

\begin{lemma}[{\cite[Lemma~4.2]{SunZhang21}}]
\label{L:CompressionsPreserveCauchySequences}
Every compression \(\mathcal{S}\rightarrow\mathcal{T}\) preserves Cauchy sequences. 
\end{lemma}

The order on good metrics on \(\mathcal{S}\), originally defined in \cite[Definition~1.2]{Neeman25A}, can be also phrased in terms of compressions. For metrics \(\mathcal{M}\) and \(\mathcal{N}\) on \(\mathcal{S}\) we write \(\mathcal{M}\leq\mathcal{N}\) if \(\left(\Id_{\mathcal{S}},\mathcal{M},\mathcal{N}\right)\) is a compression. Furthermore, the metrics \(\mathcal{M}\) and \(\mathcal{N}\) are \textit{equivalent} (denoted as \(\mathcal{M}\sim\mathcal{N}\)) if simultaneously \(\mathcal{M}\leq\mathcal{N}\) and \(\mathcal{N}\leq\mathcal{M}\). Equivalent metrics yield equivalent completions (see \cite[Remark~1.12]{Neeman25A}).

In \cite[Definition~1.14]{Neeman26}, Neeman defines the equivalence of $t$-structures on a triangulated category. Translating this definition into the language of metrics, two $t$-structures on a triangulated category are \textit{equivalent} if their aisle metrics are equivalent.

The collection of all equivalence classes of good metrics on \(\mathcal{S}\) forms a lattice (see \cite[Proposition~5.1]{Matousek26B}). Given good metrics \(\mathcal{M}=\{B_n\}_{n\in\N}\) and \(\mathcal{N}=\{C_n\}_{n\in\N}\) good metrics on \(\mathcal{S}\), we have \(\mathcal{M}\wedge\mathcal{N}=\{B_n\cap C_n\}_{n\in\N}\) and \(\mathcal{M}\vee\mathcal{N}=\{D_n\}_{n\in\N}\) where \(D_n\) is the closure of \(B_n\cup C_n\) under triangle extensions. These operation do not depend on the choice of representatives from  the equivalence classes of metrics.

\subsection{Continuous and induced functors}
\label{SS:ContinuousInduced}

\begin{definition}
\label{D:Continuous}
Let \(\varphi:\mathcal{S}\rightarrow\mathcal{T}\) be a triangulated functor. We say that a compression \((\varphi,\mathcal{M},\mathcal{N})\) is \textit{continuous} if $\varphi$ maps compactly supported Cauchy sequences in \(\mathcal{S}\) to compactly supported Cauchy sequences in \(\mathcal{T}\).
In this case, we also call $\varphi$ \textit{a continuous functor}.
\end{definition}

We want to characterise continuous functors in terms of adjoint pairs on the level of module categories. To that end, we need the following standard argument involving Kan extensions. 

\begin{lemma}[{\cite[Exercise~X.3.2]{Lane98}}]
\label{L:KanExtensionRepresentable}
Let \(\mathcal{B}\) and \(\mathcal{C}\) be essentially small additive categories and \(\varphi:\mathcal{B}\rightarrow\mathcal{C}\) an additive functor.
Then \(\blank\circ \varphi:\Mod\dashmodule\mathcal{C}\rightarrow\Mod\dashmodule\mathcal{B}\) has a~left adjoint $L$ satisfying \(L\big(\Yoneda(B)\big)\simeq\Yoneda\big(\varphi(B)\big)\) for all \(B\in\mathcal{B}\). 
\end{lemma}

\begin{proof}
The two module categories consist of functors into the cocomplete category \(\Ab\). Then by \cite[Corollary~6.2.9]{Riehl16}, to a functor \(F\in\Mod\dashmodule\mathcal{B}\) we may assign its left Kan extension \(L(F)\) along $\varphi$. This assignment is functional, and \(L\) is the left adjoint of \(\blank\circ \varphi\). 
For \(B\in \mathcal{B}\) there is an isomorphism
\[
\begin{split}
\Hom_{\Mod\dashmodule\mathcal{C}}\Big(L\big(\Yoneda(B)\big),G\Big)&\simeq\Hom_{\Mod\dashmodule\mathcal{B}}\big(\Yoneda(B),G\circ \varphi\big)\simeq G\big(\varphi(B)\big)\\
&\simeq\Hom_{\Mod\dashmodule\mathcal{C}}\Big(\Yoneda\big(\varphi(B)\big),G\Big)
\end{split}
\]
natural in \(G\in\Mod\dashmodule\mathcal{C}\). We then have \(L\big(\Yoneda(B)\big)\simeq\Yoneda\big(\varphi(B)\big)\) by the Yoneda lemma.
\end{proof}

\begin{lemma}
\label{L:ContinuousAdjointVersion}
Let \(\varphi:\mathcal{S}\rightarrow\mathcal{T}\) be a triangulated functor and \((\varphi,\mathcal{M},\mathcal{N})\) a compression. 
Consider the~adjoint pair
\(
L:\Mod\dashmodule\mathcal{S}\leftrightarrows\Mod\dashmodule\mathcal{T}:\blank\circ \varphi
\)
provided by Lemma~\ref{L:KanExtensionRepresentable}.
Let \(E=\moco\mathbf{E}\in\mathfrak{L}_{\mathcal{M}}(\mathcal{S})\) for a~Cauchy sequence \(\mathbf{E}\) in \(\mathcal{S}\).
Then \(L(E)\simeq\moco \varphi(\mathbf{E})\in\mathfrak{L}_{\mathcal{N}}(\mathcal{T})\).

In particular, the compression \((\varphi,\mathcal{M},\mathcal{N})\) is continuous if and only if \(L\big(\mathfrak{S}_{\mathcal{M}}(\mathcal{S})\big)\subseteq\mathfrak{C}_{\mathcal{N}}(\mathcal{T})\).
\end{lemma}

\begin{proof}
By Lemma~\ref{L:CompressionsPreserveCauchySequences}, the functor $\varphi$ automatically sends Cauchy sequences to Cauchy sequences.
As left adjoints commute with colimits, we have
\[
L(E)\simeq L\big(\varinjlim \Yoneda(\mathbf{E})\big)\simeq\varinjlim L\big(\Yoneda(\mathbf{E})\big)\simeq\varinjlim\Yoneda\big(\varphi(\mathbf{E})\big)\simeq\moco \varphi(\mathbf{E}) 
\]
where the third isomorphism is obtained from Lemma~\ref{L:KanExtensionRepresentable}.
We now see that \(L(E)\in\mathfrak{C}_{\mathcal{N}}(\mathcal{T})\) if and only if the Cauchy sequence \(\varphi(\mathbf{E})\) is compactly supported.
Since $E$ was chosen arbitrarily, the claim follows.     
\end{proof}

\begin{example}
A compression does not have to be continuous in general. 
Assume \(0\not\simeq\mathcal{S}\), and consider the compression \((\Id_{\mathcal{S}},0,\mathcal{S})\)
involving the constant metrics $0$ and \(\mathcal{S}\) on \(\mathcal{S}\).
Then \(\Id_{\Mod\dashmodule\mathcal{S}}\simeq\blank\circ\Id_{\mathcal{S}}\) is its own left adjoint. Let \(0\neq S\in\mathcal{S}\). Then \(\Yoneda(S)\in\mathfrak{S}_0(\mathcal{S})\) by Example~\ref{E:ZeroMetric} but \(\Id_{\Mod\dashmodule\mathcal{S}}\big(\Yoneda(S)\big)\notin\mathfrak{C}_{\mathcal{S}}(\mathcal{S})\) because \(\Yoneda(S)\) does not vanish on $S$.
Lemma~\ref{L:ContinuousAdjointVersion} tells us that \((\Id_{\mathcal{S}},0,\mathcal{S})\) cannot be continuous.
\end{example}

We now prove the main theorem of the section.

\begin{theorem}
\label{T:CompletionOfContinuousFunctor}
Let \(\varphi:\mathcal{S}\rightarrow\mathcal{T}\) be a continuous functor and \((\varphi,\mathcal{M},\mathcal{N})\) the associated compression.
Then the left adjoint to \(\blank\circ \varphi:\Mod\dashmodule\mathcal{T}\rightarrow\Mod\dashmodule\mathcal{S}\) restricts to a triangulated functor \(\mathfrak{S}(\varphi):\mathfrak{S}_{\mathcal{M}}(\mathcal{S})\rightarrow\mathfrak{S}_{\mathcal{N}}(\mathcal{T})\) acting on module colimits of compactly supported Cauchy sequences as \(\moco\mathbf{E}\mapsto\moco \varphi(\mathbf{E})\).
\end{theorem}

\begin{proof}
Denote $L$ the left adjoint to \(\blank\circ \varphi\) given by Lemma~\ref{L:KanExtensionRepresentable}. By Lemma~\ref{L:ContinuousAdjointVersion}, for any Cauchy sequence \(\mathbf{E}\) in \(\mathcal{S}\) we have
\(
L\left(\moco\mathbf{E}\right)\simeq\moco \varphi(\mathbf{E}) 
.\)
Because $\varphi$ preserves Cauchy sequences by Lemma~\ref{L:CompressionsPreserveCauchySequences}, we see that $L$ restricts to a functor \(\mathfrak{L}_{\mathcal{M}}(\mathcal{S})\rightarrow\mathfrak{L}_{\mathcal{N}}(\mathcal{T})\).
Since $L$ also restricts to a functor \(\mathfrak{S}_{\mathcal{M}}(\mathcal{S})\rightarrow\mathfrak{C}_{\mathcal{N}}(\mathcal{T})\), again by  Lemma~\ref{L:ContinuousAdjointVersion}, it yields an additive functor \(\mathfrak{S}_{\mathcal{M}}(\mathcal{S})\rightarrow\mathfrak{S}_{\mathcal{N}}(\mathcal{T})\).

All we need to do is to show that \(L:\mathfrak{S}_{\mathcal{M}}(\mathcal{S})\rightarrow\mathfrak{S}_{\mathcal{N}}(\mathcal{T})\) is triangulated.
Any distinguished triangle in \(\mathfrak{S}_{\mathcal{M}}(\mathcal{S})\) is isomorphic to a module colimit
\[
\moco \mathbf{X}\xrightarrow{} \moco \mathbf{Y}\xrightarrow{}\moco \mathbf{Z}\xrightarrow{}\Sigma\moco \mathbf{X}
\]
of a Cauchy sequence of triangles in \(\mathcal{S}\) and is sent to a candidate triangle
\[
\moco \varphi(\mathbf{X})\xrightarrow{} \moco \varphi(\mathbf{Y})\xrightarrow{}\moco \varphi(\mathbf{Z})\xrightarrow{}L\left(\Sigma\moco \mathbf{X}\right)
\]
by $L$.
Using the definition of the shift on completion and the natural isomorphism  \(\varphi\Sigma\Rightarrow\Sigma \varphi\) establishing $\varphi$ as a triangulated functor, we obtain
\[
\begin{split}
L\left(\Sigma\moco\mathbf{X}\right)&\simeq L\left(\varinjlim\Hom_{\Mod\dashmodule\mathcal{S}}\left(\Sigma^{-1}\blank,\mathbf{X}\right)\right)\simeq L\left(\varinjlim\Hom_{\Mod\dashmodule\mathcal{S}}(\blank,\Sigma\mathbf{X})\right)\\
&\simeq \varinjlim\Hom_{\Mod\dashmodule\mathcal{T}}(\blank,\varphi\Sigma\mathbf{X})\simeq \varinjlim\Hom_{\Mod\dashmodule\mathcal{T}}(\blank,\Sigma \varphi\mathbf{X}) \\
&\simeq\varinjlim\Hom_{\Mod\dashmodule\mathcal{T}}\left(\Sigma^{-1}\blank,\varphi\mathbf{X}\right)\simeq \Sigma\moco \varphi(\mathbf{X})
.\end{split}
\]
In other words, the functor $L$ sends distinguished triangles to distinguished triangles up to a natural isomorphism \(L\Sigma\Rightarrow\Sigma L\) and is thus triangulated.
\end{proof}

\begin{remark}
The induced functor from Theorem~\ref{T:CompletionOfContinuousFunctor} recovers the Sun-Zhang induced functor (from \cite[Theorem~1.1]{SunZhang21})  provided the latter exists.
The assumptions in \cite{SunZhang21} require the existence of an adjoint pair \(\psi:\mathcal{S}\rightleftarrows\mathcal{T}:\varphi\) of compression functors. In this case, there is (by \cite[Theorem~1.1.2, Corollary~3.2.1]{SunZhang21}) an adjoint pair \(\blank\circ\psi:\Mod\dashmodule\mathcal{T}\rightleftarrows\Mod\dashmodule\mathcal{S}:\blank\circ\varphi\) and \(F\circ\psi\in\mathfrak{C}(\mathcal{S})\) for all \(F\in\mathfrak{C}(\mathcal{T})\).
By Lemma~\ref{L:ContinuousAdjointVersion}, the functor \(\varphi\) is then continuous and induces  \(\mathfrak{S}(\varphi):\mathfrak{S}(\mathcal{T})\rightarrow\mathfrak{S}(\mathcal{S})\). 
Since \(\mathfrak{S}(\varphi)\simeq(\blank\circ\psi)\restriction\mathfrak{S}(\mathcal{S})\), it coincides with the~Sun-Zhang functor.
\end{remark}

\begin{corollary}
\label{C:EkvilaleceEkvivalencakEkvivalencata}
Let \(\varphi:\mathcal{S}\rightarrow\mathcal{T}\) be a triangulated equivalence and let \(\mathcal{M}:=\{B_n\}_{n\in\N}\) be a good metric on~\(\mathcal{S}\). For the image metric \(\mathcal{N}:=\big\{\varphi(B_n)\big\}_{n\in\N}\) on \(\mathcal{T}\) we get \(\mathfrak{S}_{\mathcal{M}}(\mathcal{S})\simeq\mathfrak{S}_{\mathcal{N}}(\mathcal{T})\).
\end{corollary}

\begin{proof}
The family \(\mathcal{N}\) of subcategories of \(\mathcal{T}\) is a well-defined good metric by \cite[Lemma~4.14]{Matousek26C} because \(\varphi\) is full.
We get continuous compressions \((\varphi,\mathcal{M},\mathcal{N})\) and \(\left(\varphi^{-1},\mathcal{N},\mathcal{M}\right)\). The functors \(\mathfrak{S}(\varphi)\) and \(\mathfrak{S}\left(\varphi^{-1}\right)\)
induced by Theorem~\ref{T:CompletionOfContinuousFunctor} are mutually inverse triangulated equivalences.
\end{proof}

\subsection{Properties of induced functors}
\label{SS:ContinuousProperties}

In this subsection, we fix a continuous functor \(\varphi:\mathcal{S}\rightarrow\mathcal{T}\) together with its associated compression \((\varphi,\mathcal{M},\mathcal{N})\).
Theorem~\ref{T:CompletionOfContinuousFunctor} gives us an induced triangulated functor \(\mathfrak{S}(\varphi):\mathfrak{S}(\mathcal{S})\rightarrow\mathfrak{S}(\mathcal{T})\).

\begin{definition}[{\cite[Appendix]{RingelZhang14}}]
We say that \(\varphi\) is \textit{objective} if every morphism \(f\) in \(\mathcal{S}\) such that \(\varphi(f)=0\)
factors through an~object from \(\Ker\varphi\).
\end{definition}

\begin{lemma}
\label{L:ObjectiveFaithful}
If \(\varphi\) is objective and \(\Ker\varphi\leq\mathcal{M}\), then \(\mathfrak{S}(\varphi)\) is faithful.
\end{lemma}

\begin{proof}
Let \(g:E\rightarrow F\) be a morphism in \(\mathfrak{S}(\mathcal{S})\) such that \([\mathfrak{S}(\varphi)](g)=0\).
By \cite[Lemma~2.8.i]{Neeman25A}, there exist Cauchy sequences \(\mathbf{E}=(E_n,e_n)_{n\in\N}\)  and \(\mathbf{F}=(F_n,f_n)_{n\in\N}\) and a~morphism of sequences \((g_n:E_n\rightarrow F_n)_{n\in\N}\) with \(g=\moco g_n\).
Suppose \(F\) is compactly supported at \(s\in\N\) and that \(\mathbf{E}\) stabilises at \(s+1\) from \(m\in\N\) onwards. Denote \(e:\Yoneda(E_m)\rightarrow E\) and \(f:\Yoneda(F_m)\rightarrow F\) the colimit injections.

By Theorem~\ref{T:CompletionOfContinuousFunctor}, the functor \(\mathfrak{S}(\varphi)\) is a restriction of the left adjoint \(L:\Mod\dashmodule\mathcal{S}\rightarrow\Mod\dashmodule\mathcal{T}\) of \(\blank\circ \varphi\). Applying $L$ to the commutative square
\begin{center}
	\begin{tikzcd}
		 E \arrow[r,"g"] & F  \\
		  \Yoneda(E_m)  \arrow[u,"e"]  \arrow[r,"\Yoneda(g_m)"] & \Yoneda(F_{m}) \arrow[u,"f"]
		\end{tikzcd}
\end{center}
in \(\Mod\dashmodule\mathcal{S}\) gives us (using Lemma~\ref{L:KanExtensionRepresentable}) a commutative square
\begin{center}
	\begin{tikzcd}
		 L(E) \arrow[r,"0"] & L(F)  \\
		  \Yoneda\big(\varphi(E_m)\big)  \arrow[u,"L(e)"]  \arrow[r,"\Yoneda\big(\varphi(g_m)\big)"] & \Yoneda\big(\varphi(F_m)\big) \arrow[u,"L(f)"]
	\end{tikzcd}
\end{center}
in \(\Mod\dashmodule\mathcal{T}\).
Since \(L(F)\simeq\moco \varphi(\mathbf{F})\simeq\varinjlim_{n\in\N}\Hom_{\mathcal{T}}\big(\blank,\varphi(F_n)\big)\),
the bottom commutative square tells us that the element \(\varphi(g_m)\) is annihilated by the colimit injection
\[
\Hom_{\mathcal{T}}\big(\varphi(E_m),\varphi(F_m)\big)\rightarrow\varinjlim_{n\in\N}\Hom_{\mathcal{T}}\big(\varphi(E_m),\varphi(F_n)\big)
.\]
By \cite[Lemma~2.2]{GobelTrlifaj12}, there exists \(m'\geq m\) such that \(\varphi(f_{m,m'})\varphi(g_m)=0\).
Because the functor \(\varphi\) is objective, the map \(f_{m,m'}g_m\) factors through \(\Ker\varphi\). Since \(F\) is compactly supported with respect to \(\mathcal{M}\), it is also compactly supported with respect to \(\Ker\varphi\leq\mathcal{M}\), so \(F(\Ker\varphi)=0\) and \(f\circ\Yoneda(g_m)=0\).

By Lemma~\ref{L:BasicFactorisationProperty}, the map $g$ is the unique factorisation of \(f\circ\Yoneda(g_m)\) through $e$, so \(g=0\). Hence \(\mathfrak{S}(\varphi)=L\restriction\mathfrak{S}_{\mathcal{M}}(\mathcal{S})\) is faithful.
\end{proof}

\begin{lemma}
\label{L:FullnessIsPreserved}
If \(\varphi\) is full, then \(\mathfrak{S}(\varphi)\) is full.
\end{lemma}

\begin{proof}
Let \(L:\Mod\dashmodule\mathcal{S}\rightarrow\Mod\dashmodule\mathcal{T}\) be the left adjoint to \(\blank\circ\varphi\). By Theorem~\ref{T:CompletionOfContinuousFunctor}, the~functor \(\mathfrak{S}(\varphi)\)
is a~restriction of $L$.
Let \(L(E), L(F)\in\mathfrak{S}(\mathcal{T})\) be objects in the~image of \(\mathfrak{S}(\varphi)\) for some \(E,F\in\mathfrak{S}(\mathcal{S})\).
Pick a morphism \(g:L(E)\rightarrow L(F)\).
Again by Theorem~\ref{T:CompletionOfContinuousFunctor}, it holds that \(L(E), L(F)\) are module colimits of images of Cauchy sequences in \(\mathcal{S}\) under \(\varphi\).
In particular, we can use \cite[Lemma~2.8.i]{Neeman25A} to find Cauchy sequences \(\mathbf{E}=(E_n,e_n)_{n\in\N}\) and \(\mathbf{F}=(F_n,f_n)_{n\in\N}\) in \(\mathcal{S}\) and a morphism of sequences \(\big(g_n:\varphi(E_n)\rightarrow\varphi(F_n)\big)_{n\in\N}\) from \(\varphi(\mathbf{E})\) to \(\varphi(\mathbf{E})\) such that \(g\simeq\moco g_n\).

Suppose that both \(F\) and \(L(F)\) are compactly supported at \(s\in\N\) with respect to their relevant metrics. 
We find \(m\in\N\) such that \(\mathbf{E}\) and \(\varphi(\mathbf{E})\) stabilise at \(s+1\) from \(m\) onwards.
As \(\varphi\) is full, we find \(t:E_m\rightarrow F_m\) satisfying \(\varphi(t)=g_m\).
By Lemma~\ref{L:BasicFactorisationProperty}, there is a unique morphism $h$ making the square
\begin{center}
	\begin{tikzcd}
		 E \arrow[r,"h", dashed] & F  \\
		  \Yoneda(E_m)  \arrow[u,""]  \arrow[r,"\Yoneda(t)"] & \Yoneda(G_{m}) \arrow[u,""].
		\end{tikzcd}
\end{center}
commute. As \(L\big(\Yoneda(t)\big)\simeq\Yoneda\big(\varphi(t)\big)\simeq\Yoneda(g_m)\), applying $L$ provides us with a commutative square
\begin{center}
	\begin{tikzcd}
		 L(E) \arrow[r,"L(h)"] & L(F)  \\
		  \Yoneda\big(\varphi(E_m)\big)  \arrow[u,""]  \arrow[r,"\Yoneda(g_m)"] & \Yoneda\big(\varphi(F_m)\big) \arrow[u,""].
		\end{tikzcd}
\end{center}
Replacing \(L(h)\) with $g$ would also yield a commutative square, but $g$ is the unique map making the square commute by Lemma~\ref{L:BasicFactorisationProperty}. Thus \(g=L(h)=[\mathfrak{S}(\varphi)](h)\).
We conclude that \(\mathfrak{S}(\varphi)\) is full.
\end{proof}

\begin{proposition}
\label{P:FullContinuousFullyFaithful}
If \(\varphi\) is full and \(\Ker\varphi\leq\mathcal{M}\), then \(\mathfrak{S}(\varphi)\) is fully faithful.
\end{proposition}

\begin{proof}
The functor \(\mathfrak{S}(\varphi)\) is full by Lemma~\ref{L:FullnessIsPreserved}. By \cite[Corollary~4.4]{RingelZhang15}, every full triangulated functor is objective. Hence, Lemma~\ref{L:ObjectiveFaithful} also applies, so \(\mathfrak{S}(\varphi)\) is faithful.
\end{proof}

\begin{example}
The induced functor from Theorem~\ref{T:CompletionOfContinuousFunctor} does not have do be neither full, nor faithful.

Let \(\mathcal{S}:=\derived^b(\modf\dashmodule\Z)\), \(\mathcal{T}:=\derived^b\left(\modf\dashmodule\Z\left[\frac{1}{2}\right]\right)\), and use \(\varphi\) to denote \(\blank\otimes^{\mathbf{L}}_{\Z}\Z\left[\frac{1}{2}\right]:\mathcal{S}\rightarrow\mathcal{T}\).
Equip \(\mathcal{S}\) with the constant metric \(\mathcal{M}:=\langle\Z/3\Z\rangle\) and \(\mathcal{T}\) with the constant metric \(\mathcal{N}:=\left\langle\Z/3\Z\otimes_{\Z}\Z\left[\frac{1}{2}\right]\right\rangle\).
With this choice of metrics, the functor \(\varphi\) is a compression.

Since \(\varinjlim\Hom_{\mathcal{S}}(X,\mathbf{E})\otimes^{\mathbf{L}}_{\Z}\Z\left[\frac{1}{2}\right]\simeq\varinjlim\Hom_{\mathcal{T}}\left(X\otimes^{\mathbf{L}}_{\Z}\Z\left[\frac{1}{2}\right],\mathbf{E}\otimes^{\mathbf{L}}_{\Z}\Z\left[\frac{1}{2}\right]\right)\) for any \(X\in\mathcal{S}\) and any sequence \(\mathbf{E}\) in \(\mathcal{S}\), the functor \(\varphi\) has to send compactly supported Cauchy sequences to compactly supported Cauchy sequences and is then continuous.

We calculate \(\mathfrak{S}_{\mathcal{M}}(\mathcal{S})\simeq\derived^b\left(\modf\dashmodule\Z\left[\frac{1}{3}\right]\right)\) and
\(\mathfrak{S}_{\mathcal{N}}(\mathcal{T})\simeq\derived^b\left(\modf\dashmodule\Z\left[\frac{1}{6}\right]\right)\) using \cite[Theorem~6.4]{Matousek26B}. 
The~functor \(\mathfrak{S}(\varphi)\) coincides with
\(
\blank\otimes^{\mathbf{L}}_{\Z}\Z\left[\frac{1}{2}\right]:\derived^b\left(\modf\dashmodule\Z\left[\frac{1}{3}\right]\right)\rightarrow\derived^b\left(\modf\dashmodule\Z\left[\frac{1}{6}\right]\right)
.\)
But \(\mathfrak{S}(\varphi)\) is not faithful, since \(\id_{\Z/2\Z}\) is sent to zero.
And it is not full because \(\cdot\frac{1}{2}:\Z\left[\frac{1}{6}\right]\rightarrow\Z\left[\frac{1}{6}\right]\) has no preimage.
\end{example}

\section{Verdier quotients via continuous functors}
\label{sec:quotients}

In this section, we work with a commutative ring \(\kk\), an essentially small triangulated category \(\mathcal{S}\), a thick subcategory \(\mathcal{C}\subseteq\mathcal{S}\), and  the Verdier quotient functor \(\pi:\mathcal{S}\rightarrow\mathcal{S}/\mathcal{C}\).
Our aim is to use continuous functors to show that \(\mathfrak{S}_{\mathcal{C}}(\mathcal{S})\subseteq\mathcal{S}/\mathcal{C}\) (Theorem~\ref{T:EmbeddingIntoQuotient}) and give a sufficient condition for equality (Corollary~\ref{C:CountablyGeneratedCompletionVerdierQuotient}).
 
\begin{lemma}
\label{L:FunctorIntoQuotient}
There is a well-defined triangulated functor \(\Yoneda^{-1}\circ\mathfrak{S}(\pi):\mathfrak{S}_{\mathcal{C}}(\mathcal{S})\rightarrow \mathcal{S}/\mathcal{C}\) extending $\pi$ in a sense that \(\left[\Yoneda^{-1}\circ\mathfrak{S}(\pi)\right]\left(\moco \mathbf{E}\right)\simeq \pi(E_m)\) for every compactly supported Cauchy sequence \(\mathbf{E}=(E_n,e_n)_{n\in\N}\) stabilising at \(\mathcal{C}\) from \(N\in\N\) onwards and all \(m\geq N\).    
\end{lemma}

\begin{proof}
Because \(\pi:\mathcal{S}\rightarrow\mathcal{S}/\mathcal{C}\) annihilates \(\mathcal{C}\), there is a compression \((\pi,\mathcal{C},0)\). Since every functor from \(\Mod\dashmodule\mathcal{S}/\mathcal{C}\) is compactly supported with respect to $0$, the functor $\pi$ is even continuous as the~defining condition in Definition~\ref{D:Continuous} is now void.
We obtain a triangulated functor \(\mathfrak{S}(\pi):\mathfrak{S}_{\mathcal{C}}(\mathcal{S})\rightarrow\mathfrak{S}_0(\mathcal{S}/\mathcal{C})\) by invoking Theorem~\ref{T:CompletionOfContinuousFunctor}.
Let \(\mathbf{E}=(E_n,e_n)_{n\in\N}\) be a compactly supported Cauchy sequence in \(\mathcal{S}\), which stabilises at \(\mathcal{C}\) from some \(N\in\N\) onwards.
Then for all \(N\leq n\leq m\) it holds that \(\cone e_{n,m}\in\mathcal{C}\), so \(\pi(e_{n,m})\) is an isomorphism in \(\mathcal{S}/\mathcal{C}\). Consequently, the~object \(\mathfrak{S}(\pi)\left(\moco\mathbf{E}\right)\simeq\moco \pi(\mathbf{E})\in\mathfrak{S}_0(\mathcal{S}/\mathcal{C})\) is a module colimit of an eventually constant Cauchy sequence \(\pi(\mathbf{E})\) in \(\mathcal{S}/\mathcal{C}\) and is then isomorphic to \(\Yoneda\big(\pi(E_m)\big)\) for all \(m\geq N\).

By Example~\ref{E:ZeroMetric}, the completion \(\mathfrak{S}_0(\mathcal{S}/\mathcal{C})\) is identified with \(\mathcal{S}/\mathcal{C}\) via \(\Yoneda\). This finishes the proof.
\end{proof}

\begin{lemma}
\label{L:FullCompletionQuotient}
The functor \(\Yoneda^{-1}\circ\mathfrak{S}(\pi):\mathfrak{S}_{\mathcal{C}}(\mathcal{S})\rightarrow \mathcal{S}/\mathcal{C}\) from Lemma~\ref{L:FunctorIntoQuotient} is full.   
\end{lemma}

\begin{proof}
Let \(E:=\moco\mathbf{E},G:=\moco\mathbf{G}\in\mathfrak{S}_{\mathcal{C}}(\mathcal{S})\) for some Cauchy sequences \(\mathbf{E}=(E_n,e_n)_{n\in\N}\) and \(\mathbf{G}=(G_n,g_n)_{n\in\N}\) in \(\mathcal{S}\).
W.l.o.g.\ the two sequences stabilise at \(\mathcal{C}\) from $2$ onwards. Denote
\(e:\Yoneda(E_2)\rightarrow E\) and \(g:\Yoneda(G_2)\rightarrow G\) the~co\-lim\-it injections.
By Theorem~\ref{T:CompletionOfContinuousFunctor}, the functor \(\mathfrak{S}(\pi)\) is a restriction of the~left adjoint \(L:\Mod\dashmodule\mathcal{S}\rightarrow\Mod\dashmodule\mathcal{S}/\mathcal{C}\) of \(\blank\circ \pi\).
Let \(t:L(E)\rightarrow L(G)\) be a morphism in \(\mathfrak{S}_0(\mathcal{S}/\mathcal{C})\). By Lemma~\ref{L:FunctorIntoQuotient}, there is a~morphism \(u:\pi(E_2)\rightarrow\pi(G_2)\) in \(\mathcal{S}/\mathcal{C}\) and a commutative square 
\begin{center}
	\begin{tikzcd}
		 L(E) \arrow[r,"t"] & L(F)  \\
		  \Yoneda\big(\pi(E_2)\big) \arrow[u,"L(e)"]  \arrow[r,"\Yoneda(u)"] & \Yoneda\big(\pi(G_{2})\big) \arrow[u,"L(g)"],
		\end{tikzcd}
\end{center}
where the vertical maps are isomorphisms.
By the definition of morphisms in the Verdier quotient, the~map $u$ is given by a~roof \(E_2 \xleftarrow{v} F_1 \xrightarrow{w} G_{2}\)
in \(\mathcal{S}\) with \(\cone v\in\mathcal{C}\) and \(u=\pi(w)\pi(v)^{-1}\).
We define a new Cauchy sequence \(\mathbf{F}=(F_n,f_n)_{n\in\N}\) in \(\mathcal{S}\) by setting \(f_1:=v\), \(F_n:=E_n\), and \(f_n:=e_n\) for all \(n\geq2\). Note that \(\mathbf{F}\) stabilises at \(\mathcal{C}\) from \(1\) onwards and \(\moco\mathbf{F}=E\).

By Lemma~\ref{L:BasicFactorisationProperty}, there is a commutative square 
\begin{center}
	\begin{tikzcd}
		 E \arrow[r,"h"] & F  \\
		  \Yoneda(F_1)  \arrow[u,"e\circ\Yoneda(v)"]  \arrow[r,"\Yoneda(w)"] & \Yoneda(G_{2}) \arrow[u,"g"].
		\end{tikzcd}
\end{center}
By Lemma~\ref{L:FunctorIntoQuotient}, applying $L$ to the square gives us a commutative diagram 
\begin{center}
	\begin{tikzcd}
		 L(E) \arrow[r,"L(h)"] & L(F)  \\
		  \Yoneda\big(\pi(E_2)\big) \arrow[u,"L(e)"]  \arrow[r,"\Yoneda(u)"] & \Yoneda\big(\pi(G_{2})\big) \arrow[u,"L(g)"] \\
        		  \Yoneda\big(\pi(F_1)\big) \arrow[u,"\Yoneda(v)"]  \arrow[ru,"\Yoneda(w)"'] & .
		\end{tikzcd}
\end{center}
Because \(L(e)\) and \(L(g)\) are isomorphisms, we get that \([\mathfrak{S}(\pi)](h)=L(h)=t\). Hence \(\mathfrak{S}(\pi)\) is full and so is \(\Yoneda^{-1}\circ\mathfrak{S}(\pi)\).
\end{proof}

\begin{theorem}
\label{T:EmbeddingIntoQuotient}
The triangulated functor \(\Yoneda^{-1}\circ\mathfrak{S}(\pi):\mathfrak{S}_{\mathcal{C}}(\mathcal{S})\rightarrow \mathcal{S}/\mathcal{C}\) is fully faithful.
\end{theorem}

\begin{proof}
By Lemma~\ref{L:FunctorIntoQuotient}, the~functor \(\Yoneda^{-1}\circ\mathfrak{S}(\pi)\) is well-defined and triangulated. It is also full by Lemma~\ref{L:FullCompletionQuotient}. Since \(\pi\) is objective by \cite[Lemma~2.1.26]{Neeman01}, Lemma~\ref{L:ObjectiveFaithful} ensures that \(\Yoneda^{-1}\circ\mathfrak{S}(\pi)\) is faithful.
\end{proof}

The following is a generalisation of \cite[{Lemma~4.10}]{Matousek26B}.

\begin{corollary}
\label{C:CountablyGeneratedCompletionVerdierQuotient}
Assume that \(\mathcal{S}\) is \(\kk\)-linear and that for all \(X,Y\in\mathcal{S}\) the~\(\kk\)\=/module \(\Hom_{\mathcal{S}}(X,Y)\) is finitely generated.
If \(\mathcal{C}\) is countably generated, then \(\mathfrak{S}_{\mathcal{C}}(\mathcal{S})\simeq\mathcal{S}/\mathcal{C}\).
\end{corollary}

\begin{proof}
By Theorem~\ref{T:EmbeddingIntoQuotient}, the triangulated functor \(\Yoneda^{-1}\circ\mathfrak{S}(\pi):\mathfrak{S}_{\mathcal{C}}(\mathcal{S})\rightarrow \mathcal{S}/\mathcal{C}\) is fully faithful. We just need to show that \(\mathfrak{S}(\pi)\) is essentially surjective.

We choose a countable generating subset \(\mathcal{C'}\subseteq\mathcal{C}\) such that \(\Sigma\mathcal{C}'=\mathcal{C}'\).
Then we choose an enumeration \((C_n)_{n\in\N}\) of \(\mathcal{C'}\) such that every object \(C\in\mathcal{C}'\) is repeated infinitely many times, i.e.\ there exists an infinite subset \(I_m\subseteq\N\) satisfying \(C\simeq C_m\) for each \(m\in I\).

For any \(X,Y\in\mathcal{S}\) the finite generation condition on the \(\kk\)-module \(\Hom_{\mathcal{S}}(X,Y)\) implies that the object $Y$ has an \(\addf(X)\)-precover; for example \(\coprod X\rightarrow Y\) where the coproduct is taken over a finite generating subset of \(\Hom_{\mathcal{S}}(X,Y)\) and the map is the natural summation.

Pick an object of \(\mathcal{S}/\mathcal{C}\). It is necessarily of the form \(\pi(X)\) for some \(X\in\mathcal{S}\).
We define a sequence \(\mathbf{E}=(E_n,e_n)_{n\in\N}\) inductively. 
Firstly, we set \(E_1:=X\). Once \(E_n\) is defined for \(n\in\N\), we choose an~\(\addf(C_n)\)\=/pre\-cover \(C_n'\rightarrow E_n\) of \(E_n\) and complete it into a triangle
\[
C_n'\rightarrow E_n \xrightarrow{e_n} E_{n+1} \rightarrow \Sigma C_n'
.\]
This defines \(E_{n+1}\) and \(e_n\).

It is clear that \(\mathbf{E}\) is Cauchy with respect to \(\mathcal{C}\). Hence \(E\in\mathfrak{L}_{\mathcal{C}}(\mathcal{S})\).
Let \(C\in\mathcal{C}'\). There are infinitely many \(n\in\N\) such that the first map in the~distinguished triangle
\[
C_n'\rightarrow E_n\xrightarrow{e_n} E_{n+1}\rightarrow\Sigma C_n'
\]
is an \(\addf(C)\)-precover. By the definition of a precover, the first map in the exact sequence
\[
\Hom_{\mathcal{S}}\big(C,C_n'\big) \xrightarrow{} \Hom_{\mathcal{S}}(C,E_n) \xrightarrow{} \Hom_{\mathcal{S}}(C,E_{n+1}) 
\]
is an epimorphism. By exactness, the map \(\Hom_{\mathcal{S}}(C,e_n)\) is zero for infinitely many indices \(n\in\N\). Hence \(\varinjlim_{n\in\N}\Hom_{\mathcal{S}}(C,E_n)\simeq\moco\mathbf{E}(C)\simeq0\). Since \(\moco \mathbf{E}\) then vanishes on all shifts of generators of \(\mathcal{C}\), it is compactly supported at \(\mathcal{C}\). Thus \(E\in\mathfrak{S}_{\mathcal{C}}(\mathcal{S})\).
Because \(\mathbf{E}\) stabilises at \(\mathcal{C}\) from $1$ onwards, we can conclude by using Lemma~\ref{L:FunctorIntoQuotient} and observing 
\([\mathfrak{S}(\pi)]\left(\moco\mathbf{E}\right)\simeq\pi(E_1)\simeq\pi(X)\).
\end{proof}

\begin{corollary}
\label{C:CharterisingCoreflectivity}
Assume that \(\mathcal{S}\) is \(\kk\)-linear and that for all \(X,Y\in\mathcal{S}\) the~\(\kk\)\=/module \(\Hom_{\mathcal{S}}(X,Y)\) is finitely generated.
Suppose \(\mathcal{C}\) is countably generated.
Then \(\mathfrak{S}_{\mathcal{C}}(\mathcal{S})\simeq\Yoneda\left(\mathcal{C}^{\perp}\right)\subseteq\Mod\dashmodule\mathcal{S}\) if and only if \(\mathcal{C}\) is a coreflective subcategory of \(\mathcal{S}\).
\end{corollary}

\begin{proof}
``\(\Rightarrow\)'' 
By Corollary~\ref{C:CountablyGeneratedCompletionVerdierQuotient} and the assumptions, we have a~commutative square
\begin{center}
	\begin{tikzcd}
			\mathfrak{S}_{\mathcal{C}}(\mathcal{S}) \arrow[rr, "\Yoneda^{-1}\circ\mathfrak{S}(\pi)", "\simeq"'] & & \mathcal{S}/\mathcal{C}  \\
			\mathcal{C}^{\perp} \arrow[u, "\simeq", "\Yoneda"'] \arrow[rr, hookrightarrow, ""] & & \mathcal{S} \arrow[u, "\pi"'] 
		\end{tikzcd}
\end{center}
with the composition being an equivalence.
By \cite[Proposition~4.9.1]{Krause10B}, this is equivalent to \(\mathcal{C}\) being coreflective.

``\(\Leftarrow\)'' By Corollary~\ref{C:CoreflectiveGivesPerp}.
\end{proof}

\section{Discrete cluster categories and their completions}
\label{sec:DCC}

In this section, we briefly recapitulate the most important facts about discrete cluster categories (Subsection~\ref{SS:DiscreteClusterCategories}) and their various notions of completion (Subsection~\ref{SS:DiscreteClusterCompletions}). The main result (Theorem~\ref{T:TheEquivalenceIsAnEquivalence}) in Subsection~\ref{SS:Construction} is that the additive equivalence between the Paquette-Yıldırım completion and a specific instance of a~Cummings-Gratz metric completion established in \cite[Corollary~6.24]{CummingsGratz24} is a triangulated equivalence.

We shall work with a fixed \(p\in\N\) and a field $K$.

\subsection{Discrete cluster categories}
\label{SS:DiscreteClusterCategories}
Discrete cluster categories of type \(\A\) were introduced by Igusa and Todorov in \cite{IgusaTodorov15}. For the fixed \(p\) and $K$ \textit{the discrete cluster category} \(\mathcal{C}_p\) is a $K$-linear Hom-finite Krull-Schmidt 2-Calabi-Yau triangulated category, whose structure can be understood through the following combinatorial model.

We consider the circle \(S^1\subseteq\R^2\) with the usual topology and impose a cyclic order on its points by the rule \(x\leq y\leq z\) if and only if moving from \(x\) to \(z\) in the anti-clockwise direction passes through \(y\). 
We fix \textit{accumulation points} \(A_1<A_2<\cdots<A_p\) on \(S^1\). Set \(A_{p+1}:=A_1\).
For each \(1\leq i\leq p\),
we pick countably many distinct points \(\Z_i:=\{n_i\}_{i\in\Z}\) on \(S^1\) ordered by \(\Z\) between \(A_i\) and \(A_{i+1}\) in such a way, that the~ordering on \(\Z_i\) is compatible with the cyclic order on \(S^1\), \(\lim_{n\rightarrow-\infty}n_i=A_{i}\), and \(\lim_{n\rightarrow\infty}n_i=A_{i+1}\).

\textit{An arc} is a two-element subset of \(\mathcal{Z}:=\cup_{i=1}^p\Z_i\) not-containing neighbouring points; or more formally, it is an element of
\(
\left\{X\in2^\mathcal{Z}:|X|=2\right\}\setminus\bigcup_{n\in\Z}\bigcup_{i=1}^p\big\{\{n_i,(n+1)_{i}\}\big\}
\). The two elements of the~arc are called its \textit{endpoints}.
Two arcs \(X,Y\) \textit{cross} if we can name their endpoints in such a way, that \(X=\{x,x'\}\), \(Y=\{y,y'\}\), and \(x<y<x'<y'\). For \(i,j\in\{1,\ldots,p\}\), \(x,y\in\Z\), and an arc \(X=\{x_i,y_j\}\) we define \(\Sigma X:=\{(x-1)_i,(y-1)_j\}\).

Given this data, we can describe the triangulated category \(\mathcal{C}_p\). 
The indecomposable objects of \(\mathcal{C}_p\) are in bijection with arcs.
The operation \(X\mapsto\Sigma X\) or arcs extends to the shift functor on \(\mathcal{C}_p\). 
For arcs \(X\) and \(Y\) we have
\[
\Hom_{\mathcal{C}_p}(X,Y)\simeq
\begin{cases}
K & \text{if \(\Sigma X\) and $Y$ cross},\\
0 & \text{otherwise}.
\end{cases}
\]
If we want to complete a morphism \(0\neq f\in\Hom_{\mathcal{C}_p}(X,Y)\) between two arcs to a triangle, we first take a look at the crossing between \(\Sigma X=\{x,x'\}\) and \(Y=\{y,y'\}\) enabled by the inequality \(x<y<x'<y'\), and then obtain a distinguished triangle (Figure~\ref{F:Triangle})
\[
X\xrightarrow{f}Y\rightarrow\{x,y\}\oplus\{x',y'\}\rightarrow\Sigma X
\]
where we identify any potential ``invalid arc'' connecting neighbouring endpoints with the zero object.

\begin{figure}[!h]
\begin{center}
\centering
\includegraphics[width=0.4\linewidth]{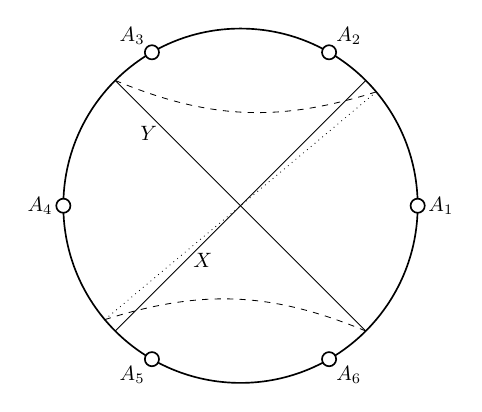}
\end{center}
\caption{A morphism between the arcs \(X\) and \(Y\) in \(\mathcal{C}_6\) exists because \(\Sigma X\) (the dotted arc) and \(Y\) cross. The cone of the morphism is the direct sum of the dashed arcs.}\label{F:Triangle}
\end{figure}

The triangulated category \(\mathcal{C}_p\) is uniquely determined up to triangulated equivalence by the field \(K\) and the number of accumulation points \(p\).

\subsection{Combinatorial and metric completions of discrete cluster categories}
\label{SS:DiscreteClusterCompletions}

\subsubsection{Paquette-Yıldırım combinatorial completion}

The purpose of the combinatorial completion \(\overline{\mathcal{C}_p}\) of \(\mathcal{C}_p\) is to add new arcs with endpoints in accumulation points while retaining the triangulated structure.
The~construction by Paquette and Yıldırım in \cite{PaquetteYildirim21} (see also \cite{Franchini25}) goes as follows.
We consider the~discrete cluster category  \(\mathcal{C}_{2p}\) with twice as many accumulation points \(A_{1'}<A_1<A_{2'}<A_2<\cdots<A_{p'}<A_p\).
Define the thick subcategories \(\mathcal{D}:=\big\{\{x,y\}\in\mathcal{C}_{2p}:x,y\in\Z_{i'},1\leq i\leq p\big\}\subseteq\mathcal{C}_{2p}\) and \(\mathcal{D}':=\mathcal{D}^{\perp}\).
It holds that \(\mathcal{D}'={}^{\perp}\mathcal{D}\) and \(\mathcal{D}'\simeq\mathcal{C}_p\), so we may identify \(\mathcal{C}_p\) as the thick subcategory \(\mathcal{D}'\).
\textit{The Paquette-Yıldırım completion} is then defined as the Verdier quotient \(\overline{\mathcal{C}_p}:=\mathcal{C}_p/\mathcal{D}\).

As \(\mathcal{C}_p=\mathcal{D}^{\perp}={}^\perp\mathcal{D}\), the composition \(\mathcal{C}_p\xhookrightarrow{}\mathcal{C}_{2p}\rightarrow\overline{\mathcal{C}_p}\) is fully faithful. What is more, the combinatorial model of \(\mathcal{C}_p\) using arcs on \(S^1\) with $p$ accumulation points extends to a model for \(\overline{\mathcal{C}_p}\).

\begin{theorem}[{\cite[Corollary~3.11, Proposition~3.14]{PaquetteYildirim21}}]
\label{T:PaquetteYildirimTheorem}
The triangulated category \(\overline{\mathcal{C}_p}\) is Hom-finite and Krull-Schmidt, whose indecomposable objects are arcs of the form \(\{x_i,y_j\}\in\mathcal{C}_p\), \(\{x_i,A_j\}\), or \(\{A_i,A_k\}\) for some \(x,y\in\Z\) and \(i,j,k\in\{1,\ldots,p\}\), \(i\neq k\).

The~shift functor acts on the newly added arcs as \(\Sigma\{x_i,A_j\}=\{(x-1)_i,A_j\}\) and \(\Sigma\{A_i,A_k\}=\{A_i,A_k\}\).

And for indecomposable objects \(X,Y\in\overline{\mathcal{C}_p}\) it holds that \(\Hom_{\overline{\mathcal{C}_p}}(X,Y)\simeq0\) or \(\Hom_{\overline{\mathcal{C}_p}}(X,Y)\simeq K\); the~latter happens if and only if at least one of the following conditions holds:
\begin{enumerate}[label=(\roman*)]
\item \(X=Y\),
\item \(\Sigma X\) and \(Y\) cross, 
\item \(\Sigma X=\{x,A_i\}\), \(Y=\{y,A_i\}\) for some \(1\leq i\leq p\) and \(x,y\in\mathcal{Z}\cup\{A_1,\ldots,A_{i-1}\}\cup\{A_{i+1},\ldots,A_p\}\) and the inequality \(y<A_i<x\) holds.
\end{enumerate}
\end{theorem}

An arc \(\{x,y\}\in\mathcal{C}_{2p}\) with \(x\in\Z_i\) and \(y\in\Z_{j'}\) for \(i,j\in\{1,\ldots,p\}\) is mapped to the arc \(\{x,A_{j}\}\in\overline{\mathcal{C}_{p}}\) by the~quotient functor \(\mathcal{C}_{2p}\rightarrow\overline{\mathcal{C}_{p}}\); and similarly \(\{y,z\}\mapsto\{A_{i},A_{j}\}\) if additionally \(z\in\Z_{i'}\).

\begin{example}
\label{E:TrianglesInCompletion}
The rules for determining distinguished triangles in \(\overline{\mathcal{C}_p}\) with two indecomposable terms are rather similar to those in \(\mathcal{C}_p\). For crossing arcs, the rule is formally the same.

We give an example with arcs with one endpoint being an accumulation point \(A\).
Let us have \(X,Y\in\overline{\mathcal{C}_{p}}\) indecomposable. Suppose that \(X=\{x_i,A_j\}, Y=\{y,A_j\}\) for some \(i,j\in\{1,\ldots,p\}\), \(x\in\Z\), and \(y\in\mathcal{Z}\).
If \(0\neq f:X\rightarrow Y\) is a morphism enabled by the inequality \(y<A_j<x_i\), then we get a triangle
\[
X\xrightarrow{f} Y \rightarrow \{(x-1)_i,y\} \rightarrow \Sigma X
.\]
\end{example}

\begin{lemma}
\label{L:PaquetteYildirimGeneration}
The arcs of the form \(\{x,A\}\) for some \(x\in\mathcal{Z}\) and \(A\) an accumulation point generate \(\overline{\mathcal{C}_p}\) as a thick subcategory. 
\end{lemma}

\begin{proof}
Every arc with endpoints in \(\mathcal{Z}\) is generated by arcs with precisely one accumulation endpoint; this follows from Example~\ref{E:TrianglesInCompletion}.
Let \(1\leq i<j\leq p\). Choose \(x\in\mathcal{Z}\) with \(x<A_i<A_j\). Then for \(X:=\{x,A_i\}\) and \(Y:=\{x,A_j\}\) the arcs \(\Sigma X\) and \(Y\) cross, so there is a triangle 
\[
X\rightarrow Y\rightarrow\{A_i,A_j\}\rightarrow\Sigma X
.\]
Therefore, even the arcs with two accumulation endpoints are generated by the arcs with one accumulation endpoint.
\end{proof}

\subsubsection{Cummings-Gratz metric completions}

In \cite{CummingsGratz24}, Cummings and Gratz fully classify metric completions of \(\mathcal{C}_p\) with respect to aisle and coaisle metrics for any possible choice of a $t$-structure on \(\mathcal{C}_p\). 
They prove that completions with respect to aisle metrics are always thick subcategories for, in fact, any triangulated category (see Theorem~\ref{T:CummingsGratzAisleTheorem}). The completions of \(\mathcal{C}_p\) with respect to coaisle metrics are classified via \cite[Theorem~6.22]{CummingsGratz24} and for each such a completion a combinatorial model is given.

We will now focus on a specific choice of a coaisle metric. By \cite[Corollary~4.30]{GratzZvonareva23}, there exists a~unique non-degenerate bounded above $t$\=/structure \(\left(\mathcal{C}_p^{\leq0},\mathcal{C}_p^{>0}\right)\) up to equivalence (see Subsection~\ref{SS:Compressions}). For simplicity, we choose a representative this $t$-structure with
\[
\mathcal{C}_p^{>0}=\addf\big(\{n_i,m_i\}\in\mathcal{C}_p:n,m\in\N\cup\{0,-1\},1\leq i\leq p\big)
\]
but remind that the resulting completion does not depend on equivalences of neither metrics, nor $t$\=/structures. Denote \(\mathcal{M}=\{B_1\}_{n\in\N}=\big\{\mathcal{C}_p^{> n}\big\}_{n\in\N}\) the corresponding coaisle $t$-structure. Spelling out the description \cite[Theorem~6.22]{CummingsGratz24} of the combinatorial model using arcs in \(S^1\) for this specific metric completion \(\mathfrak{S}_{\mathcal{M}}(\mathcal{C}_p)\) gives us:
\begin{theorem}
The triangulated category \(\mathfrak{S}_{\mathcal{M}}(\mathcal{C}_p)\) is Hom-finite and Krull-Schmidt, whose indecomposable objects are arcs of the form \(\{x_i,y_j\}\in\mathcal{C}_p\), \(\{x_i,A_j\}\), or \(\{A_i,A_k\}\) for some \(i,j,k\in\{1,\ldots,p\}\), \(i\neq k\), and \(x,y\in\Z\).
The description of homomorphisms between indecomposables and of the action of \(\Sigma\) on indecomposables is identical to that in Theorem~\ref{T:PaquetteYildirimTheorem}.   
\end{theorem}

Every indecomposable in \(\mathfrak{S}_{\mathcal{M}}(\mathcal{C}_p)\) can be expressed as module colimit of a Cauchy sequence called a~\textit{double fan} (see \cite[Definition~5.17, Remark~5.18]{CummingsGratz24}). For \(X\in\mathcal{C}_p\subseteq\mathfrak{S}_{\mathcal{M}}(\mathcal{C}_p)\), the sequence can be taken constant \(X\xrightarrow{\id_X}X\xrightarrow{\id_X}\cdots\). If \(X=\{x_i,A_{j+1}\}\) for \(x\in\Z\) and \(i,j\in\{1,\ldots,p\}\), then $X$ corresponds to \(\moco_{n\in\N}\big\{x_i,(x+n+1)_j\big\}\). And if \(X=\{A_{i+1},A_{j+1}\}\) for \(1\leq i< j\leq p\), then one obtains \(X\) as \(\moco_{n\in\N}\big\{n_i,n_j\big\}\). See Figure~\ref{F:Two}.

\begin{figure}[!ht]
\begin{center}
\centering
\includegraphics[width=1\linewidth]{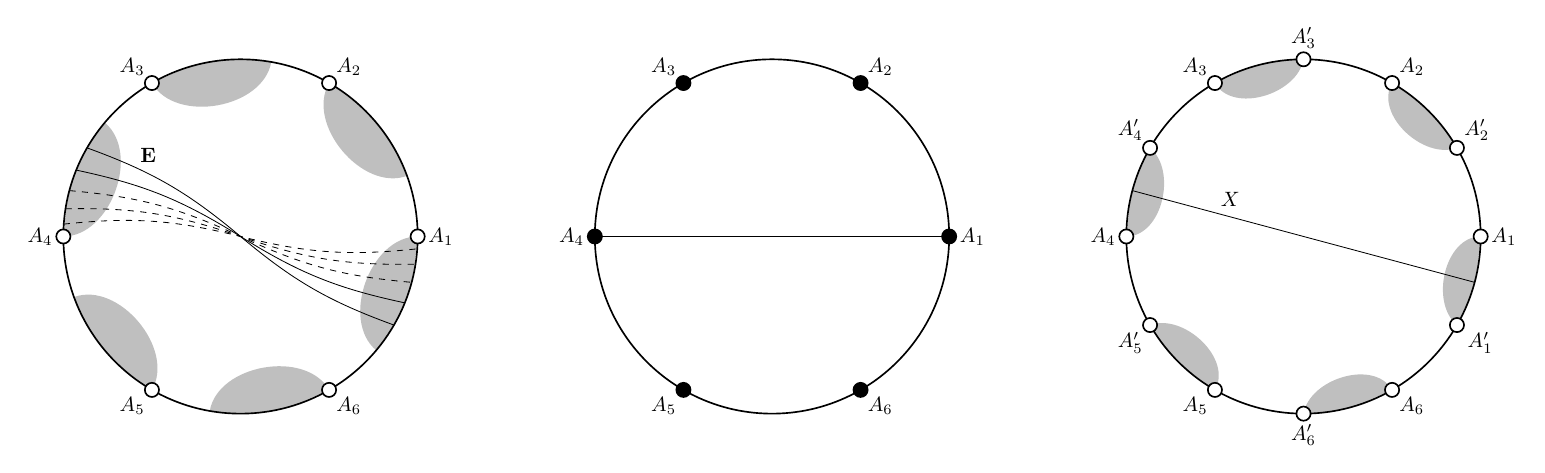}
\end{center}
\caption{The coaisle \(\mathcal{C}^{>0}_6\) (grey area) and a double fan \(\mathbf{E}\) with respect respect to the~co\-aisle metric \(\mathcal{M}\) (left). The subcategory \(\mathcal{D}\subseteq\mathcal{C}_{12}\) (grey area) and \(X\in\mathcal{C}_{12}\) (right). Both \(\mathbf{E}\) and \(X\) correspond to the infinite arc \(\{A_1,A_4\}\) (middle) in \(\mathfrak{S}_{\mathcal{M}}(\mathcal{C}_6)\), or \(\overline{\mathcal{C}_6}\), resp.}\label{F:Two}
\end{figure}

The rules for computing triangles in  \(\mathfrak{S}_{\mathcal{M}}(\mathcal{C}_p)\) are also the same as in \(\overline{\mathcal{C}_p}\), so even an analogue of Example~\ref{E:TrianglesInCompletion} is true, and the following lemma holds:
\begin{lemma}
\label{L:CummingsGratzGeneration}
The arcs of the form \(\{x,A\}\) for some \(x\in\mathcal{Z}\) and \(A\) an accumulation point generate \(\mathfrak{S}_{\mathcal{M}}(\mathcal{C}_p)\) as a thick subcategory. 
\end{lemma}

All the above tells us (as originally proved in \cite[Corollary~6.24, Remark~6.25]{CummingsGratz24}), that there exists an~additive equivalence \(\mathfrak{S}_{\mathcal{M}}(\mathcal{C}_p)\xrightarrow{\simeq}\overline{\mathcal{C}_p}\) sending every arc \(\{x,y\}\in\mathfrak{S}_{\mathcal{M}}(\mathcal{C}_p)\) to ``the very same arc'' \(\{x,y\}\in\overline{\mathcal{C}_p}\), commuting with \(\Sigma\) and preserving triangles with two indecomposable terms. However, it is \textit{a priori} not apparent that this additive equivalence gives rise to a triangulated functor.

The rest of this section is dedicated to proving that it does.

\subsection{Construction of the equivalence}
\label{SS:Construction}

\begin{notation}
We shall still use the notation from Subsection~\ref{SS:DiscreteClusterCompletions}, namely the subcategories \(\mathcal{D},\mathcal{D}'\subseteq\mathcal{C}_{2p}\) and the coaisle metric \(\mathcal{M}=\{B_n\}_{n\in\N}\).
Let \(\nu:\mathcal{C}_p\xhookrightarrow{}\mathcal{C}_{2p}\) be the fully faithful triangulated functor exhibiting \(\mathcal{C}_p\) as \(\mathcal{D}'\). We shall perceive \(\nu\) as the inclusion.
The family \(\mathcal{N}:=\big\{\nu(B_n)\big\}_{n\in\N}\) of subcategories of \(\mathcal{D}'\subseteq\mathcal{C}_{2p}\) is a good metric on \(\mathcal{C}_{2p}\) by \cite[Lemma~4.14]{Matousek26C}.
Moreover, the subcategory \(\nu(B_1)\) is a coaisle of a $t$-structure on \(\mathcal{C}_{2p}\) by \cite[Lemma~5.5.1]{Franchini25}.   
\end{notation}

We want to construct a triangulated equivalence \(\mathfrak{S}_{\mathcal{M}}(\mathcal{C}_p)\xrightarrow{\simeq}\overline{\mathcal{C}_p}\). The category \(\overline{\mathcal{C}_p}\) is a Verdier quotient, so it is also a metric completion \(\mathfrak{S}_{\mathcal{D}}(\mathcal{C}_{2p})\) by Corollary~\ref{C:CountablyGeneratedCompletionVerdierQuotient}. So all we need to do is to compare two metric completions.
Our strategy is to construct an auxiliary triangulated category \(\mathcal{T}\) into which both \(\mathfrak{S}_{\mathcal{M}}(\mathcal{C}_p)\) and \(\overline{\mathcal{C}_p}\) embed via continuous functors.
Then we show that there exists a TTF-triple \(\big(\mathfrak{S}_{\mathcal{M}}(\mathcal{C}_p),\mathcal{G},\overline{\mathcal{C}_p}\big)\) in~\(\mathcal{T}\). The natural equivalence between the outer terms of the TTF-triple will then be our desired triangulated functor.

\begin{lemma}
\label{L:DimensionCounting}
Let us have Cauchy sequences \(\mathbf{E}=(E_n,e_n)_{n\in\N}\) and \(\mathbf{F}=(F_n,f_n)_{n\in\N}\) with respect to a~good metric on \(\mathcal{C}_p\). Set \(E:=\moco\mathbf{E}\) and \(F:=\moco\mathbf{F}\).
Suppose that \(F\) is compactly supported at \(s\in\N\) and \(\mathbf{E}\) stabilises at \(s+1\) from \(N\in\N\) onwards.
Let \(k\) be the number of indecomposable direct summands of \(E_N\). 
If there exists \(l\in\N\) such that for all \(n\in\N\) the object \(F_n\) has at most $l$ indecomposable direct summands, then \(\dim_K\Hom_{\Mod\dashmodule\mathcal{C}_p}(E,F)\leq k\cdot l\).
\end{lemma}

\begin{proof}
By Lemma~\ref{L:BasicFactorisationProperty}, we have
\[
\Hom_{\Mod\dashmodule\mathcal{C}_p}(E,F)\simeq\Hom_{\Mod\dashmodule\mathcal{C}_p}\big(\Yoneda(E_N),F\big)\simeq\varinjlim_{n\in\N}\Hom_{\mathcal{C}_p}(E_N,F_n)
.\]
Assume for all \(n\in\N\) that the number of direct summands of \(F_n\) is bounded by $l$.
Since the Hom-space between two indecomposable objects can only be zero or one-dimensional depending on whether the~corresponding arcs cross or not,  for all \(n\in\N\) it holds that that \(\dim_K\Hom_{\mathcal{C}_p}(E_N,F_n)\leq k\cdot l\). Hence the claim follows.
\end{proof}

Below, we need the notion of a join of two metrics from Subsection~\ref{SS:Compressions} and Definition~\ref{D:Continuous}.

\begin{lemma}
There is a continuous compression \(\big(\nu,\mathcal{M},\mathcal{N}\vee\mathcal{D}\big)\). 
\end{lemma}

\begin{proof}
It is obvious that \(\big(\nu,\mathcal{M},\mathcal{N}\vee\mathcal{D}\big)\) is a compression. Let \(\mathbf{E}=(E_n,e_n)_{n\in\N}\) be a compactly supported Cauchy sequence in \(\mathcal{C}_p\). Then \(\nu(\mathbf{E})\) is clearly compactly supported with respect to \(\mathcal{N}\). Given that \(\nu(E_n)\in\mathcal{D}^{\perp}\) for all \(n\in\N\), the sequence \(\mathbf{E}\) is also compactly supported at \(\mathcal{D}\).
The claim now follows from \(\mathfrak{C}_{\mathcal{N}\vee\mathcal{D}}(\mathcal{C}_{2p})=\mathfrak{C}_{\mathcal{N}}(\mathcal{C}_{2p})\cap\mathfrak{C}_{\mathcal{D}}(\mathcal{C}_{2p})\) (see \cite[Lemma~5.3]{Matousek26B}).
\end{proof}

\begin{lemma}
There is a continuous compression \(\big(\Id_{\mathcal{C}_{2p}},\mathcal{D},\mathcal{N}\vee\mathcal{D}\big)\). 
\end{lemma}

\begin{proof}
It is obvious that \(\big(\Id_{\mathcal{C}_{2p}},\mathcal{D},\mathcal{N}\vee\mathcal{D}\big)\) is a compression. Let \(\mathbf{E}=(E_n,e_n)_{n\in\N}\) be a~compactly supported Cauchy sequence in \(\mathcal{C}_{2p}\) with respect to \(\mathcal{D}\). 
Suppose \(\mathbf{E}\) stabilises at \(\mathcal{D}\) from \(N\in\N\) onwards. 
The subcategory \(\nu(B_1)\subseteq\mathcal{C}_{2p}\) is a coaisle of a bounded above $t$-structure on \(\mathcal{C}_{2p}\), so we may find \(m\in\N\) such that \(E_N\in{}^{\perp}\Sigma^{-m}\nu(B_1)={}^{\perp}\nu(B_m)\). By the 2-Calabi-Yau property of \(\mathcal{C}_{2p}\), we get
\[
E_N\in\Sigma^{-2}\nu(B_m)^\perp=\nu(B_{m+2})^\perp
.\]
Let \(n\geq N\) and consider the distinguished triangle 
\[
E_N\rightarrow E_n\rightarrow D\rightarrow\Sigma E_N
\]
with \(D\in\mathcal{D}\subseteq\nu(B_{m+2})^\perp\). For all \(X\in B_{m+2}\) the exactness of 
\[
\Hom_{\mathcal{C}_{2p}}(X,E_N)\rightarrow \Hom_{\mathcal{C}_{2p}}(X,E_n)\rightarrow \Hom_{\mathcal{C}_{2p}}(X,D)
\]
yields \(\Hom_{\mathcal{C}_{2p}}(X,E_n)\simeq0\). In other words, for all \(n\geq N\) the object \(E_n\) is compactly supported at \(m+2\), and so is \(\moco\mathbf{E}\).
\end{proof}

Now we are in a situation where we have two fully faithful continuous functors into \(\mathcal{C}_{2p}\). By Proposition~\ref{P:FullContinuousFullyFaithful}, the completions \(\mathfrak{S}_{\mathcal{M}}(\mathcal{C}_p)\) and \(\mathfrak{S}_{\mathcal{D}}(\mathcal{C}_{2p})\) embed into \(\mathfrak{S}_{\mathcal{N}\vee\mathcal{D}}(\mathcal{C}_{2p})\). We know by Corollary~\ref{C:CountablyGeneratedCompletionVerdierQuotient}, that \(\mathfrak{S}_{\mathcal{D}}\big(\mathcal{C}_{2p}\big)\simeq\overline{\mathcal{C}_p}\).

To simplify the notation, we will from now on refer to  \(\mathfrak{S}_{\mathcal{N}\vee\mathcal{D}}(\mathcal{C}_{2p})\) as \(\mathcal{T}\). We also refer to \(\mathfrak{S}_{\mathcal{M}}(\mathcal{C}_p)\) as \(\mathcal{S}\) and identify it with the corresponding thick subcategory of  \(\mathcal{T}\). Lastly, we shall refer to \(\mathfrak{S}_{\mathcal{D}}(\mathcal{C}_{2p})\) as \(\overline{\mathcal{C}_p}\) and again identify it with the subcategory of \(\mathcal{T}\). On the side-note, it even holds \(\mathcal{C}_p\subseteq\mathcal{S}\cap\overline{\mathcal{C}_p}\subseteq\mathcal{T}\).

\subsubsection{Building the TTF triple}
We need one more thick subcategory \(\mathcal{G}\subseteq\mathcal{T}\) to obtain the triple \(\big(\mathcal{S},\mathcal{G},\overline{\mathcal{C}_p}\big)\).

For each \(1\leq i\leq p\) we define an object \(G^i\in\mathcal{T}\) in the following way. Under the convention \(\Z_{p+1}:=\Z_{1'}\), set \(G^i_n:=\left\{n_i,n_{(i+1)'}\right\}\in\mathcal{C}_{2p}\) for all \(n\in\N\) and notice that \(G^i_{n+1}=\Sigma^{-1} G^i_n\). 
We assemble these objects into a sequence \(\mathbf{G}^i:=\big(G^i_n,g^i_n\big)_{n\in\N}\) where the morphisms give rise to distinguished triangles 
\[
G^i_n\xrightarrow{g^i_n} G^i_{n+1} \rightarrow \{(n-1)_i,(n+1)_i\}\oplus \left\{(n-1)_{(i+1)'},(n+1)_{(i+1)'}\right\}\rightarrow \Sigma G^i_n
\]
with \(\cone g^i_n\in \nu(B_{n})\coprod\mathcal{D}\). Hence \(\mathbf{G}^i\) is Cauchy.

Put \(G^i:=\moco \mathbf{G}^i\in\mathfrak{L}_{\mathcal{N}\vee\mathcal{D}}\big(\mathcal{C}_{2p}\big)\).
We need to verify that \(G^i\) is compactly supported.
Let \(X\in \nu(B_1)\) be an indecomposable arc. If the endpoints of $X$ do not lie in \(\Z_i\), then clearly \(G^i(X)\simeq0\).
If the endpoints \(x,y\) of \(X\) lie in \(\Z_i\), we find \(N\in\N\) such that \(x<y<N_i\) or \(y<x<N_i\). Then \(\Hom_{\mathcal{C}_{2p}}(X,G_n)\simeq0\) for all \(n\geq N\) and again \(G^i(X)\simeq0\).
Hence \(G^i\) is compactly supported at \(\nu(B_1)\).
An~analogous argument shows that \(G^i\) is compactly supported at \(\mathcal{D}\). Hence \(G^i\in\mathfrak{C}_{\mathcal{N}\vee\mathcal{D}}(\mathcal{C}_{2p})\), so \(G^i\in\mathcal{T}\).

We observe that \(\Sigma G^i\simeq\moco \Sigma\mathbf{G}^i\simeq\moco\mathbf{G}^i\simeq G^i\).
Finally, we define a thick subcategory \(\mathcal{G}:=\langle G^1,\ldots,G^p\rangle\subseteq \mathcal{T}\).

\begin{lemma}
\label{L:DzisOrthogonal}
Let \(i,j\in\{1,\ldots,p\}, i\neq j\). Then \(\Hom_{\mathcal{T}}\left(G^i,G^j\right)\simeq0\).
\end{lemma}

\begin{proof}
We have
\[
\Hom_{\mathcal{T}}\left(G^i,G^j\right)\simeq\Hom_{\mathcal{T}}\left(\varinjlim_{n\in\N}\Yoneda\big(G^i_n\big),\varinjlim_{m\in\N}\Yoneda\big(G^j_m\big)\right)\simeq\varprojlim_{n\in\N}\varinjlim_{m\in\N}\Hom_{\mathcal{C}_{2p}}\left(G^i_n,G^j_m\right)
.\]
But for no combination of \(n,m\in\N\) the arcs \(G^i_n\) and \(G^j_n\) cross. Hence \(\Hom_{\mathcal{T}}\left(G^i,G^j\right)\simeq0\).
\end{proof}

\begin{lemma}
\label{L:DzisIndecomposable}
Let \(1\leq i\leq p\).
Then \(\End_{\mathcal{T}}\left(G^i\right)\simeq K\).
\end{lemma}

\begin{proof}
By Lemma~\ref{L:DimensionCounting}, the dimension of \(\End_{\mathcal{T}}(G^i)\) is either $1$, or $0$.
If \(G^i\) was the zero object, then by \cite[Lemma~4.19]{Matousek26C} there would exist \(n\in\N\) such that \(G^i_n\in\nu(B_1)\coprod\mathcal{D}\). Since this cannot happen, the~endomorphism ring of \(G^i\) has to be $K$ itself.
\end{proof}

\begin{lemma}
\label{L:AdditiveGenerators}
\(\mathcal{G}=\addf(G^1,\ldots,G^p)\)
\end{lemma}

\begin{proof}
From Lemma~\ref{L:DzisOrthogonal}, we infer \(\mathcal{G}\simeq\coprod_{i=1}^p\left\langle G^i\right\rangle\).
Let \(1\leq i\leq p\).
Since \(\Sigma G^i=G^i\), any triangle extension of elements of \(\addf G^i\) is given by a morphism \(f\) in \(\addf G^i\).
By \cite[Proposition~4.2]{Shah23} and Lemma~\ref{L:DzisIndecomposable}, there is an additive equivalence \(\Hom_{\mathcal{G}}(G^i,\blank):\addf G^i\xrightarrow{\simeq}\modf\dashmodule K\), so any such $f$ is a direct sum of zeros and isomorphisms.
We conclude that \(\addf G^i=\left\langle G^i\right\rangle\) and \(\mathcal{G}=\addf(G^1,\ldots,G^p)\). 
\end{proof}

\begin{lemma}
\label{L:CummingsGratzAisle}
 \(\mathcal{S}={}^\perp\mathcal{G}\)
\end{lemma}

\begin{proof}
``\(\subseteq\)'' Let \(E\in\mathcal{S}\subseteq\mathcal{T}\). Then \(E\simeq\moco\mathbf{E}\) for sequence \(\mathbf{E}=(E_n,e_n)_{n\in\N}\) in \(\mathcal{C}_p\subseteq\mathcal{C}_{2p}\) which is Cauchy with respect to \(\mathcal{N}\).
Pick \(1\leq i\leq p\).
By Lemma~\ref{L:BasicFactorisationProperty}, there is an isomorphism 
\[
\Hom_{\mathcal{T}}\left(E,G^i\right)\simeq\Hom_{\Mod\dashmodule\mathcal{C}_{2p}}\left(\Yoneda(E_N),G^i\right)\simeq\varinjlim_{m\in\N}\Hom_{\mathcal{C}_{2p}}\left(E_N,G^i_m\right)
\]
for some \(N\in\N\).
Since \(\nu(B_1)\) is a coaisle of a bounded above $t$-structure on \(\mathcal{C}_{2p}\), we may find \(M\in\N\) such that \(E_N\in{}^{\perp}\Sigma^{-M}\nu(B_1)={}^{\perp}\nu(B_M)\).
Let \(X\) be an indecomposable direct summand of \(E_N\).
Then the arc corresponding to \(X\) cannot have an endpoint in the set \(\{z_i:z>M\}\). And since we embed \(\mathcal{C}_p\) into \(\mathcal{C}_{2p}\) as~\(\mathcal{D}'\), the arc of \(X\) cannot have an endpoint in \(\Z_{(i+1)'}\) either.
Thus for all \(m\geq M\) the arcs of \(\Sigma X\) and \(G^i_m\) do not cross. Hence \(\Hom_{\mathcal{C}_{2p}}\left(E_N,G^i_m\right)\simeq0\) for all \(m\geq M\), so \(\Hom_{\mathcal{T}}\left(E,G^i\right)\simeq0\). Because $i$ was arbitrary and the objects of the~form \(G^i\) generate \(\mathcal{G}\) by Lemma~\ref{L:AdditiveGenerators}, we get \(E\in{}^{\perp}\mathcal{G}\).

``\(\supseteq\)''
Let \(E\in{}^{\perp}\mathcal{G}\subseteq\mathcal{T}\). Then \(E\simeq\moco\mathbf{E}\) for some Cauchy sequence \(\mathbf{E}=(E_n,e_n)_{n\in\N}\) in \(\mathcal{C}_{2p}\) with respect to \(\mathcal{N}\vee\mathcal{D}\). 
By Lemma~\ref{L:CummingsGratzTrivialisation}, we can further assume that \(\mathbf{E}\) is \(\mathcal{D}\)-trivial.

Assume for contradiction that there exist infinitely many indices \(n\in\N\) such that \(E_n\notin\mathcal{D}'\).
We can apply Lemma~\ref{L:BasicFactorisationProperty} to obtain an isomorphism 
\[
\Hom_{\mathcal{T}}\left(E,\coprod_{i=1}^pG^i\right)\simeq\Hom_{\Mod\dashmodule\mathcal{C}_{2p}}\left(\Yoneda(E_N),\coprod_{i=1}^pG^i\right)\simeq\coprod_{i=1}^p\varinjlim_{m\in\N}\Hom_{\mathcal{C}_{2p}}\left(E_N,G^i_m\right)
\]
for some \(N\in\N\), and we can even choose \(N\) large enough for \(E_N\) to contain an~indecomposable direct summand \(X\in\mathcal{C}_{2p}\setminus\mathcal{D}'\).
Then the arc corresponding to \(X\) has an endpoint of the form \(x_{(i+1)'}\) for some \(x\in\Z\) and \(1\leq i\leq p\). By the~\(\mathcal{D}\)\=/triviality, the other endpoint of \(X\) cannot lie in \(\Z_{(i+1)'}\), so there exist
\[
y\in\Z\text{\quad\quad\quad and\quad\quad\quad} k\in\left\{j,j':1\leq j\leq p\right\}\setminus\left\{(i+1)'\right\}
\] such that \(X=\big\{x_{(i+1)'},y_k\big\}\).
Let \(m\geq\max\{1,x,y\}\). 
Consider the distinguished triangle
\[
\begin{split}
\big\{(m+1)_i,(m+3)_i\big\}&\oplus \big\{(m+1)_{(i+1)'},(m+3)_{(i+1)'}\big\}\rightarrow G^i_{m+1} \xrightarrow{g^i_{m+1}} G^i_{m+2}\\
&\rightarrow\big\{m_i,(m+2)_i\big\}\oplus \big\{m_{(i+1)'},(m+2)_{(i+1)'}\big\}.
\end{split}
\]
Since \(\Sigma X\) does not cross any of the indecomposables of the outer terms of the~triangle but crosses the~inner terms, the map \(\Hom_{\mathcal{C}_{2p}}\left(X,g^i_{m+1}\right)\) is an isomorphism of one-dimensional vector spaces. As $m$ was arbitrary, we get that
\[
\Hom_{\Mod\dashmodule\mathcal{C}_{2p}}\left(\Yoneda(X),G^i\right)\simeq\varinjlim_{m\geq\max\{1,x,y\}}\Hom_{\mathcal{C}_{2p}}\left(X,G^i_{m+1}\right)
\]
is a non-zero direct summand of \(\Hom_{\mathcal{T}}\left(E,G^i\right)\), which is a contradiction.

Now, we know that \(\mathbf{E}\) consists of terms in \(\mathcal{D}'\simeq\mathcal{C}_p\). All that rests to be shown is that \(\mathbf{E}\) is Cauchy with respect to \(\mathcal{N}\).
Choose \(s\in\N\) and let \(N\in\N\) be such that \(\mathbf{E}\) stabilises at \(\nu(B_s)\coprod\mathcal{D}\) from \(N\) onwards. Let \(n\geq N\).
Consider the distinguished triangle
\[
E_n\rightarrow E_{n+1}\rightarrow C\oplus D\rightarrow\Sigma E_n
\]
where \(C\in B_s\) and \(D\in\mathcal{D}\). Since \(E_n,E_{n+1}\in \mathcal{D}'=\mathcal{D}^{\perp}={}^{\perp}\mathcal{D}\), the object $D$ is forced to be zero. Hence \(\mathbf{E}\) stabilises at \(\nu(B_s)\) from \(N\) onwards and \(E\in\mathcal{S}\).
\end{proof}

\begin{lemma}
\label{L:TTF}
The thick subcategory \(\mathcal{G}\subseteq\mathcal{T}\) is bireflective.  
\end{lemma}

\begin{proof}
Let \(E\in\mathcal{T}\). By Lemma~\ref{L:DimensionCounting}, the vector space \(\Hom_{\mathcal{T}}\left(E,G^i\right)\) is finite dimensional for every \(1\leq i\leq p\). As \(\mathcal{G}=\addf(G^1,\ldots,G^p)\) by Lemma~\ref{L:AdditiveGenerators}, we can construct a \(\mathcal{G}\)-preenvelope of~\(E\) by a standard argument (as in the proof of Corollary~\ref{C:CountablyGeneratedCompletionVerdierQuotient}). By Proposition~\ref{P:CharacterisationReflectivity}, we now know that the subcategory \(\mathcal{G}\) is enveloping. 

Pick a \(\mathcal{G}\)-envelope \(E\rightarrow G\) of $E$ and complete it to a triangle
\[
G\rightarrow F\rightarrow E\rightarrow G
.\]
By the triangulated Wakamatsu's lemma (e.g.\ the dual of \cite[Lemma~2.1]{Joergensen08}), the~object $F$ lies in \({}^\perp\mathcal{G}\). Fix \(1\leq i\leq p\).
We have an exact sequence
\[
\Hom_{\mathcal{T}}\big(G^i,F\big)\rightarrow\Hom_{\mathcal{T}}\big(G^i,E\big)\rightarrow\Hom_{\mathcal{T}}\big(G^i,G\big)
.\]
Since Lemma~\ref{L:CummingsGratzAisle} tells us that \(F\in\mathcal{S}\) and \(\mathcal{S}\) is a Krull-Schmidt category where indecomposables are module colimits of double fans with entries consisting of a single indecomposable object from \(\mathcal{C}_{2p}\), Lemma~\ref{L:DimensionCounting} ensures that the outer terms of the exact sequence are finite dimensional. Thus so is \(\Hom_{\mathcal{T}}(G^i,E)\).

With this information, the proof that \(\mathcal{G}\) is precovering is dual to the proof of the preenveloping property above. Then \(\mathcal{G}\) is bireflective by Proposition~\ref{P:CharacterisationReflectivity}.
\end{proof}

By Lemma~\ref{L:CummingsGratzAisle} and Lemma~\ref{L:TTF}, there is a TTF-triple \((\mathcal{S},\mathcal{G},\mathcal{G}^{\perp})\) on \(\mathcal{T}\).

\begin{theorem}
\label{T:TheEquivalenceIsAnEquivalence}
Consider the $t$-structure \((\mathcal{G},\mathcal{G}^{\perp})\) on \(\mathcal{T}\).
Then the coaisle truncation functor restricts to a triangulated equivalence \(\varphi:\mathcal{S}\xrightarrow{\simeq}\overline{\mathcal{C}_p}\).
Moreover, as an additive equivalence \(\varphi\) coincides with the~equivalence from \cite[Corollary~6.24]{CummingsGratz24} preserving the combinatorial model.
\end{theorem}

\begin{proof}
Since \((\mathcal{S},\mathcal{G},\mathcal{G}^{\perp})\) is a TTF-triple, the coaisle truncation already restricts to a triangulated equivalence \(\varphi:\mathcal{S}\xrightarrow{\simeq}\mathcal{G}^{\perp}\).
All we need to show is \(\varphi(\mathcal{S})\simeq\overline{\mathcal{C}_p}\) and discuss how \(\varphi\) acts on \(\mathcal{S}\).

Firstly, we consider a special case. Let \(E\simeq\moco\mathbf{E}\in\mathcal{S}\) be a double fan \(\mathbf{E}=(E_n,e_n)_{n\in\N}\) satisfying \(E_n=\{-1_i,n_j\big\}\in\mathcal{C}_p\) for some \(i,j\in\{1,\ldots,p\}\) and all \(n\in\N\).
For each \(k\in\{1,\ldots,p\}\setminus\{j\}\) and \(n\in\N\), the arc \(G^k_m\) does not cross \(E_n\) for all \(m>n\). Thus \(\Hom_{\mathcal{T}}(G^k,E)\simeq0\).
By Lemma~\ref{L:BasicFactorisationProperty}, we know \(\Hom_{T}(G^j,E)\simeq\Hom_{\mathcal{S}}\big(G^j_N,E\big)\) for some \(N\in\N\).
For each \(n\geq N\), applying \(\Hom_{\mathcal{T}}\big(G^j_N,\blank\big)\) to the~distinguished triangle
\[
\{n_i,(n+2)_i\} \rightarrow E_n \xrightarrow{e_{n} } E_{n+1}\rightarrow\{(n-1)_i,(n+1)_i\}
\]
yields an isomorphism \(\Hom_{\mathcal{T}}(G^j_N,e_n)\). Hence \(\Hom_{T}(G^j,E)\simeq K\), so the approximation triangle for \(E\) has to look like 
\[
G^j \xrightarrow{\alpha} E \xrightarrow{} \varphi(E)\rightarrow G^j
\]
with \(\alpha\neq 0\).
By \cite[Lemma~2.8.iii]{Neeman25A}, this approximation triangle is a module colimit of a Cauchy sequence of distinguished triangles
\[
G^j_n \xrightarrow{\alpha_n} E_n' \xrightarrow{} F_n' \rightarrow \Sigma G^j_n
\]
as \(n\rightarrow\infty\). We are also allowed to impose for all \(n\in\N\) that \(E'_n\simeq E_m\) for some \(m> n\).

We define another sequence \(\mathbf{F}=(F_n,f_n)_{n\in\N}\) in \(\mathcal{C}_{2p}\) with \(F_n=\{-1_i,n_{(j+1)'}\}\) and \(f_n\neq0\) for all \(n\in\N\). 
Set \(F:=\moco\mathbf{F}\).
Clearly \(\mathbf{F}\) is Cauchy with respect to \(\mathcal{D}\). But also $F$ is compactly supported at \(\mathcal{D}\) because every arc in \(\mathcal{D}\) crosses only finitely many entries of \(\mathbf{F}\). This means \(F\in\mathfrak{S}_{\mathcal{D}}(C_{2p})\simeq\overline{\mathcal{C}_p}\subseteq\mathcal{T}\).

Let \(s,M\in\N\) be such that \(F,\varphi(E)\) are compactly supported at \(s\) and stabilise at \(s+1\) from \(M\) onwards and \(\alpha_{M+1}\neq0\).
We find \(L> M+1\) such that \(E_{M+1}'\simeq E_L\).
Then the existence of the triangle 
\[
G^j_{M+1} \xrightarrow{\alpha_{M+1}} E_L \xrightarrow{} \{M_j,L_j\}\oplus F_M \rightarrow \Sigma G^j_{M+1}
\]
tells us that \(F_M\) is a direct summand of \(F_M'\). By Lemma~\ref{L:BasicFactorisationProperty}, the factorisation \(F_M\rightarrow F_M'\rightarrow F_M\) in \(\mathcal{C}_{2p}\) of \(\id_{F_M}\) uniquely lifts to the level of completions as the factorisation of \(\id_F\) as \(F\rightarrow \varphi(E)\rightarrow F\).
Since \(\varphi\restriction\mathcal{S}\) is fully faithful, the object \(\varphi(E)\) is indecomposable, so \(F\simeq\varphi(E)\).

The object \(E\) corresponds to the arc \(\{-1_i,A_{j+1}\}\) in the combinatorial model of \(\mathcal{S}\). 
By Lemma~\ref{L:FunctorIntoQuotient}, the~object \(F\) is the image of \(\{-1_i,1_{(j+1)'}\}\) under the quotient functor \(\mathcal{C}_p\rightarrow\overline{\mathcal{C}_p}\), which then corresponds the arc \(\{-1_i,A_{j+1}\}\) in the combinatorial model of \(\overline{\mathcal{C}_p}\).

Now we move to the general case. In both of the combinatorial models of \(\mathcal{S}\) and \(\overline{\mathcal{C}_p}\),
every arc of the~form \(\{x_i,A_j\}\) for some \(i,j\in\{1,\ldots,p\}\) and some \(x\in\Z\) can be obtained as an integer-iterated shift of \(\{-1_i,A_j\}\). So from the special case, we can infer that \(\varphi\) preserves all arcs of this form. 

But by Lemma~\ref{L:PaquetteYildirimGeneration} and Lemma~\ref{L:CummingsGratzGeneration}, both \(\varphi(\mathcal{S})\) and \(\overline{\mathcal{C}_p}\) are generated by indecomposable objects corresponding to the arcs of the~form \(\{x_i,A_j\}\) for some \(i,j\in\{1,\ldots,p\}\) and some \(x\in\Z\), each within their own respective combinatorial models. And every indecomposable object is a cone of a morphism between arcs of this form. With this in mind, we may finally conclude
that 
\(\varphi(\mathcal{S})\simeq\overline{\mathcal{C}_p}\)
and \(\varphi:\mathcal{S}\rightarrow\overline{\mathcal{C}_p}\) is a triangulated equivalence fully preserving the combinatorial model. 
\end{proof}

\section{Metric completions from weighted projective lines}
\label{sec:CompletingWPL}

This final section is dedicated to classifying metric completions of \(\dbcohX\) for \(\X\) a weighted projective line of domestic or tubular type.
Firstly, we summarise the key properties coherent sheaves of weighted projective lines (Subsection~\ref{SS:WPL}) and their derived categories (Subsection~\ref{SS:DerivedSheaves}). The classification result (Theorem~\ref{T:WeightedCompletions}) is proved in Subsection~\ref{SS:Classification}.

\subsection{Weighted projective lines}
\label{SS:WPL}

Weighted projective lines and their categories of coherent sheaves were introduced in \cite{GeigleLenzing87}. Roughly speaking, a weighted projective line is obtained from the usual projective line by assigning weights to finitely many distinguished points. Weighted points then give rise to exceptional simple object in the category of coherent sheaves. While a weighted projective line was originally defined in \cite{GeigleLenzing87} using graded sheaf theory, here we follow a more direct approach from \cite{ChenKrause10} and \cite{Lenzing11} where coherent sheaves are constructed as a quotient of suitable category of graded modules. For the construction of quasi-coherent sheaves on weighted projective lines, we refer to \cite[pg.\ 75]{HugelKussin17} and \cite{ArtinZhang94}.

Let \(K\) be an algebraically closed field. Let \(k\in\N\) and \(\mathbf{p}=(p_1,\ldots,p_k)\in\N^k\).
Let \(\bflambda=(\lambda_1,\ldots,\lambda_k)\)
be a $k$-tuple of pairwise distinct points on \(\ProjectiveLine{K}\).
By \textit{a weighted projective line} we understand the triple \(\X=\left(\ProjectiveLine{K},\bflambda,\mathbf{p}\right)\).
Let \(\gradingL\) be a rank one abelian group with generators \(\vX_1,\ldots,\vX_k\) subjected to relations \(p_1\vX_1=\ldots=p_n\vX_n\). 
We specify two distinguished elements of \(\gradingL\), \textit{the canonical element} \mbox{\(\CanonicalElement:=p_1\vX_1\)} and \textit{the dualising element} \(\DualisingElement:=(k-2)\CanonicalElement-\sum_{i=1}^k\vX_i\), and a partial order on \(\gradingL\) by \(\vL\geq0\Leftrightarrow \vL\in\sum_{i=1}^k\N_0\vX_i\).
The group \(\gradingL\) is almost totally ordered in a sense that for all \(\vL\in\gradingL\) we have \(\vL\geq0\) or \(\vL\leq\CanonicalElement+\DualisingElement\).

We define a \(K\)-algebra
\[\gradedS:=K[u,v,x_1,\ldots,x_k]/(f_1,\ldots,f_k)
\text{\quad where \quad}
f_i:=x_i^{p_i}-\left(\lambda^1_i u-\lambda^0_i v\right)\]
for each \(1\leq i\leq k\) and \(\lambda_i=\left[\lambda^0_i,\lambda^1_i\right]\in\bflambda\).
The algebra \(\gradedS\) is \(\gradingL\)-graded by imposing grading \(\deg(u)=\deg(v)=\CanonicalElement\) and \(\deg(x_i)=\vX_i\) for each \(1\leq i\leq k\).
Thus \(\gradedS=\bigoplus_{\vL\in\gradingL}\gradedS_{\vL}\) decomposes as a direct sum of its homogenous components \(\gradedS_{\vL}\) of degree \(\vL\). It holds that \(\gradedS_0=K\) and \(\gradedS_{\vL}=0\) whenever \(\vL\not\geq0\).
Consequently, the algebra \(\gradedS\) is graded local with the unique graded maximal ideal
\(\bigoplus_{0\neq\vL\in\gradingL}\gradedS_{\vL}\).

We consider the category \(\Mod^{\gradingL}\dashmodule\gradedS\) of \(\gradingL\)-graded \(\gradedS\)-modules and its subcategories:
\begin{itemize}
    \item \(\modf^{\gradingL}\dashmodule\gradedS\) finitely generated graded \(\gradedS\)-modules,
    \item \(\modf_0^{\gradingL}\dashmodule\gradedS\) graded \(\gradedS\)-modules of finite length,
    \item \(\Mod_0^{\gradingL}\dashmodule\gradedS:=\varinjlim\modf_0^{\gradingL}\dashmodule\gradedS\) (i.e.\ closure under directed colimits).
\end{itemize}
On a graded module \(M=\bigoplus_{\vX\in\gradingL}M_{\vX}\), we can perform a graded shift by \(\vY\in\gradingL\) to obtain \(M\left(\vY\right)\) with homogenous components \(M(\vY)_{\vX}:=M_{\vX+\vY}\).

As \(\Mod_0^{\gradingL}\dashmodule\gradedS\) is a Serre subcategory of \(\Mod^{\gradingL}\dashmodule\gradedS\), we may apply Serre's construction and define the category of \textit{quasi-coherent sheaves} over \(\X\) as the quotient
\[
\QcohX:=\Mod^{\gradingL}\dashmodule\gradedS\left/\Mod_0^{\gradingL}\dashmodule\gradedS\right.
.\]
This construction restricts to \(\modf^{\gradingL}\dashmodule\gradedS\) giving us the category of \textit{coherent sheaves} over \(\X\) as \(\cohX:=\modf^{\gradingL}\dashmodule\gradedS/\modf_0^{\gradingL}\dashmodule\gradedS\).

The subcategory of finite length objects of \(\cohX\) (also called \textit{torsion sheaves}) is denoted \(\cohfX\).
On the~other hand, the subcategory \(\vectX\) of \textit{vector bundles} consists precisely of objects of \(\cohX\) without any non-zero finite length subobjects.

The graded shift action of \(\gradingL\) on \(\Mod^{\gradingL}\dashmodule\gradedS\) induces a graded shift on \(\QcohX\).
The image \(\mathcal{O}\) of \(\gradedS\) in \(\cohX\) under the quotient functor is called \textit{the~structure sheaf}; its shifts \(\mathcal{O}(\vX)\) by some \(\vX\in\gradingL\) are \textit{line bundles}.  

\begin{remark}[{\cite[Corollary~7.4.4]{ChenKrause10}}]
The category \(\cohX\) does not depend on the choice of coordinates of the weighted points \(\lambda_i=\left[\lambda^0_i,\lambda^1_i\right]\in\bflambda\), \(1\leq i\leq k\), during the construction of \(\gradedS\).
\end{remark}

\begin{definition}
Let \(\mathcal{H}\) be a Hom-finite, $K$-linear, length, abelian category satisfying Serre duality. 
Assume that the Auslander-Reiten translate \(\tau:\mathcal{H}\rightarrow\mathcal{H}\) has finite order, and let \(n\in\N\) be minimal such that \(\tau^n\simeq\Id_{\mathcal{H}}\).
We say that \(\mathcal{H}\) is \textit{a tube} of \textit{rank} $n$ if \(\tau\) acts transitively on simple objects, i.e.\ for all simples \(U,V\in\mathcal{H}\) there exists \(m\in\N\) satisfying \(\tau^mU=V\).
Tube of rank $1$ is \textit{homogenous}.
\end{definition}

The basic properties of \(\cohfX\) are summarised in the theorem below.

\begin{theorem}[{\cite[Theorem~2.1]{Jasso15},\cite[Theorem~2.2]{Lenzing11}}]
The category \(\cohX\) is an abelian, noetherian, hereditary, $K$\=/linear, Hom-finite, Krull-Schmidt category. Moreover:
\begin{enumerate}[label=(\roman*)]
\item Every object \(X\in\cohX\) uniquely decomposes as a direct sum \(X=Y\oplus Z\) with \(Y\in\vectX\) and \(Z\in\cohfX\).
\item For all \(\vX,\vY\in\gradingL\) it holds that \(\Hom_{\cohX}\big(\mathcal{O}(\vX),\mathcal{O}(\vY)\big)\simeq\gradedS_{\vY-\vX}\).
\item Simple objects of \(\cohfX\) are up to isomorphism parametrised by \(\ProjectiveLine{\lambda}\). Each simple is of one of the two following forms:
\begin{enumerate}[label=(\alph*)]
  \item for \(\lambda\in\ProjectiveLine{K}\setminus\bflambda\), \(\left[\lambda^0,\lambda^1\right]\in\ProjectiveLine{K}\), there is a unique \textit{ordinary} simple \(S_{\lambda}\) determined by the short exact sequence
  \[
  0\rightarrow\mathcal{O}\xrightarrow{\lambda^1u-\lambda^0v}\mathcal{O}(\CanonicalElement)\rightarrow S_{\lambda}\rightarrow0,
  \]
  \item for \(1\leq i\leq k\) and \(1\leq j\leq p_i\) 
  there is a unique \textit{exceptional} simple \(S_{\lambda_i,j}\) determined by the~short exact sequence
  \[
  0\rightarrow\mathcal{O}\big((j-1)\vX_i\big)\xrightarrow{x_i}\mathcal{O}(j\vX)\rightarrow S_{\lambda_i,j}\rightarrow0.
  \]
\end{enumerate}
    \item The category \(\cohX\) satisfies Serre duality with the graded shift \(X\mapsto X(\DualisingElement)\) being the~Auslander-Reiten translation \(\tau\). In particular, we have \(S(\DualisingElement)\simeq S\) for every exceptional simple \(S\), and \(S_{\lambda_i,j}(\DualisingElement)\simeq S_{\lambda_i,(j \modf p_i)+1}\) for all \(1\leq i\leq k\) and \(1\leq j\leq p_i\). 
\end{enumerate}
\end{theorem}

We can see that the Auslander-Reiten translation \(\tau:\cohX\rightarrow\cohX\) restricts to an~Aus\-lan\-der-Rei\-ten translation on torsion sheaves and that simple objects concentrated at a fixed point \(\lambda\in\ProjectiveLine{K}\) form a~\(\tau\)\=/orbit. But \(\cohfX\) is a length category,  so \cite[Proposition~1.8.2]{ChenKrause10} implies \(\cohfX=\coprod_{\lambda\in\ProjectiveLine{K}}\mathcal{H}_\lambda\) where the wide subcategories \(\mathcal{H}_{\lambda}\) are generated by \(\tau\)-orbits of simple objects. If \(\lambda=\lambda_i\in\bflambda\) for some \(1\leq i\leq k\), then \(\mathcal{H}_{\lambda}\) is a tube of rank~\(p_i\). Otherwise, the tube \(\mathcal{H}_{\lambda}\) is homogenous.

\begin{lemma}[{\cite[Lemma~2.4]{Lenzing11}}]
Let \(X\in\cohX\). Then \(X\simeq0\) if and only if \([X]=0\) in~\(K_0(\cohX)\).   
\end{lemma}

\begin{definition}[{\cite[Section~2.2]{Lenzing11}}]
Let \(p\) be the least common multiple of \(p_1,\ldots,p_k\).
We define the~following two linear forms on the~Grothendieck group \(K_0(\cohX)\):
\begin{enumerate}[label=(\roman*)]
\item \textit{The rank} is the linear form \(\rank: K_0(\cohX)\rightarrow\Z\) determined by imposing \(\rank([X])=1\) for every line bundle $X$.
\item \textit{The degree} is the linear form \(\deg:K_0(\cohX)\rightarrow\Z\) characterised by \(\deg([\mathcal{O}])=0\), \(\deg([S])=p\) for every ordinary simple $S$, and \(\rank([S_{\lambda_i,j}])=\frac{p}{p_i}\) for every exceptional simple \(S_{\lambda_i,j}\)  given by \(1\leq i\leq k\) and \(1\leq j\leq p_i\).  
\end{enumerate}
For \(X\in\cohX\) we denote \(\rank(X):=\rank([X])\) and \(\deg(X):=\deg([X])\).
And for \(0\neq X\in\cohX\) we define \textit{the slope} of \(X\) as \(\mu(X):=\frac{\deg(X)}{\rank(X)}\in\Q\cup\{\infty\}\) where the~slope is infinite precisely when \(X\in\cohfX\).
\end{definition}

\begin{definition}[{\cite[Section~3.1]{Lenzing11}}]
We say that \(0\neq X\in\cohX\) is \textit{semistable} if for every subobject \(0\neq Y\) of $X$ we get \(\mu(Y)\leq\mu(X)\).

Let \(q\in\Q\cup\{\infty\}\). We define the category \(\mathcal{H}^{(q)}\) as the subcategory of \(\cohX\) consisting of the zero object and all the semistable objects \(X\) with \(\mu(X)=q\).
\end{definition}

\begin{proposition}[{\cite[Lemma~3.1]{Lenzing11}}]
Let \(q\in\Q\cup\{\infty\}\). Then \(\mathcal{H}^{(q)}\) is abelian, hereditary, length, $K$-linear category which is extension-closed as a subcategory of \(\cohX\).
\end{proposition}

Note that \(\mathcal{H}^{(\infty)}=\cohfX\).

\begin{definition}
We define \textit{the Euler characteristic} of \(\X\) as \(\chi_{\X}:=2-\sum_{i=1}^k\left(1-\frac{1}{p_i}\right)\) and say that \(\X\) is:
\begin{itemize}
    \item \textit{domestic} if \(\chi_{\X}>0\),
    \item  \textit{tubular} if \(\chi_{\X}=0\),
    \item or \textit{wild} if \(\chi_{\X}<0\).
\end{itemize}
\end{definition}

\begin{remark}
As discussed in \cite[pg.\ 12]{Jasso15}, the weighted projective line \(\X\) is tubular if and only if the~tuple \(\mathbf{p}\) is a permutation of \((2,2,2,2)\), \((3,3,3)\), \((2,4,4)\), or \((2,3,6)\).
\end{remark}

\begin{remark}
The Euler characteristic is tied to the degree of the dualising element \(\DualisingElement\) as
\[
\chi_{\X}=-\frac{\deg\big(\mathcal{O}(\DualisingElement)\big)}{p}
\]
holds for \(p\) being the least common multiple of \(p_1,\ldots,p_k\). By \cite[Lemma~2.5]{Lenzing11}, we get
\[
\mu(\tau X)=\mu\big(X(\DualisingElement)\big)=\mu(X)+\deg\big(\mathcal{\DualisingElement}\big)=\mu(X)-p\chi_{\X}
.\]
In other words, the Euler characteristic tells us how \(\tau\) changes the slope of an object.
\end{remark}

\begin{proposition}[{\cite[Proposition~5.5)]{GeigleLenzing87}}]
If \(\chi_{\X}\geq0\), then every indecomposable object in \(\cohX\) is semistable.

Moreover, if \(\X\) is domestic, then every indecomposable vector bundle is exceptional.
\end{proposition}

\subsection{Derived categories of weighted projective lines}
\label{SS:DerivedSheaves}

In this subsection, we fix a weighted projective line \(\X=\left(\ProjectiveLine{K},\bflambda,\mathbf{p}\right)\) with \(k\in\N\) weighted points over an algebraically closed field $K$.

The quasi-coherent sheaves \(\QcohX\) form a hereditary, locally noetherian Gro\-then\-dieck category (\cite[pg.\ 1481]{Kussin16}) with \(\cohX\) being its noetherian objects. The derived category \(\derived(\QcohX)\) is then compactly generated with compact elements being isomorphic to bounded complexes of objects from \(\cohX\) by \cite[Proposition~2.5]{Stovicek09} (see also \cite{Krause05A}). Hence, we may identify \(\dbcohX\simeq\derived(\QcohX)^c\) as a~triangulated subcategory of \(\derived(\QcohX)\).
As \(\QcohX\) is hereditary, every object of \(\derived(\QcohX)\) decomposes as a~coproduct of its cohomologies (see e.g.\ \cite[Proposition~4.4.15]{Krause21}).

The category \(\dbcohX\) is generated by an exceptional sequence. For the sake of completeness, we include the definition.

\begin{definition}[{\cite[Section~2]{Bondal90}}]
Let \(\mathcal{T}\) be a $K$-linear triangulated category.
An object \(E\in\mathcal{T}\) is \textit{exceptional} if \(\End_{\mathcal{T}}(E)\simeq K\) and
\(
\Hom_{\mathcal{T}}(E,\Sigma^n E)\simeq0
\)
for all
\(n\in\Z\setminus\{0\}\).

A tuple \((E_1,\ldots,E_n)\) of exceptional objects \(E_1,\ldots,E_n\in\mathcal{T}\) for some \(n\in\N_0\) is said to be \textit{an exceptional sequence} if \(\Hom_{\mathcal{T}}(X_j,\Sigma^m X_i)\simeq0\) for all \(m\in\Z\) and \(1\leq i<j\leq n\).
\end{definition}

\begin{lemma}[{See eg.\ \cite[pg.\ 38]{Elagin25}}]
\label{L:LineBundlesFormExceptionalSequence}
The category \(\dbcohX\) is generated by an exceptional sequence.
Line bundles of the form \(\mathcal{O}(\vL)\) for \(0\leq\vL\leq\CanonicalElement\) can be arranged into
a particular example of an exceptional sequence
\[
\Big(\mathcal{O},\mathcal{O}(\vX_1),\ldots,\mathcal{O}\big((p_1-1)\vX_1\big),\ldots,\mathcal{O}(\vX_k),\ldots,\mathcal{O}\big((p_k-1)\vX_k\big),\mathcal{O}(\CanonicalElement)\Big)
.\]
of length \(2+\sum_{i=1}^k(p_i-1)\) generating \(\dbcohX\).    
\end{lemma}

The direct sum \(\bigoplus_{0\leq\vL\leq\CanonicalElement}\mathcal{O}(\vL)\) is known as \textit{the canonical tilting bundle} (see \cite[Theorem~2.9]{Lenzing11} and \cite[Theorem~7.1.1]{ChenKrause10}) and its endomorphism ring is isomorphic to \textit{the canonical algebra} $A$ (originally defined by Ringel in \cite[Section~3.7]{Ringel84}) of type \((\mathbf{p},\bflambda)\) which is the path algebra of the quiver
\begin{center}
\[
\xymatrix@!0@R20pt@C30pt{
             && \mathcal{O}(\vX_1)\ar[rr]^{x_1}&&\mathcal{O}(2\vX_1)\ar[rr]&&\cdots\ar[rr]&&\mathcal{O}\big((p_1-1)\vX_1\big)\ar[rrrdddd]^{x_1}\\
             &&                  &&                 &&             &&        &                  \\
             &&\mathcal{O}(\vX_2)\ar[rr]_{x_2}&&\mathcal{O}(2\vX_2)\ar[rr]&&\cdots\ar[rr]&&\mathcal{O}\big((p_2-1)\vX_2\big)\ar[rrrdd]_{x_2} && \\
             &&           &&  \vdots    &&           &&  \vdots   &&   & \\
\mathcal{O}\ar[rruuuu]^{x_1}\ar[rruu]_{x_2}\ar[rrdd]^{x_n} \ar@<-.5ex>[rrrrrrrrrrr]_u \ar@<.5ex>[rrrrrrrrrrr]^v &&           &&      &&           &&     &&   &\mathcal{O}(\CanonicalElement)\\
             &&              &&     \vdots       &&             &&      \vdots & \\
             &&\mathcal{O}(\vX_n)\ar[rr]^{x_n}&&\mathcal{O}(2\vX_n)\ar[rr]&&\cdots\ar[rr]&&\mathcal{O}\big((p_n-1)\vX_n\big)\ar[rrruu]^{x_n} & \\
}
\]
\end{center}
bound by the relations \(x_i^{p_i}=\left(\lambda^1_i u-\lambda^0_i v\right)\) for each \(1\leq i\leq n\).
By \cite[Theorem~3.2.5]{ChenKrause10} (a~generalisation of \cite[Theorem~III.2.10]{Happel88}), the~category \(\dbcohX\) is then equivalent to \(\derived^b(\modf\dashmodule A)\).

\begin{theorem}[{\cite[Theorem~3.5]{Lenzing11}}]
If \(\X\) is domestic, then there exists a tilting object \(T\in\cohX\) such that \(A:=\End_{\cohX}(T)\) is isomorphic to a path algebra of a Euclidean quiver over $K$. Hence \(A\) is a tame finite dimensional $K$-algebra and \(\dbcohX\simeq\derived^b(\modf\dashmodule A)\).
\end{theorem}

\begin{proposition}
\label{P:BireflectiveExceptionalCollection}
The following statements are equivalent for a thick subcategory \(\mathcal{S}\subseteq\dbcohX:\)
\begin{enumerate}[label=(\roman*)]
    \item \(\mathcal{S}\) is generated by an exceptional sequence,
    \item \(\mathcal{S}\) is bireflective,
    \item \(\mathcal{S}\) is reflective,
    \item \(\mathcal{S}\) is coreflective.
\end{enumerate}
\end{proposition}

\begin{proof}
The implication ``i)\(\Rightarrow\)ii)'' holds by \cite[Theorem~3.2.a]{Bondal90}, while \cite[Corollary~8.4]{Elagin25} proves the converse.
The remaining implications have been implicitly proved in \cite[Section~2]{BondalKapranov90} for \(K=\C\) but the~same argument works for any field. 
By \cite[Corollary~on~pg.\ 530]{BondalKapranov90}, the category \(\dbcohX\) is left and right saturated - see \cite[Definition~2.5]{BondalKapranov90} for the definition.
If \(\mathcal{S}\) if reflective, then it is right saturated by \cite[Proposition~2.8]{BondalKapranov90}, so \cite[Proposition~2.6]{BondalKapranov90} implies that \(\mathcal{S}\) is coreflective. The~rest is dual.
\end{proof}

\begin{corollary}
\label{C:LiberteEgaliteFraternite}
Let \(m,n\in\N_0\).
Let \((E_1,\ldots,E_m)\) and \((F_1,\ldots,F_n)\) be two exceptional sequences in \(\dbcohX\).
Denote \(\mathcal{U}:=\langle E_1,\ldots,E_m\rangle\) and \(\mathcal{V}:=\langle F_1,\ldots,F_n\rangle\).
If \(\mathcal{U}\subseteq\mathcal{V}\), then \(m\leq n\).

Moreover, the equality \(\mathcal{U}=\mathcal{V}\) holds if and only if \(m=n\).
\end{corollary}

\begin{proof}
As \(\mathcal{U}\) is bireflective, the truncation functors of the $t$-structure \(\left(\mathcal{U},\mathcal{U}^{\perp}\cap\mathcal{V}\right)\) on \(\mathcal{V}\) induce a splitting of the Grothendieck group
\(
K_0\left(\mathcal{V}\right)\simeq K_0\left(\mathcal{U}\right)\oplus K_0\left(\mathcal{U}^{\perp}\cap\mathcal{V}\right)
.\)
One can prove by induction on \(m\) that \(K_0(\mathcal{U})\simeq\Z^m\). Then \(K_0\left(\mathcal{U}^{\perp}\cap\mathcal{V}\right)\simeq\Z^{n-m}\) shows \(m\leq n\).

If we furthermore assume the equality \(m=n\), we get an isomorphism \(K_0\left(\mathcal{U}^{\perp}\cap\mathcal{V}\right)\simeq0\). However, \cite[Corollary~8.10]{Elagin25} tells us that the only triangulated subcategory of \(\dbcohX\) with zero Grothendieck group is the zero category.
The equality \(\mathcal{U}=\mathcal{V}\) directly follows.
\end{proof}

\begin{remark}
The implications ``i)\(\Rightarrow\)ii)\(\Leftrightarrow\)iii)\(\Leftrightarrow\)iv)'' of Proposition~\ref{P:BireflectiveExceptionalCollection} hold even more generally for proper $K$-linear triangulated categories generated by an~exceptional sequence. But ``i)\(\Leftarrow\)ii)'' and Corollary~\ref{C:LiberteEgaliteFraternite} are specific to \(\dbcohX\) because it does not contain \textit{phantom} subcategories; those are non\=/zero bireflective triangulated subcategories with vanishing Grothendieck group.
A counterexample to the~statement of Corollary~\ref{C:LiberteEgaliteFraternite} for a $K$-linear triangulated category possessing phantom subcategories can be found in \cite[Theorem~1.1]{Krah23}.
\end{remark}

\begin{lemma}
\label{L:EverythingExceptional}
Let \(\mathcal{S}\) be a thick subcategory of \(\dbcohX\). If every indecomposable object of \(\mathcal{S}\) is exceptional, then \(\mathcal{S}\) is bireflective.
\end{lemma}

\begin{proof}
By Lemma~\ref{L:LineBundlesFormExceptionalSequence} and Corollary~\ref{C:LiberteEgaliteFraternite}, the length of any exceptional sequence in \(\dbcohX\) is bounded by \(2+\sum_{i=1}^k(p_i-1)\).
Let \((E_1,\ldots,E_m)\) be an exceptional sequence in \(\mathcal{S}\) with \(m\in\N\) maximal possible, and \(\mathcal{E}:=\langle E_1,\ldots,E_m\rangle\subseteq\mathcal{S}\). 
Assume for contradiction, that \(\mathcal{E}\subsetneq\mathcal{S}\). Choose \(X\in\mathcal{S}\setminus\mathcal{E}\). Since \(\left(\mathcal{E},\mathcal{E}^{\perp}\right)\) is a $t$-structure on \(\dbcohX\), there exists a~distinguished triangle
\[
E\rightarrow X\rightarrow Y\rightarrow \Sigma E
\]
in \(\mathcal{S}\) with \(Y\in\mathcal{E}^\perp\). Pick an indecomposable direct summand \(E_0\) of \(Y\).
Then the exceptional sequence \((E_0,E_1,\ldots,E_m)\) in \(\mathcal{S}\) has \(m+1\) terms. A contradiction.
\end{proof}

\begin{proposition}
\label{P:VectBireflective}
If \(\X\) is domestic, then any thick subcategory of \(\dbcohX\) containing a vector bundle is bireflective.    
\end{proposition}

\begin{proof}
By {\cite[Proposition~7.1]{Cheng24}}, if a thick subcategory of \(\dbcohX\) contains a vector bundle, it is generated by an exceptional sequence. The rest follows from Proposition~\ref{P:BireflectiveExceptionalCollection}.
\end{proof}

\begin{theorem}[{\cite[{Proposition~10.18, Theorem~10.19}]{Lenzing07}}]
\label{T:TubularMutation}
Let \(q,q'\in\Q\cup\{\infty\}\). If \(\X\) is tubular, then there exists a triangulated autoequivalence \(\dbcohX\xrightarrow{\simeq}\dbcohX\) restricting to an equivalence \(\mathcal{H}^{(q)}\xrightarrow{\simeq}\mathcal{H}^{(q')}\).
In particular, every \(\mathcal{H}^{(q)}\) decomposes as a coproduct of tubes, i.e.\ \(\mathcal{H}^{(q)}\simeq\coprod_{\lambda\in\ProjectiveLine{K}}\mathcal{H}^{(q)}_\lambda\) where each \(\mathcal{H}^{(q)}_{\lambda}\) is equivalent to the tube \(\mathcal{H}_{\lambda}\subseteq\cohfX\). 
\end{theorem}

We shall use the notation \(\mathscr{H}^{(q)}:=\big\langle\mathcal{H}^{(q)}\big\rangle\) for the corresponding thick categories generated by sem\-i\-sta\-ble objects of a fixed slope \(q\in\Q\cup\{\infty\}\). 
If \(q=\infty\) or \(\X\) is tubular, we also denote \(\mathscr{H}^{(q)}_{\lambda}:=\big\langle\mathcal{H}^{(q)}_{\lambda}\big\rangle\) for each \(\lambda\in\ProjectiveLine{K}\).
In that case, the thick categories generated by tubes of the same slope are mutually orthogonal, meaning that
\[
\mathscr{H}^{(q)}_{\lambda}\subseteq\left(\mathscr{H}^{(q)}_{\lambda'}\right)^{\perp}
\text{\ \ and\ \ }
\mathscr{H}^{(q)}_{\lambda'}\subseteq\left(\mathscr{H}^{(q)}_{\lambda}\right)^{\perp}
\text{\ \ in \(\dbcohX\) for any \(\lambda\neq\lambda'\in\ProjectiveLine{K}\)}
.\]

\begin{lemma}
\label{L:TwoSlopesBireflective}
Let \(q,q'\in\Q\cup\{\infty\},q\neq q'\), and let \(\mathcal{S}\) be a thick subcategory of \(\dbcohX\) containing \(X,Y\in\cohX\) with \(\mu(X)=q\) and \(\mu(Y)=q'\).
If \(\chi_{\X}\geq0\), then \(\mathcal{S}\) is bireflective.
\end{lemma}

\begin{proof}
Note that \(X\) or \(Y\) is a vector bundle contained in \(\mathcal{S}\).
Firstly, if \(\X\) is domestic, then the statement now directly follows from Proposition~\ref{P:VectBireflective}.
Secondly, if every indecomposable object of \(\mathcal{S}\) is exceptional, then \(\mathcal{S}\) is bireflective by Lemma~\ref{L:EverythingExceptional}.

Finally, suppose that \(\X\) is tubular and \(\mathcal{S}\) contains a non-exceptional indecomposable object. W.l.o.g.\ the object \(X\) is non-exceptional and indecomposable.
By Theorem~\ref{T:TubularMutation}, we may also w.l.o.g.\ assume \(q=\infty\).
Since \(\mathcal{S}\cap\,\cohX\) now contains a vector bundle and a non-exceptional indecomposable torsion sheaf $X$, it holds by \cite[Proposition~6.2, Theorem~6.3]{Cheng24} that \(\mathcal{S}=\langle E_1,\ldots,E_n\rangle^{\perp}\) for some exceptional sequence \((E_1,\ldots,E_n)\) of objects from \(\cohfX\) and index \(n\in\N_0\).

Considering that \(\langle E_1,\ldots,E_n\rangle\) is bireflective, this already implies that \(\mathcal{S}\) is reflective. 
Using Proposition~\ref{P:BireflectiveExceptionalCollection}, we get that \(\mathcal{S}\) is bireflective.
\end{proof}

We now state variants of \cite[Lemma~2.25]{Matousek26B} for domestic and tubular weighted projective lines.

\begin{lemma}
\label{L:TubularInfiniteChainDiffersAtHomogenousTubes}
Assume \(\X\) is tubular.
Then there exists a number \(M\in\N\) (depending on~$\X$) such that for any strictly decreasing chain \(\mathcal{S}_1\supsetneq\mathcal{S}_2\supsetneq\cdots\supsetneq\mathcal{S}_M\) of thick subcategories of~\(\dbcohX\) we can find \(1\leq i<j\leq M\), \(q\in\Q\cup\{\infty\}\), and \(\lambda\in\ProjectiveLine{K}\setminus\bflambda\) with \(\mathcal{S}_i,\mathcal{S}_j\subseteq\mathcal{H}^{(q)}\) such that \(\mathcal{H}^{(q)}_{\lambda}\subseteq\mathcal{S}_i\) and \(\mathcal{H}^{(q)}_{\lambda}\cap\,\mathcal{S}_j=0\).
\end{lemma}

\begin{proof}
If \(\mathcal{S}\subsetneq\mathcal{S}'\) are bireflective subcategories of \(\dbcohX\), they are generated by exceptional collections due to Proposition~\ref{P:BireflectiveExceptionalCollection}. Then the exceptional sequence generating \(\mathcal{S}\) must be shorter than the one generating \(\mathcal{S}'\) because of Corollary~\ref{C:LiberteEgaliteFraternite}.
Using induction, we infer that any strictly decreasing chain \(\mathcal{S}_1\supsetneq\mathcal{S}_2\supsetneq\cdots\supsetneq\mathcal{S}_N\), \(N\in\N\), of thick subcategories of \(\dbcohX\) can include at most \(3+\sum_{i=1}^k(p_i-1)\) bireflective subcategories, where the upper bound is provided by Lemma~\ref{L:LineBundlesFormExceptionalSequence}. 

We are now left with proving the statement for
\(
M:=1+\prod_{i=1}^k\binom{2p_i}{p_i}
\)
and a~strictly decreasing chain \(\mathcal{S}_1\supsetneq\mathcal{S}_2\supsetneq\cdots\supsetneq\mathcal{S}_M\) of non-bireflective thick subcategories of \(\dbcohX\). Then by Lemma~\ref{L:TwoSlopesBireflective} we get \(\mathcal{S}_1\subseteq\mathscr{H}^{(q)}\) for some \(q\in\Q\cup\{\infty\}\).
As \(\mathcal{H}^{(q)}\simeq\coprod_{\lambda\in\ProjectiveLine{K}}\mathcal{H}^{(q)}_{\lambda}\) and
\(
\left\langle\coprod_{\lambda\in\bflambda}\mathcal{H}^{(q)}_{\lambda}\right\rangle\)
has only
\(
\prod_{i=1}^k\binom{2p_i}{p_i}
\)
thick subcategories (see {\cite[Proposition~2.4.2]{Dichev09}}), there must consequently exist indices \(1\leq i< j\leq M\) such that the~subcategories \(\mathcal{S}_i\) and \(\mathcal{S}_j\) differ by \(\mathcal{S}_i\) containing a homogenous tube \(\mathcal{H}^{(q)}_{\lambda}\) for some \(\lambda\in\ProjectiveLine{K}\setminus\bflambda\), while  \(\mathcal{H}^{(q)}_{\lambda}\cap\,\mathcal{S}_j=0\).
\end{proof}

\begin{lemma}
\label{L:DomesticInfiniteChainDiffersAtHomogenousTubes}
Assume \(\X\) is domestic.
Then there exists a number \(M\in\N\) (depending on~$\X$) such that for any strictly decreasing chain \(\mathcal{S}_1\supsetneq\mathcal{S}_2\supsetneq\cdots\supsetneq\mathcal{S}_M\) of thick subcategories of~\(\dbcohX\) we can find \(1\leq i<j\leq M\) and \(\lambda\in\ProjectiveLine{K}\setminus\bflambda\) with \(\mathcal{S}_i,\mathcal{S}_j\subseteq\mathscr{H}^{(\infty)}\) such that \(\mathcal{H}_{\lambda}\subseteq\mathcal{S}_i\) and \(\mathcal{H}_{\lambda}\cap\,\mathcal{S}_j=0\).
\end{lemma}

\begin{proof}
Considering that any thick subcategory of \(\dbcohX\) containing a vector bundle is bireflective by Proposition~\ref{P:VectBireflective}, the proof of the lemma is identical to the~proof of Lemma~\ref{L:TubularInfiniteChainDiffersAtHomogenousTubes} in the special case of \(q=\infty\).    
\end{proof}

\begin{lemma}
\label{L:HomogenousTubesExtSeeVectorBundles}
Let \(\lambda\in\ProjectiveLine{K}\setminus\bflambda\). Let \(0\neq X\in\mathcal{H}_{\lambda}\) and \(V\in\vectX\). Then \(\Hom_{\dbcohX}(X,\Sigma V)\neq0\).
\end{lemma}

\begin{proof}
By \cite[Corollary~2.7]{GeigleLenzing87}, there exists a short exact sequence
\[
0\rightarrow V\rightarrow L\rightarrow T\rightarrow 0
\]
in \(\cohX\) with $L$ a direct sum of line bundles and \(0\neq T\in\mathcal{H}_{\lambda}\). This sequence does not split, so \(\Hom_{\dbcohX}(T,\Sigma V)\neq0\).
The tube \(\mathcal{H}_{\lambda}\) is homogenous, so any of its elements is filtered by the unique simple of the tube. 
This implies \(\Hom_{\dbcohX}(X,\Sigma V)\neq0\).
\end{proof}

\subsection{Classifying metric completions}
\label{SS:Classification}

As in the preceding subsection, we fix a weighted projective line \(\X=\left(\ProjectiveLine{K},\bflambda,\mathbf{p}\right)\) with \(k\in\N\) weighted points over an algebraically closed field $K$.
We also set \(\mathcal{S}:=\derived^b(\cohX)\) and \(\mathcal{T}:=\derived(\QcohX)\).
We know that \(\mathcal{S}\xhookrightarrow{}\mathcal{T}\) is a good extension with respect to all metrics due to Theorem~\ref{T:PrototypicalGoodExtension} because \(\mathcal{T}\) has coproducts and \(\mathcal{S}=\mathcal{T}^c\). Hence all completions \(\mathfrak{S}'(\mathcal{S})\) will be calculated as subcategories of \(\mathcal{T}\).

This section is dedicated to the classification of completions of \(\mathcal{S}\) (Theorem~\ref{T:WeightedCompletions}) when \(\X\) is domestic or tubular. We mimic the overall strategy of Section~7 of \cite{Matousek26B}.

We shall now state four results from \cite{Matousek26B}, which directly generalise from the original setting of derived categories of hereditary, finitely generated algebras, to the setting of derived categorises of locally noetherian, hereditary Grothendieck categories. In particular, they apply to \(\derived(\QcohX)\).
As the original proofs need only formal modifications to include the stronger hypothesis, we omit their proofs here.

\begin{lemma}[{\cite[Lemma~4.5]{Matousek26B}}]
\label{L:GrothendieckBruningWideTheoremForMetrics}
Equip \(\mathcal{S}\) with an~additive good metric \(\mathcal{M}=\{B_n\}_{n\in\N}\). Let \(\mathcal{C}\subseteq\cohX\). Denote \(\mathcal{W}\subseteq\cohX\) the wide subcategory generated by \(\mathcal{C}\). If \(\mathcal{C}\subseteq B_{n+2}\) as objects of \(\mathcal{S}\), then \(\mathcal{W}\subseteq B_n\).
\end{lemma}

\begin{lemma}[{\cite[Lemma~4.9]{Matousek26B}}]
\label{L:LivingInCompletionGivenByThickSubcategory}
Let \(\mathcal{M}=\{B_n\}_{n\in\N}\) be an additive good metric on \(\mathcal{S}\) and \(E\in\mathfrak{S}'_{\mathcal{M}}(\mathcal{S})\). Then for every \(n\in\N\) there exist a~countably generated thick subcategory \(\mathcal{C}\subseteq\mathcal{S}\) such that \(\mathcal{C}\subseteq\big\langle H^0(B_n)\big\rangle\) and \(E\in\mathfrak{S}'_{\mathcal{C}}(\mathcal{S})\).
\end{lemma}

We recall that for a good metric \(\mathcal{M}=\{B_n\}_{n\in\N}\) on \(\mathcal{S}\) and the standard $t$-structure \((\mathcal{S}^{\leq0},\mathcal{S}^{>0})\), we define metric \(\mathcal{M}_{\infty}:=\mathcal{M}\wedge\left(\{\mathcal{S}^{\leq n}\}_{n\in\N}\vee\{\mathcal{S}^{>n}\}_{n\in\N}\right)\). Informally speaking, the metric \(\mathcal{M}_{\infty}\) is obtained by reducing the entries of \(\mathcal{M}\) to objects concentrated in very large or very small cohomological degrees.
In the upcoming, the functor \(H^0:\mathcal{S}\rightarrow\cohX\) is the zeroth cohomology with respect to this $t$-structure.

\begin{lemma}[{\cite[Proposition~5.10]{Matousek26B}}]
\label{L:DecompostionMetricsCompletion}
Equip \(\mathcal{S}\) with additive good metrics \(\mathcal{M}\) and \(\mathcal{N}\).
If \(\mathcal{M}=\mathcal{M}_{\infty}\vee\mathcal{N}\), then \(\mathfrak{S}'_{\mathcal{M}}(\mathcal{S})=\mathfrak{S}'_{\mathcal{N}}(\mathcal{S})\).
\end{lemma}

\begin{lemma}[{\cite[Lemma~5.11]{Matousek26B}}]
\label{L:ObstructingObjectsPC}
Equip \(\mathcal{S}\) with an~additive good metric \(\mathcal{M}=\{B_n\}_{n\in\N}\).
If \(\mathcal{M}\not\sim\mathcal{M}_{\infty}\vee\mathcal{B}\), where \(\mathcal{B}:=\bigcap_{n\in\N} B_n\), then for every \(E\in\mathfrak{S}'_{\mathcal{M}}(\mathcal{S})\) there exist an~additive good metric \(\mathcal{N}=\{C_n\}_{n\in\N}\sim\mathcal{M}\) on \(\mathcal{S}\) and a family \((X_n)_{n\in\N}\) of objects \(\cohX\) such that for all \(n\in\N\):
\begin{itemize}
    \item \(\Hom_{\mathcal{T}}\left(\Sigma^iX_n,E\right)=0\) for all \(i\in\Z\),
    \item \(X_n\in C_n\),
    \item \(X_{n}\notin\big\langle H^0(C_{n+1})\big\rangle\).
\end{itemize}
\end{lemma}

\begin{lemma}
\label{L:ObstructingObjectsSubcategories}
Equip \(\mathcal{S}\) with an additive good metric \(\mathcal{M}=\{B_n\}_{n\in\N}\) and set \(\mathcal{B}:=\bigcap_{n\in\N} B_n\).
If it holds that \(\mathcal{M}\not\sim\mathcal{M}_{\infty}\vee\mathcal{B}\), then for every \(E\in\mathfrak{S}'_{\mathcal{M}}(\mathcal{S})\) there exists an~additive good metric \(\mathcal{N}=\{C_n\}_{n\in\N}\) on \(\mathcal{S}\) such that for all \(n\in\N\): 
\begin{itemize}
    \item \(C_n\subseteq\big\langle H^0(B_n)\rangle\) is a thick subcategory,
    \item \(C_{n+1}\neq C_n\),
    \item \(E\in\mathfrak{S}'_{C_n}(\mathcal{S})\).
\end{itemize}
\end{lemma}

\begin{proof}
By Lemma~\ref{L:ObstructingObjectsPC}, we can assume that there exists a family \((X_n)_{n\in\N}\) of coherent sheaves such that \(X_n\in B_n\setminus\big\langle H^0(B_{n+1})\big\rangle\) and \(\langle X_n\rangle\subseteq{}^{\perp}E\).
Using Lemma~\ref{L:LivingInCompletionGivenByThickSubcategory}, we pick a thick subcategory \(\mathcal{E}_n\subseteq\big\langle H^0(B_n)\big\rangle\) satisfying \(E\in\mathfrak{S}'_{\mathcal{E}_n}(\mathcal{S})\) for each \(n\in\N\).

Then for every \(n\in\N\) we set \(C_n:=\langle\mathcal{E}_m,X_m:m\geq n\rangle\). By the~choice of the~family \((X_n)_{n\in\N}\), the~thick subcategories \(C_1\supsetneq C_2\supsetneq\ldots\) form a~strictly decreasing chain and, therefore, an additive good metric 
\(\mathcal{N}:=\{C_n\}_{n\in\N}\).
Since for all \(n,m\in\N\) we have \(E\in\mathfrak{L}'_{\mathcal{E}_n}(\mathcal{S})\cap\mathfrak{C}_{\mathcal{E}_n}'(\mathcal{S})\cap\mathfrak{C}_{\langle X_m\rangle}'(\mathcal{S})\), we also get \(E\in\mathfrak{S}'_{C_n}(\mathcal{S})\).
\end{proof}

The statement below is based on \cite[Lemma~7.1]{Matousek26B}. The formulations are adjusted to \(\cohX\).

\begin{proposition}
\label{P:UncountablyGeneratedCompletion}
Let \(\mathcal{D}\subsetneq\mathcal{C}\) be thick subcategories of \(\mathcal{S}\) generated by torsion sheaves. Let \(\lambda\in\ProjectiveLine{K}\setminus\bflambda\) and pick indecomposable \( T\in\mathcal{H}_{\lambda}\).
Then:
\begin{enumerate}[label=(\roman*)]
\item If \(T\in\mathcal{C}\setminus\mathcal{D}\), then it holds that \(\mathfrak{S}'_{\mathcal{C}}(\mathcal{S})\cap\mathfrak{S}'_{\mathcal{D}}(\mathcal{S})\subseteq\mathcal{C}^{\perp}\cap\mathscr{H}^{(\infty)}\).
\item If \(\mathcal{C}\) is not countably generated, then \(\mathfrak{S}'_{\mathcal{C}}(\mathcal{S})=\mathcal{C}^{\perp}\cap\mathscr{H}^{(\infty)}\). 
\end{enumerate}
\end{proposition}

\begin{proof}
``Part i)''
Pick \(E\in\mathfrak{S}'_{\mathcal{C}}(\mathcal{S})\cap\mathfrak{S}'_{\mathcal{D}}(\mathcal{S})\) and choose a Cauchy sequence \(\mathbf{E}=(E_n,e_n)_{n\in\N}\) with respect to \(\mathcal{D}\) satisfying \(E\simeq\hoco\mathbf{E}\). The sequence \(\mathbf{E}\) is also Cauchy with respect to \(\mathcal{C}\). By Lemma~\ref{L:CummingsGratzTrivialisation}, we may additionally assume that the~sequence \(\mathbf{E}\) is \(\mathcal{C}\)-trivial. We may also assume \(\cone e_n\in\mathcal{D}\) for all \(n\in\N\); otherwise we just remove finitely many entries from \(\mathbf{E}\).

Since \(T\notin\mathcal{D}\), we infer \(\mathcal{D}\subseteq T^{\perp}\) using the tube orthogonality.
Fix \(j\in\Z\).
The functor \(\Hom_{\mathcal{S}}(\Sigma^jT,\blank)\) vanishes on \(\cone e_n\) for all \(n\in\N\). Hence, the sequence \(\Hom_{\mathcal{S}}(\Sigma^jT,\mathbf{E})\) is constant. For each \(n\in\N\) we get
\[
\Hom_{\mathcal{S}}\left(\Sigma^jT,E_n\right)\simeq\varinjlim\Hom_{\mathcal{S}}\left(\Sigma^jT,\mathbf{E}\right)\simeq\Hom_{\mathcal{S}}(\Sigma^jT,E)\simeq0
.\]
Using \(\Hom_{\mathcal{S}}\big(\Sigma^{-1}T,V\big)\neq0\) for all \(V\in\vectX\) from Lemma~\ref{L:HomogenousTubesExtSeeVectorBundles}, we get that for all \(n\in\N\) the object \(E_n\) does contain any shift of vector bundle as a~direct summand.
In other words, we get \(E_n\in\mathscr{H}^{(\infty)}\) for all \(n\in\N\). Because \(\mathbf{E}\) is also \(\mathcal{C}\)-trivial, the mutual orthogonality of tubes implies \(E_n\in\mathcal{C}^{\perp}\) for all \(n\in\N\).

We now know that all the entries of \(\mathbf{E}\) are all compactly supported with respect to \(\mathcal{C}\) at all natural numbers. By \cite[Lemma~6.2]{Matousek26B}, this property of \(\mathbf{E}\) forces \(E\in\mathcal{S}\). We again use \(\Hom_{\mathcal{S}}\big(\Sigma^{-1}T,V\big)\neq0\) for all \(V\in\vectX\) to deduce \(E\in\mathscr{H}^{(\infty)}\). Thus we may conclude \(E\in\mathcal{C}^{\perp}\cap\mathscr{H}^{(\infty)}\).

``Part ii)''  The category \(\mathcal{C}\) is now not countably generated.
Therefore, we can find an~uncountable subset \(I\subseteq\ProjectiveLine{K}\) and for each \(\lambda\in I\) a thick subcategory
\[\tube_{\lambda}\subseteq\mathscr{H}^{(\infty)}_{\lambda}
\text{\quad\quad\quad such that \quad\quad\quad}
\mathcal{C}\simeq\coprod_{\lambda\in I}\tube_\lambda.
\]
Let \(E\in\mathfrak{S}'_{\mathcal{C}}(\mathcal{S})\).
By an application of Lemma~\ref{L:LivingInCompletionGivenByThickSubcategory}, we get  \(E\in\mathfrak{S}'_{\mathcal{D}}(\mathcal{S})\) for some countably generated thick subcategory \(\mathcal{D}\subseteq\mathcal{C}\).
We have \(\mathcal{D}\subseteq\coprod_{\lambda\in J}\tube_\lambda\) for some \(J\subseteq I\) countable.  
Because \(I\) is uncountable, there exists \(\lambda\in I\setminus(\bflambda\cup J)\) satisfying
\(
\tube_{\lambda}=\mathscr{H}^{(\infty)}_{\lambda}
\)
and
\(
\mathscr{H}^{(\infty)}_{\lambda}\cap\,\mathcal{D}=0
.\)
Consequently, part i) applies, so \(E\in\mathcal{C}^{\perp}\cap\mathscr{H}^{(\infty)}\).
As $E$ was arbitrary, the~inclusion \(\mathfrak{S}'_{\mathcal{C}}(\mathcal{S})\subseteq\mathcal{C}^{\perp}\cap\mathscr{H}^{(\infty)}\) follows.
The other inclusion \(\mathcal{C}^{\perp}\cap\mathscr{H}^{(\infty)}\subseteq\mathfrak{S}'_{\mathcal{C}}(\mathcal{S})\) is automatic because every compactly supported element of \(\mathcal{S}\) is a homotopy colimit of a constant sequence.
\end{proof}

\begin{corollary}
\label{C:UncountablyGeneratedCompletionTubular}
Assume \(\X\) is tubular. Let \(q\in\Q\cup\{\infty\}\).
Let \(\mathcal{D}\subsetneq\mathcal{C}\) be thick subcategories of \(\mathscr{H}^{(q)}\subseteq\mathcal{S}\). Pick an indecomposable object \(T\in\mathscr{H}^{(q)}_{\lambda}\) for some \(\lambda\in\ProjectiveLine{K}\setminus\bflambda\).
Then:
\begin{enumerate}[label=(\roman*)]
\item If \(T\in\mathcal{C}\setminus\mathcal{D}\), then it holds that \(\mathfrak{S}'_{\mathcal{C}}(\mathcal{S})\cap\mathfrak{S}'_{\mathcal{D}}(\mathcal{S})\subseteq\mathcal{C}^{\perp}\cap\mathscr{H}^{(q)}\).
\item If \(\mathcal{C}\) is not countably generated, then \(\mathfrak{S}'_{\mathcal{C}}(\mathcal{S})=\mathcal{C}^{\perp}\cap\mathscr{H}^{(q)}\).
\end{enumerate}
\end{corollary}

\begin{proof}
By Theorem~\ref{T:TubularMutation}, there is a triangulated autoequivalence \(F:\mathcal{S}\rightarrow\mathcal{S}\) such that \(F\left(\mathcal{H}^{(q)}\right)=\cohfX\).
By Corollary~\ref{C:EkvilaleceEkvivalencakEkvivalencata}, the autoequivalence $F$ induces a triangulated equivalence \(\mathfrak{S}'_{\mathcal{M}}(\mathcal{S})\simeq\mathfrak{S}'_{F(\mathcal{M})}(\mathcal{S})\) for any good metric \(\mathcal{M}\) on \(\mathcal{S}\). Since \(F(\mathcal{C})\subseteq\mathscr{H}^{(\infty)}\), the~statement now directly follows from Proposition~\ref{P:UncountablyGeneratedCompletion}.
\end{proof}

For the purpose of describing completions of \(\mathcal{S}\), we introduce a technical notion of a~slope of metric. We stress that not every metric has to have a well-defined slope. 

\begin{definition}
Equip \(\mathcal{S}\) with an additive good metric \(\mathcal{M}=\{B_n\}_{n\in\N}\).
If there exists a~unique \(q\in\Q\cup\{\infty\}\) such that for every \(n\in\N\) it holds that \(\mathcal{H}^{(q)}\cap B_n\neq0\), then we say that \(\mathcal{M}\) has \textit{slope} \(q\) and denote \(\mu(\mathcal{M}):=q\). 
\end{definition}

\begin{lemma}
\label{L:WellDefinedSlopeUncountable}
Assume \(\chi_{\X}\geq0\).
Let \(\mathcal{C}\) be a thick subcategory of \(\mathcal{S}\).
If \(\mathcal{C}\) is not countably generated, then it has a well-defined slope as constant metric on \(\mathcal{S}\), and
\[
\mathfrak{S}'_{\mathcal{C}}(\mathcal{S})=\mathcal{C}^{\perp}\cap\mathscr{H}^{\MetricSlope{C}}
.\]
\end{lemma}

\begin{proof}
Because \(\mathcal{C}\) is not countably generated, it is not generated by an exceptional collection, so it is not bireflective by Proposition~\ref{P:BireflectiveExceptionalCollection}.
Then by Lemma~\ref{L:TwoSlopesBireflective}, there exists \(q\in\Q\cup\{\infty\}\) such that \(\mathcal{C}\subseteq\mathcal{H}^{(q)}\). If \(\X\) is domestic, then necessiarily \(q=\infty\) due to Proposition~\ref{P:VectBireflective}.
This forces \(q=\mu(\mathcal{C})\).
By Proposition~\ref{P:UncountablyGeneratedCompletion} and Corollary~\ref{C:UncountablyGeneratedCompletionTubular}, we get
\(
\mathfrak{S}'_{\mathcal{C}}(\mathcal{S})=\mathcal{C}^{\perp}\cap\mathscr{H}^{\MetricSlope{C}}
.\)
\end{proof}

\begin{lemma}
\label{L:WellDefinedSlopeAntithesisUniformConverence}
Assume \(\chi_{\X}\geq0\).
Let \(\mathcal{C}\) be a thick subcategory of \(\mathcal{S}\).
Equip \(\mathcal{S}\) with
an additive good metric \(\mathcal{M}=\{B_n\}_{n\in\N}\). Consider the~thick subcategory \(\mathcal{B}:=\bigcap_{n\in\N} B_n\) of~\(\mathcal{S}\). 
If \(\mathcal{M}\not\sim\mathcal{M}_{\infty}\vee\mathcal{B}\), then \(\mathcal{M}\) has a well-defined slope.
\end{lemma}

\begin{proof}
By Lemma~\ref{L:ObstructingObjectsPC} applied to the choice of \(E=0\in\mathfrak{S}'_{\mathcal{M}}(\mathcal{S})\), we get that (up to replacing \(\mathcal{M}\) by an~equivalent metric) there exists a family \((X_n)_{n\in\N}\) of objects of \(\cohX\) such that for all \(n\in\N\) it holds that \(X_n\in B_n\) but simultaneously \(X_n\notin\thick\big\langle H^0(B_{n+1})\rangle\).

This implies that \(\thick\big(H^0(B_{1})\big)\supsetneq\thick\big(H^0(B_{2})\big)\supsetneq\ldots\) is an~infinite, strictly decreasing chain of thick subcategories of \(\mathcal{S}\). By Lemma~\ref{L:TubularInfiniteChainDiffersAtHomogenousTubes} and Lemma~\ref{L:DomesticInfiniteChainDiffersAtHomogenousTubes}, there exist \(n\in\N\) and \(q\in\Q\cup\{\infty\}\) satisfying
\(
\big\langle H^0(B_{n})\big\rangle\subseteq\mathscr{H}^{(q)}
.\)
This $q$ is then necessarily unique and equal to the slope of \(\mathcal{M}\).
\end{proof}

We may now finally state the final result of the section.

\begin{theorem}
\label{T:WeightedCompletions}
Assume \(\chi_{\X}\geq0\).
Equip \(\mathcal{S}\) with
an additive good metric \(\mathcal{M}=\{B_n\}_{n\in\N}\). Let us denote the thick subcategory \(\mathcal{B}:=\bigcap_{n\in\N} B_n\subseteq\mathcal{S}\). 
Then:
\begin{enumerate}[label=(\roman*)]
\item If \(\mathcal{M}\sim\mathcal{M}_{\infty}\vee\mathcal{B}\), and if \(\mathcal{B}\) is countably generated, then
\[
\mathfrak{S}'_{\mathcal{M}}(\mathcal{S})\simeq\mathfrak{S}'_{\mathcal{B}}(\mathcal{S})\simeq\mathcal{S}/\mathcal{B}
.\]
 In this case, it holds that \(\mathfrak{S}'_{\mathcal{M}}(\mathcal{S})\simeq\mathcal{B}^{\perp}\cap\mathcal{S}\xhookrightarrow{}\mathcal{S}\) if and only if \(\mathcal{B}\) is generated by an~exceptional sequence.
\item Otherwise, we have
\[
\mathfrak{S}'_{\mathcal{M}}(\mathcal{S})=\mathcal{B}^{\perp}\cap\mathscr{H}^{\MetricSlope{M}}
.\]
\end{enumerate}
\end{theorem}

\begin{proof}
The proof will done by considering the following cases:
\begin{enumerate}[label=(\alph*)]
\item Condition i) holds, i.e.\ \(\mathcal{M}\sim\mathcal{M}_{\infty}\vee\mathcal{B}\) and \(\mathcal{B}\) is countably generated.
\item Condition ii) holds, \(\mathcal{M}\sim\mathcal{M}_{\infty}\vee\mathcal{B}\), and \(\mathcal{B}\) is not countably generated. 
\item Condition ii) holds, \(\mathcal{M}\gneq\mathcal{M}_{\infty}\vee\mathcal{B}\).
\end{enumerate}

\textbf{``Case a)''} By Corollary~\ref{C:CountablyGeneratedCompletionVerdierQuotient} and Lemma~\ref{L:DecompostionMetricsCompletion}, we get \(\mathfrak{S}'_{\mathcal{M}}(\mathcal{S})\simeq\mathfrak{S}'_{\mathcal{B}}(\mathcal{S})\simeq\mathcal{S}/\mathcal{B}\).
Corollary~\ref{C:CharterisingCoreflectivity} and Proposition~\ref{P:BireflectiveExceptionalCollection} tell us that
the equivalence \(\mathfrak{S}'_{\mathcal{B}}(\mathcal{S})\simeq\mathcal{B}^{\perp}\cap\,\mathcal{S}\) holds if an only if \(\mathcal{B}\) is coreflective if and only if \(\mathcal{B}\) is generated by an exceptional sequence.

\textbf{``Case b)''} 
Assume that \(\mathcal{M}\sim\mathcal{M}_{\infty}\vee\mathcal{B}\) and \(\mathcal{B}\) is not countably generated.
Then for all but finitely many \(n\in\N\) it holds that \(H^0(B_n)=H^0(\mathcal{B})\). Lemma~\ref{L:WellDefinedSlopeUncountable} implies \(\mu(\mathcal{M})=\mu(\mathcal{B})\) exists.
By Lemma~\ref{L:DecompostionMetricsCompletion} and Lemma~\ref{L:WellDefinedSlopeUncountable} again, we get
\(
\mathfrak{S}'_{\mathcal{M}}(\mathcal{S})\simeq\mathfrak{S}'_{\mathcal{B}}(\mathcal{S})\simeq\mathcal{B}^{\perp}\cap\mathscr{H}^{\MetricSlope{M}}
.\)

\textbf{``Case c)''}
Assume \(\mathcal{M}\gneq\mathcal{M}_{\infty}\vee\mathcal{B}\).
By Lemma~\ref{L:WellDefinedSlopeAntithesisUniformConverence}, the metric \(\mathcal{M}\) has a well-defined slope \(q\in\Q\cup\{\infty\}\).
We prove
\(
\mathfrak{S}'_{\mathcal{M}}(\mathcal{S})=\mathcal{B}^{\perp}\cap\mathscr{H}^{(q)}
.\)

``$\supseteq$''
It is enough to prove that any indecomposable coherent sheaf \(T\in\mathcal{B}^\perp\cap\,\mathcal{H}^{(q)}\) is an element of the~completion \(\mathfrak{S}'_{\mathcal{M}}(\mathcal{S})\) because any other object of the right-hand side category is a direct sum of shifts of indecomposable coherent sheaves from \(\mathcal{B}^\perp\cap\,\mathcal{H}^{(q)}\).

Denote \(\mathcal{C}_n:=\big\langle X\in\cohX:X\in B_{n},\Hom_{\mathcal{S}}(X,T\oplus\Sigma T)\neq0\big\rangle\) for each \(n\in\N\).
Consider the decreasing chain \(\mathcal{C}_1\supseteq\mathcal{C}_2\supseteq\ldots\) of thick subcategories of \(\mathcal{S}\).
If the chain did not stabilise, then by Lemma~\ref{L:TubularInfiniteChainDiffersAtHomogenousTubes} or Lemma~\ref{L:DomesticInfiniteChainDiffersAtHomogenousTubes}, there would exist \(n\in\N\) such that
\[
\left\langle\mathcal{H}^{(q')}_{\lambda},\mathcal{H}^{(q')}_{\lambda'}\right\rangle\subseteq\mathcal{C}_n\subseteq\mathscr{H}^{(q')}
\text{\quad\quad for some }
\lambda,\lambda'\in\ProjectiveLine{K}, \lambda\neq\lambda', q'\in\Q\cup\{\infty\}
,\]
forcing \(q=q'\) by the uniqueness of the slope of a metric. This would be a contradiction since the~orthogonality of tubes and the indecomposability of $T$ would imply \(T\in\mathcal{H}^{(q')}_{\lambda}\cap\mathcal{H}^{(q')}_{\lambda'}=0\).

Hence, there exists \(N\in\N\) such that for all \(m,n\geq N\) we have \(\mathcal{C}_m=\mathcal{C}_n\).
Assume for contradiction that there exists an indecomposable object \(Y\in B_{N+1}\) with \(\Hom_{\mathcal{S}}(Y,T)\neq0\). Then $Y$ is concentrated in cohomological degrees either $0$, or $1$. We get \(\Sigma Y,Y\in B_N\), so \(Y\in\mathcal{C}_{N}=\mathcal{C}_n\) for all \(n\geq N\). By Lemma~\ref{L:GrothendieckBruningWideTheoremForMetrics}, this means that \(Y\in\mathcal{B}\), contradicting \(T\in\mathcal{B}^{\perp}\).

We conclude that \(T\) is compactly supported at \(N+1\). While observing \(T\simeq\hoco T\), we obtain \(T\in\mathfrak{S}'_{\mathcal{M}}(\mathcal{S})\).

``$\subseteq$''
Let \(E\in\mathfrak{S}'_{\mathcal{M}}(\mathcal{S})\).
Since \(\mathcal{M}\gneq\mathcal{M}_{\infty}\vee\mathcal{B}\), Lemma~\ref{L:ObstructingObjectsSubcategories} provides us with an~infinite, strictly decreasing chain \(\mathcal{C}_1\supsetneq\mathcal{C}_2\supsetneq\cdots\) of thick subcategories of \(\mathcal{S}\) such that \(E\in\mathfrak{S}'_{\mathcal{C}_n}(\mathcal{S})\) and \(\mathcal{C}_n\subseteq\big\langle H^0(B_n)\big\rangle\) for each \(n\in\N\).
Using Lemma~\ref{L:TubularInfiniteChainDiffersAtHomogenousTubes} or Lemma~\ref{L:DomesticInfiniteChainDiffersAtHomogenousTubes} we find \mbox{\(n>m\geq N\)}, \(q'\in\Q\cup\{\infty\}\), and \(\lambda\in\ProjectiveLine{K}\setminus\bflambda\) such that
\(
\mathscr{H}^{(q')}_{\lambda}\subseteq\mathcal{C}_m\subseteq\mathscr{H}^{(q')}
\)
and
\mbox{\(\mathscr{H}^{(q')}_{\lambda}\cap\,\mathcal{C}_n=0.\)}
By the~definition of a slope of a metric, we obtain \(q=q'\).
Thus Proposition~\ref{P:UncountablyGeneratedCompletion} and Corollary~\ref{C:UncountablyGeneratedCompletionTubular} apply, so \(E\in\mathscr{H}^{(q)}\).
As $E$ is compactly supported with respect to \(\mathcal{M}\), inevitably \(E\in\mathcal{B}^{\perp}\). We may now conclude the proof. 
\end{proof}

We conclude the paper by illustrating Theorem~\ref{T:WeightedCompletions} with examples.

\begin{example}
\label{E:Homogenous}
Let \(\lambda=[\lambda^0,\lambda^1]\in\ProjectiveLine{K}\setminus\bflambda\). Consider the short exact sequence 
 \[
  0\rightarrow\mathcal{O}\xrightarrow{\lambda^1u-\lambda^0v}\mathcal{O}(\CanonicalElement)\rightarrow S_{\lambda}\rightarrow0,
  \]
Define \(\mathcal{C}:=\mathscr{H}_{\lambda}\). We know by Theorem~\ref{T:WeightedCompletions}, that \(\mathfrak{S}'_{\mathcal{C}}(\mathcal{S})\simeq\mathcal{S}/\mathcal{C}\). We shall explicitly express the~completion as a bounded derived category of a graded localisation ring.

Denote as \(y\) the~homogenous element \(\lambda^1u-\lambda^0v\in\gradedS\) of degree \(\CanonicalElement\in\gradingL\). Set \(Y:=\{y^n:n\in\N_0\}\) the homogenous multiplicative set generated by $y$.
The localisation \(\gradedS\rightarrow\gradedS_{Y^{-1}}\) of graded rings and the~corresponding graded localisation functor
\[
F:\modf^{\gradingL}\dashmodule\gradedS\rightarrow\modf^{\gradingL}\dashmodule\gradedS_{Y^{-1}}
\]
induce an isomorphism \(\modf^{\gradingL}\dashmodule\gradedS/\Ker F\rightarrow\modf^{\gradingL}\dashmodule\gradedS_{Y^{-1}}\)
where the quotient is taken the~sense of Serre (see e.g.\ \cite[Section~1.2]{ChenKrause10}).
Since \(y\in\gradedS\) is not invertible and \(\gradedS\) is graded local, we get
\[
\modf_0^{\gradingL}\dashmodule\gradedS\subseteq \Ker F
.\]
In fact, the image of \(\Ker F\) in \(\cohX\) is precisely \(\mathcal{H}_{\lambda}\).
By the universal property of Serre's quotients and the definition of \(\cohX\), the functor $F$ factors through \(\cohX\), 
thus giving us an~isomorphism
\[
\bar{F}:\cohX/\mathcal{H}_{\lambda}\rightarrow\modf^{\gradingL}\dashmodule\gradedS_{Y^{-1}}
.\]
As \(\bar{F}\) is exact, it extends to the derived level
\[
\mathfrak{S}'_{\mathcal{C}}(\mathcal{S})\simeq\derived^b(\cohX)/\mathcal{C}\simeq\derived^b\left(\modf^{\gradingL}\dashmodule\gradedS_{Y^{-1}}\right)
.\]
\end{example}

\begin{example}
Following Example~\ref{E:Homogenous}, we aim to complete with respect to all non-homogenous tubes. Set \(\mathcal{C}:=\langle\mathcal{H}_{\lambda}:\lambda\in\bflambda\rangle\).
Again
\[
\mathfrak{S}'_{\mathcal{C}}(\mathcal{S})\simeq\derived^b\left(\modf^{\gradingL}\dashmodule\gradedS_{Y^{-1}}\right)
\]
where the graded localisation \(\gradedS_{Y^{-1}}\) is taken with respect to the homogenous multiplicative set \(Y\) generated by \(x_1,\ldots,x_n\).
Now the \(\gradingL\)-graded ring is strongly graded \(\gradedS_{Y^{-1}}\), i.e.~for every \(\vL\in\gradingL\) there exists a homogenous invertible element of \(\gradedS_{Y^{-1}}\) of degree \(\vL\). By Dade's theorem \cite[Theorem~1.5.1]{Hazrat16}, the graded module category \(\modf^{\gradingL}\dashmodule\gradedS_{Y^{-1}}\) is equivalent to the ungraded module category \(\modf^{0}\dashmodule\gradedS^0_{Y^{-1}}\) concentrated at homogenous degree $0$ over the~$0$\=/degree subring \(\gradedS^0_{Y^{-1}}\) of \(\gradedS_{Y^{-1}}\).
Hence, we get \[
\mathfrak{S}'_{\mathcal{C}}(\mathcal{S})\simeq\derived^b\left(\modf^{0}\dashmodule\gradedS^0_{Y^{-1}}\right)
.\]
The ring \(\gradedS^0_{Y^{-1}}\) is isomorphic to some (ungraded) localisation of the (ungraded) polynomial ring \(K[x]\) in one variable.
\end{example}

In Example~\ref{E:Uncountable} and Example~\ref{E:Countable}, we discuss the effects of the cardinality of the ground field $K$ on the metric completions of \(\mathcal{S}\).

\begin{example}
\label{E:Uncountable}
Assume $K$ is uncountable.
Then any thick subcategory \(\mathcal{C}\subseteq\mathcal{S}\) can be expressed as \(\mathfrak{S}'_{\mathcal{M}}(\mathcal{S})\simeq\mathcal{C}\) for some good additive metric \(\mathcal{M}\) on \(\mathcal{S}\).
If \(\mathcal{C}\) is bireflective, it is a coaisle of a $t$-structure \(\left({}^{\perp}\mathcal{C},\mathcal{C}\right)\) on \(\mathcal{S}\). Hence \(\mathfrak{S}'_{{}^{\perp}\mathcal{C}}(\mathcal{S})\simeq\mathcal{C}\) by Corollary~\ref{C:CoreflectiveGivesPerp}.

Suppose that \(\mathcal{C}\) is not bireflective. Then \(\mathcal{C}\subseteq\mathscr{H}^{(q)}\) for some \(q\in\Q\cup\{\infty\}\) by Proposition~\ref{P:VectBireflective} and Lemma~\ref{L:TwoSlopesBireflective}. If \(\X\) is domestic, then necessarily \(q=\infty\).

Firstly, assume that \(\mathcal{C}\) is generated by finitely many elements. Then \(\mathcal{D}:=\mathscr{H}^{(q)}\cap{}^{\perp}\mathcal{C}\) is generated by uncountably many elements, since \(\ProjectiveLine{K}\) is uncountable and \(\cohfX\) has to contain uncountably many tubes. 
Theorem~\ref{T:WeightedCompletions} gives us \(\mathfrak{S}'_{\mathcal{D}}(\mathcal{S})\simeq\mathcal{D}^{\perp}\cap\mathscr{H}^{(q)}\simeq\mathcal{C}\), where the second equivalence is implied by \cite[Theorem~2.3.25]{Dichev09}.

Secondly, assume that \(\mathcal{C}\) is not generated by finitely many elements. We find a countable subset \(I=\{i_n\}_{n\in\N}\subseteq\ProjectiveLine{K}\setminus\bflambda\) satisfying \(\mathscr{H}^{(q)}_{\lambda_{i_n}}\subseteq\mathcal{C}\) for each \(n\in\N\).
Then \[
\mathcal{M}=\{B_n\}_{n\in\N}=\left\{\left({}^{\perp}\mathcal{C}\cap\mathscr{H}^{(q)}\right)\oplus\coprod_{m\geq n}\mathscr{H}^{(q)}_{\lambda_{i_m}}\right\}_{n\in\N}
\]
is an additive good metric on \(\mathcal{S}\) with \(\mathcal{B}:=\bigcap_{n\in\N}B_n={}^{\perp}\mathcal{C}\cap\mathscr{H}^{(q)}\).
As \(\mathcal{M}\not\sim\mathcal{M}_{\infty}\vee\mathcal{B}\), Theorem~\ref{T:WeightedCompletions} yields \(\mathfrak{S}'_{\mathcal{M}}(\mathcal{S})\simeq\mathcal{B}^{\perp}\cap\mathscr{H}^{(q)}\simeq\mathcal{C}\).
\end{example}

\begin{example}
\label{E:Countable}
The statement of Example~\ref{E:Uncountable} does not hold when $K$ is countable. 
Let \(\mathcal{C}\) be, for example, the thick subcategory generated by all non-homogenous tubes of torsion sheaves.

Considering that in our situation any thick subcategory of \(\mathcal{S}\) is countably generated,
if there existed a thick subcategory \(\mathcal{D}\) such that  \(\mathfrak{S}'_{\mathcal{D}}(\mathcal{S})\simeq\mathcal{C}\), then \(\mathcal{D}\) would be bireflective in \(\mathcal{S}\) by Theorem~\ref{T:WeightedCompletions}. But then \(\mathcal{C}\simeq\mathcal{D}^{\perp}\) would also be bireflective in \(\mathcal{S}\) by Proposition~\ref{P:BireflectiveExceptionalCollection}, which is nevertheless not true.

Our only other chance to exhibit \(\mathcal{C}\) as a completion is to find a good additive metric \(\mathcal{M}=\{B_n\}_{n\in\N}\) satisfying
\[
\mathcal{B}:=\bigcap_{n\in\N}B_n={}^{\perp}\mathcal{C}\cap\mathscr{H}^{(\infty)}=\coprod_{\lambda\in\ProjectiveLine{K}\setminus\bflambda}\mathscr{H}_{\lambda}
\]
and \(\mathcal{M}\not\sim\mathcal{M}_{\infty}\vee\mathcal{B}\).

Assume for contradiction that such a metric \(\mathcal{M}\) exists.
Because \(\mathcal{B}\) contains all homogenous tubes and \(\mathcal{B}\subseteq\big\langle H^0(B_n)\big\rangle\) for all \(n\in\N\), Lemma~\ref{L:TubularInfiniteChainDiffersAtHomogenousTubes} and Lemma~\ref{L:DomesticInfiniteChainDiffersAtHomogenousTubes} force the~chain \(\big\langle H^0(B_1)\big\rangle\supseteq\big\langle H^0(B_2)\big\rangle\supseteq\cdots\) of thick subcategories to stabilise at some \(N\in\N\).
By Lemma~\ref{L:GrothendieckBruningWideTheoremForMetrics}, we then get \(\big\langle H^0(B_N)\big\rangle=\mathcal{B}\). 
But this is impossible, since the condition \(\mathcal{M}\not\sim\mathcal{M}_{\infty}\vee\mathcal{B}\) ensures that \(H^0(B_N)\subseteq\big\langle H^0(B_N)\big\rangle\) contains an~indecomposable object from \(\mathcal{S}\setminus\mathcal{B}\). 
\end{example}

\bibliographystyle{plain}
\bibliography{literatura}

@Book{GobelTrlifaj12,
  author    = {Göbel, Rüdiger and Trlifaj, Jan},
  publisher = {Walter de Gruyter},
  title     = {Approximations and Endomorphism Algebras of Modules},
  year      = {2012},
  address   = {Berlin},
  isbn      = {978-3-11-021810-7},
  number    = {41},
  series    = {De Gruyter Expositions in Mathematics},
}

@Book{Lane98,
  author    = {Mac Lane, Saunders},
  publisher = {Springer},
  title     = {Categories for the {W}orking {M}athematician},
  year      = {1998},
  address   = {New York},
  edition   = {Second},
  series    = {Graduate Texts in Mathematics},
}

@Article{Neeman18,
  author        = {Neeman, Amnon},
  journal       = {arXiv},
  title         = {The categories {${\mathcal T}^c$} and {${\mathcal T}^b_c$} determine each other},
  year          = {2018},
  note          = {Preprint. First version. Available online at https://arxiv.org/abs/1806.06471v1.},
  archiveprefix = {arXiv},
  eprint        = {1806.06471},
  primaryclass  = {math.CT},
}

@Article{Neeman20,
  author  = {Neeman, Amnon},
  journal = {Journal of Pure and Applied Algebra},
  title   = {Metrics on triangulated categories},
  year    = {2020},
  issn    = {0022-4049},
  number  = {4},
  pages   = {106206},
  volume  = {224},
  doi     = {https://doi.org/10.1016/j.jpaa.2019.106206},
  url     = {https://www.sciencedirect.com/science/article/pii/S0022404919302130},
}

@Article{BiswasChenRahulParkerZheng24,
  author  = {Biswas, Rudradip and Chen, Hongxing and Rahul, Kabeer Manali and Parker, Chris J. and Zheng, Junhua},
  journal = {arXiV},
  title   = {Bounded t-structures, finitistic dimensions, and singularity categories of triangulated categories},
  year    = {2024},
  note    = {Preprint. Available online at https://arxiv.org/abs/2401.00130v2.},
}

@Article{Neeman96,
  author  = {Neeman, Amnon},
  journal = {Journal of the American Mathematical Society},
  title   = {The {G}rothendieck {D}uality {T}heorem {V}ia {B}ousfield's {T}echniques and {B}rown {R}epresentability},
  year    = {1996},
  number  = {1},
  pages   = {205-236},
  volume  = {9},
}

@Article{SunZhang21,
  author  = {Sun, Yongliang and Zhang, Yaohua},
  journal = {Theory and Applications of Categories},
  title   = {{L}adders and {C}ompletion of {T}riangulated {C}ategories},
  year    = {2021},
  number  = {4},
  pages   = {95-106},
  volume  = {37},
}

@Book{Neeman01,
  author      = {Neeman, Amnon},
  publisher   = {Princeton University Press},
  title       = {Triangulated {C}ategories},
  year        = {2001},
  address     = {Princeton},
  isbn        = {9781400837212},
  series      = {Annals of Mathematics Studies},
  volume      = {148},
  doi         = {doi:10.1515/9781400837212},
  lastchecked = {2024-05-06},
  url         = {https://doi.org/10.1515/9781400837212},
}

@Book{Happel88,
  author    = {Happel, Dieter},
  publisher = {Cambridge University Press},
  title     = {Triangulated categories in the representation of finite-dimensional algebras},
  year      = {1988},
  address   = {Cambridge},
  series    = {London Mathematical Society Lecture Note Series},
  volume    = {119},
}

@Article{BondalKapranov90,
  author  = {Bondal, Alexey I. and Kapranov, Mikhail},
  journal = {Mathematics of the USSR-Izvestiya},
  title   = {Representable {F}unctors, {S}erre {F}unctors, and {M}utations},
  year    = {1990},
  number  = {3},
  pages   = {519-541},
  volume  = {35},
}

@Book{Krause21,
  author     = {Krause, Henning},
  publisher  = {Cambridge University Press},
  title      = {Homological {T}heory of {R}epresentations},
  year       = {2021},
  series     = {Cambridge Studies in Advanced Mathematics},
  collection = {Cambridge Studies in Advanced Mathematics},
  place      = {Cambridge},
}

@Article{Neeman10,
  author  = {Neeman, Amnon},
  journal = {Annals of Mathematics},
  title   = {Some adjoints in homotopy categories},
  year    = {2010},
  note    = {Second Series},
  number  = {3},
  pages   = {2143-2155},
  volume  = {171},
}

@Article{LakingVitoria20,
  author  = {Laking, Rosanna and Vitória, Jorge},
  journal = {Pacific Journal of Mathematics},
  title   = {Definability and approximations in triangulated categories},
  year    = {2020},
  number  = {2},
  pages   = {557–586},
  volume  = {306},
  doi     = {https://doi.org/10.2140/pjm.2020.306.557},
}

@Article{IyamaYoshino08,
  author  = {Iyama, Osamu and Yoshino, Yuji},
  journal = {Inventiones mathematicae},
  title   = {Mutation in triangulated categories and rigid {C}ohen–{M}acaulay modules},
  year    = {2008},
  pages   = {117–168},
  volume  = {172},
}

@Article{Lawvere73,
  author  = {Lawvere, Francis William},
  journal = {Rendiconti del seminario matématico e fisico di Milano},
  title   = {Metric spaces, generalized logic, and closed categories},
  year    = {1973},
  pages   = {135-166},
  volume  = {XLIII},
}

@Article{Rickard89,
  author  = {Rickard, Jeremy},
  journal = {Journal of the London Mathematical Society},
  title   = {Morita theory for derived categories},
  year    = {1989},
  number  = {2},
  pages   = {436-456},
  volume  = {39},
}

@Article{CummingsGratz24,
  author  = {Cummings, Charley and Gratz, Sira},
  journal = {arXiv},
  title   = {Metric completions of discrete cluster categories},
  year    = {2024},
  note    = {Preprint. Available online at https://www.arxiv.org/abs/2407.17369v2.},
}

@Article{NicolasSaorin09,
  author  = {Nicolás, Pedro and Saorín, Manuel},
  journal = {Journal of Algebra},
  title   = {Parametrizing recollement data for triangulated categories},
  year    = {2009},
  pages   = {1220–1250},
  volume  = {322},
}

@Article{Neeman25A,
  author        = {Neeman, Amnon},
  journal       = {arXiv},
  title         = {The categories {${\mathcal T}^c$} and {${\mathcal T}^b_c$} determine each other},
  year          = {2025},
  note          = {Preprint. Available online at https://arxiv.org/abs/1806.06471v2.},
  archiveprefix = {arXiv},
  eprint        = {1806.06471},
  primaryclass  = {math.CT},
}

@Article{Neeman25B,
  author  = {Neeman, Amnon},
  journal = {arXiv},
  title   = {Excellent metrics on triangulated categories, and the involutivity of the map taking {$\mathcal S$} to {$\mathfrak{S}({\mathcal{S})^{\mathrm{op}}}$}},
  year    = {2025},
  note    = {Preprint. Available online at https://arxiv.org/abs/2505.09120v1.},
}

@Article{BeilinsonBernsteinDeligne82,
  author  = {Beilinson, Alexander and Bernstein, Joseph and Deligne, Pierre},
  journal = {Astérisque},
  title   = {Faisceaux pervers},
  year    = {1982},
  number  = {100},
  pages   = {5--171},
  volume  = {1},
}

@Book{BeilinsonBernsteinDeligneGabber18,
  author    = {Beilinson, Alexander and Bernstein, Joseph and Deligne, Pierre and Gaber, Ofer},
  publisher = {Société Mathématique de France},
  title     = {Faisceaux pervers},
  year      = {2018},
  isbn      = {978-2-85629-878-7},
}

@PhdThesis{Dichev09,
  author = {Dichev, Nikolay Dimitrov},
  school = {University of Paderborn},
  title  = {Thick subcategories for quiver representations},
  year   = {2009},
}

@Article{IgusaTodorov15,
  author  = {Igusa, Kiyoshi and Todorov, Gordana},
  journal = {Communications in Algebra},
  title   = {Cluster {C}ategories {C}oming from {C}yclic {P}osets},
  year    = {2015},
  number  = {10},
  pages   = {4367--4402},
  volume  = {43},
}

@Article{PaquetteYildirim21,
  author  = {Paquette, Charles and Yıldırım, Emine},
  journal = {Transactions of the London Mathematical Society},
  title   = {Completions of discrete cluster categories of type {$\mathbb{A}$}},
  year    = {2021},
  number  = {1},
  pages   = {35--64},
  volume  = {8},
}

@Article{Rahul25A,
  author  = {Rahul, Kabeer Manali},
  journal = {arXiv},
  title   = {Admissible subcategories and metric techniques},
  year    = {2025},
  note    = {Preprint. Available online at https://www.arxiv.org/abs/2504.11772.},
}

@Article{Rahul25B,
  author  = {Rahul, Kabeer Manali},
  journal = {arXiv},
  title   = {Representability theorems via metric techniques},
  year    = {2025},
  note    = {Preprint. Available online at https://www.arxiv.org/abs/2504.11768.},
}

@Article{Keller05,
  author  = {Keller, Bernhard},
  journal = {Documenta Mathematica},
  title   = {On triangulated orbit categories},
  year    = {2005},
  pages   = {551--581},
  volume  = {10},
}

@Article{Balmer11,
  author  = {Balmer, Paul},
  journal = {Advances in Mathematics},
  title   = {Separability and triangulated categories},
  year    = {2011},
  number  = {5},
  pages   = {4352--4372},
  volume  = {226},
}

@Article{Beligiannis00,
  author  = {Beligiannis, Apostolos},
  journal = {Journal of Algebra},
  title   = {Relative homological algebra and purity in triangulated categories},
  year    = {2000},
  number  = {1},
  pages   = {268--361},
  volume  = {227},
}

@Article{BalmerStevenson21,
  author  = {Balmer, Paul and Stevenson, Greg},
  journal = {Journal of the London Mathematical Society},
  title   = {Relative stable categories and birationality},
  year    = {2021},
  number  = {4},
  pages   = {1765--1794},
  volume  = {104},
}

@InProceedings{Neeman23,
  author    = {Neeman, Amnon},
  booktitle = {International Congress of Mathematicians 2022},
  title     = {Finite approximations as a tool for studying triangulated categories},
  year      = {2023},
  pages     = {1636--1658},
  publisher = {EMS Press},
  volume    = {3},
}

@Article{GratzZvonareva23,
  author  = {Gratz, Sira and Zvonareva, Alexandra},
  journal = {Journal of the London Mathematical Society},
  title   = {Lattices of t-structures and thick subcategories for discrete cluster categories},
  year    = {2023},
  number  = {3},
  pages   = {973--1001},
  volume  = {107},
}

@Article{Kussin16,
  author  = {Kussin, Dirk},
  journal = {Journal of Noncommutative Geometry},
  title   = {Weighted noncommutative regular projective curves},
  year    = {2016},
  number  = {4},
  pages   = {1465--1540},
  volume  = {10},
}

@Article{Krause05A,
  author  = {Krause, Henning},
  journal = {Compositio Mathematica},
  title   = {The stable derived category of a noetherian scheme},
  year    = {2005},
  pages   = {1128--1162},
  volume  = {141},
}

@PhdThesis{Stovicek09,
  author = {Šťovíček, Jan},
  school = {NTNU Trondheim},
  title  = {Compactly generated triangulated categories and the telescope conjecture},
  year   = {2009},
}

@InCollection{Lenzing07,
  author    = {Lenzing, Helmut},
  booktitle = {Handbook of Tilting Theory},
  publisher = {Cambridge University Press},
  title     = {Hereditary categories},
  year      = {2007},
  pages     = {105--146},
  series    = {London Mathematical Society Lecture Note Series},
  volume    = {332},
}

@Article{Elagin25,
  author  = {Elagin, Alexey Dmitrievich},
  journal = {arXiv},
  title   = {Thick subcategories on weighted projective curves and nilpotent representations of quivers},
  year    = {2025},
  note    = {Preprint. Available online at https://arxiv.org/abs/2407.01207v2.},
}

@InProceedings{GeigleLenzing87,
  author    = {Geigle, Werner and Lenzing, Helmut},
  booktitle = {Singularities, Representation of Algebras, and Vector Bundles},
  title     = {A class of weighted projective curves arising in representation theory of finite dimensional algebras},
  year      = {1987},
  address   = {Berlin, Heidelberg},
  pages     = {265--297},
  publisher = {Springer-Verlag},
  series    = {Lecture Notes in Mathematics},
  volume    = {1273},
  isbn      = {978-3-540-47851-5},
}

@Article{Cheng24,
  author  = {Cheng, Yiyu},
  journal = {Journal of Pure and Applied Algebra},
  title   = {Wide subcategories of a domestic weighted projective line},
  year    = {2024},
  number  = {9},
  pages   = {107669},
  volume  = {228},
}

@InProceedings{Lenzing11,
  author    = {Lenzing, Helmut},
  booktitle = {Representations of Algebras and Related Topics},
  title     = {Weighted projective lines and applications},
  year      = {2011},
  address   = {Zürich},
  pages     = {153--187},
  series    = {EMS Series of Congress Reports},
}

@Book{Ringel84,
  author    = {Ringel, Claus Michael},
  publisher = {Springer},
  title     = {Tame {A}lgebras and {I}ntegral {Q}uadratic {F}orms},
  year      = {1984},
  address   = {Berlin, Heidelberg, New York, Tokyo},
  number    = {1099},
  series    = {Lecture Notes in Mathematics},
}

@Article{ChenKrause10,
  author  = {Chen, Xiao-Wu and Krause, Henning},
  journal = {arXiv},
  title   = {Introduction to coherent sheaves on weighted projective lines},
  year    = {2010},
  note    = {Preprint. Available online at https://www.arxiv.org/abs/0911.4473v3.},
}

@Book{Hazrat16,
  author     = {Hazrat, Roozbeh},
  publisher  = {Cambridge University Press},
  title      = {Graded {R}ings and {G}raded {G}rothendieck {G}roups},
  year       = {2016},
  number     = {435},
  series     = {London Mathematical Society Lecture Note Series},
  collection = {London Mathematical Society Lecture Note Series},
  place      = {Cambridge},
}

@Article{Bondal90,
  author  = {Bondal, Alexey I.},
  journal = {Mathematics of the USSR-Izvestiya},
  title   = {Representation of associative algebras and coherent sheaves},
  year    = {1990},
  number  = {1},
  pages   = {23--42},
  volume  = {34},
}

@Article{Shah23,
  author  = {Shah, Amit},
  journal = {Expositiones Mathematicae},
  title   = {Krull-{R}emak-{S}chmidt decompositions in {H}om-finite additive categories},
  year    = {2023},
  number  = {1},
  pages   = {220--237},
  volume  = {41},
}

@Article{CanonacoNeemanStellari25,
  author  = {Canonaco, Alberto and Neeman, Amnon and Stellari, Paolo},
  journal = {Bollettino dell’Unione Matematica Italiana},
  title   = {Weakly approximable triangulated categories and enhancements: a survey},
  year    = {2025},
  pages   = {109--134},
  volume  = {18},
}

@Article{Neeman21C,
  author  = {Neeman, Amnon},
  journal = {Annals of Mathematics},
  title   = {Strong generators in {${\mathbf{D}}^{\text{perf}}(X)$} and {${\mathbf{D}}^{b}_{\text{coh}}(X)$}},
  year    = {2021},
  number  = {3},
  pages   = {689--732},
  volume  = {193},
}

@Article{Neeman24,
  author  = {Neeman, Amnon},
  journal = {Acta Mathematica},
  title   = {Bounded {$t$}-structures on the category of perfect complexes},
  year    = {2024},
  pages   = {239--284},
  volume  = {233},
}

@Article{CanonacoHaesemeyerNeemanStellari24,
  author  = {Canonaco, Alberto and Neeman, Amnon and Stellari, Paolo},
  journal = {arXiv},
  title   = {The passage among the subcategories of weakly approximable triangulated categories},
  year    = {2024},
  note    = {With an appendix by Christian Haesemeyer. Preprint. Available online at https://arxiv.org/abs/2402.04605v1.},
}

@Article{DeynLankRahul25,
  author  = {De Deyn, Timothy and Lank, Pat and Manali Rahul, Kabeer},
  journal = {arXiv},
  title   = {Approximability and {R}ouquier dimension for noncommutative algebras over schemes},
  year    = {2025},
  note    = {Preprint. Available online at https://arxiv.org/abs/2408.04561v3.},
}

@Article{Matousek26B,
  author  = {Matoušek, Cyril},
  journal = {arXiv},
  title   = {Metric completions of triangulated categories from hereditary rings},
  year    = {2026},
  note    = {Preprint. Available online at https://www.arxiv.org/abs/2508.20283v2.},
}

@Article{HugelKussin17,
  author  = {Angeleri H{\"u}gel, Lidia and Kussin, Dirk},
  journal = {Documenta Mathematica},
  title   = {Large {T}ilting {S}heaves over {W}eighted {N}oncommutative {R}egular {P}rojective {C}urves},
  year    = {2017},
  pages   = {67--134},
  volume  = {22},
}

@Article{ArtinZhang94,
  author  = {Artin, Michael and Zhang, James Jian},
  journal = {Advances in Mathematics},
  title   = {Noncommutative {P}rojective {S}chemes},
  year    = {1994},
  number  = {2},
  pages   = {228--287},
  volume  = {109},
}

@Article{Jasso15,
  author  = {Jasso, Gustavo},
  journal = {Advances in Mathematics},
  title   = {{${\tau}^2$}-{S}table tilting complexes over weighed projective lines},
  year    = {2015},
  pages   = {1--31},
  volume  = {273},
}

@Article{Krah23,
  author  = {Krah, Johannes},
  journal = {Inventiones mathematicae},
  title   = {A phantom on a rational surface},
  year    = {2023},
  pages   = {1009--1018},
  volume  = {235},
}

@Article{Matousek26C,
  author  = {Matoušek, Cyril},
  journal = {Applied Categorical Structures},
  title   = {Metric {C}ompletions of {T}riangulated {C}ategories from {F}inite {D}imensional {A}lgebras},
  year    = {2026},
  number  = {27},
  pages   = {1--32},
  volume  = {34},
}

@PhdThesis{Matousek26D,
  author = {Matoušek, Cyril},
  school = {Aarhus University},
  title  = {Completing derived representation theory of algebras: Metrics on triangulated categories and applications},
  year   = {2026},
  note   = {Available online at \url{https://data.math.au.dk/publications/phd/2026/math-phd-2026-cm.pdf}.},
}

@Book{Riehl16,
  author    = {Riehl, Emily},
  publisher = {Dover Publications},
  title     = {Category {T}heory in {C}ontext},
  year      = {2016},
  isbn      = {978-0-486-80903-8},
  url       = {https://books.google.de/books?id=6B9MDgAAQBAJ},
}

@Article{RingelZhang14,
  author  = {Ringel, Claus Michael and Zhang, Pu},
  journal = {Mathematische Zeitschrift},
  title   = {From submodule categories to preprojective algebras},
  year    = {2014},
  pages   = {55--73},
  volume  = {278},
}

@Article{RingelZhang15,
  author  = {Ringel, Claus Michael and Zhang, Pu},
  journal = {Science China Mathematics},
  title   = {Objective triangle functors},
  year    = {2015},
  number  = {2},
  pages   = {221--232},
  volume  = {58},
}

@InCollection{Krause10B,
  author    = {Krause, Henning},
  booktitle = {Triangulated Categories},
  publisher = {Cambridge University Press},
  title     = {Localization theory for triangulated categories},
  year      = {2010},
  pages     = {161--235},
  series    = {London Mathematical Society Lecture Note Series},
  volume    = {375},
}

@Article{Franchini25,
  author  = {Franchini, Sofia},
  journal = {Mathematische Zeitschrift},
  title   = {Torsion pairs, t-structures, and co-t-structures forcompletions of discrete cluster categories},
  year    = {2025},
  number  = {86},
  pages   = {1--56},
  volume  = {310},
}

@Article{Joergensen08,
  author  = {Jørgensen, Peter},
  journal = {Journal of K-Theory},
  title   = {Auslander-{R}eiten triangles in subcategories},
  year    = {2008},
  number  = {3},
  pages   = {583--601},
  volume  = {3},
}

@Article{Neeman26,
  author  = {Neeman, Amnon},
  journal = {Inventiones mathematicae},
  title   = {Triangulated categories with a single compact generator and two {B}rown representability theorems},
  year    = {2026},
  pages   = {531--616},
  volume  = {244},
}

@Article{HolmJoergensen12,
  author  = {Thorsten, Holm and Jørgensen, Peter},
  journal = {Mathematische Zeitschrift},
  title   = {On a cluster category of infinite {D}ynkin type, and the relation to triangulations of the infinity-gon},
  year    = {2012},
  pages   = {277--295},
  volume  = {270},
}

@Article{GratzHolmJoergensen19,
  author  = {Gratz, Sira and Holm, Thorsten and Jørgensen, Peter},
  journal = {Mathematische Zeitschrift},
  title   = {Cluster tilting subcategories and torsion pairs in {I}gusa-{T}odorov cluster categories of {D}ynkin type {${A}_{\infty}$}},
  year    = {2019},
  pages   = {33--56},
  volume  = {292},
}

@PhdThesis{Fisher17,
  author = {Fisher, Thomas Andrew},
  school = {Newcastle University},
  title  = {On the homological algebra of clusters, quivers, and triangulations},
  year   = {2017},
}

@Article{AugustCheungFaberGratzSchroll23,
  author  = {August, Jenny and Cheung, Man Wai and Faber, Eleonore and Gratz, Sira and Schroll, Sibylle},
  journal = {Journal of the London Mathematical Society},
  title   = {Cluster structures for the {${A}_{\infty}$} singularity},
  year    = {2023},
  number  = {6},
  pages   = {2121--2149},
  volume  = {107},
}

@Article{Murphy25A,
  author  = {Murphy, Dave},
  journal = {Journal of Algebra},
  title   = {The {G}rothendieck groups of discrete cluster categories of {D}ynkin type {${A}_{\infty}$}},
  year    = {2025},
  pages   = {545--567},
  volume  = {662},
}

@Article{Murphy25B,
  author  = {Murphy, Dave},
  journal = {Journal of Pure and Applied Algebra},
  title   = {Bounding the {O}rlov spectrum for a completion of discrete cluster categories},
  year    = {2025},
  number  = {10},
  pages   = {108078},
  volume  = {229},
}

@Article{CanakciKalckPressland25,
  author  = {Çanakçı, İlke and Kalck, Martin and Pressland, Matthew},
  journal = {Journal of the London Mathematical Society},
  title   = {Cluster categories for completed infinity-gons {I}: {C}ategorifying triangulations},
  year    = {2025},
  number  = {2},
  pages   = {e70092},
  volume  = {111},
}

\end{document}